\documentclass[11pt]{amsart}
 \usepackage{setspace}

\usepackage{color}
\usepackage{comment}
\usepackage[margin=1in]{geometry}
\usepackage{graphicx}
\usepackage{float}
\usepackage{mathrsfs}
\usepackage{yfonts}
\usepackage{amssymb}

\newtheorem{rhp}{Riemann-Hilbert Problem}
\newtheorem{theorem}{Theorem}
\newtheorem{lemma}{Lemma}
\newtheorem{proposition}{Proposition}
\theoremstyle{remark}

\def\eq{\begin{equation}}
\def\endeq{\end{equation}}
\def\bbm{\begin{bmatrix}}
\def\ebm{\end{bmatrix}}
\def\bpm{\begin{pmatrix}}
\def\epm{\end{pmatrix}}

\numberwithin{equation}{section}

\title[NLS $n$-solitons]{Far-field asymptotics for multiple-pole solitons in the large-order limit}

\author{Deniz Bilman}
\address{Deniz Bilman: Department of Mathematical Sciences, University of Cincinnati, Cincinnati, OH, USA}
\email{bilman@uc.edu}

\author{Robert Buckingham}
\address{Robert Buckingham: Department of Mathematical Sciences, University of Cincinnati, Cincinnati, OH, USA}
\email{buckinrt@uc.edu}

\author{Deng-Shan Wang}
\address{Deng-Shan Wang: Laboratory of Mathematics and Complex Systems (Ministry of Education), School of Mathematical Sciences, Beijing Normal University, Beijing, People's Republic of China}
\email{dswang@bnu.edu.cn}

\thanks{
D. Bilman was partially supported by a research fellowship from Charles Phelps Taft Research Center.
R. Buckingham was supported by the National Science Foundation through grant
DMS-1615718.
D. S. Wang was supported by the National Natural Science Foundation of China through grant 11971067 and the Fundamental Research Funds for the Central Universities through grant 2020NTST22.
}

\begin{document}

\begin{abstract}

The integrable focusing nonlinear Schr\"odinger equation admits soliton solutions 
whose associated spectral data consist of a single pair of conjugate poles of 
arbitrary order.  We study families of such multiple-pole solitons generated by 
Darboux transformations as the pole order tends to infinity.  We show that 
in an appropriate scaling, there are four regions in the space-time plane where 
solutions display qualitatively distinct behaviors:  an exponential-decay region, 
an algebraic-decay region, a non-oscillatory region, and an oscillatory region.  
Using the nonlinear steepest-descent method for analyzing Riemann-Hilbert problems, 
we compute the leading-order asymptotic behavior in the algebraic-decay, 
non-oscillatory, and oscillatory regions.

\end{abstract}

\maketitle


\section{Introduction}

The one-dimensional focusing cubic nonlinear Schr\"odinger (NLS) equation
\eq
\label{nls}
i\psi_t + \frac{1}{2}\psi_{xx} + |\psi|^2\psi = 0, \quad x,t\in\mathbb{R},
\endeq
is well known to be a completely integrable equation admitting solitons, 
i.e.\ localized traveling-wave solutions.  Each initial datum from an 
appropriate function space (Schwartz space is sufficient for our needs) is 
associated with a set of scattering data, consisting of poles and norming 
constants encoding solitons, as well as a reflection coefficient encoding 
radiation.  The scattering data for a standard soliton consist of a 
complex-conjugate pair of first-order poles (and an associated norming 
constant) and an identically zero reflection coefficient.  However, for any 
$n\in\mathbb{Z}_+$, the NLS equation also has solutions whose scattering data 
consist of a complex-conjugate pair of poles order $n$ (plus $n$ auxiliary 
parameters that are higher-order analogues of norming constants) and no reflection.
These \emph{mulitple-pole solitons} ($n\geq 2$) have very 
different qualitative behavior than standard solitons.  At sufficiently large 
time scales, the $n$th-order pole soliton resembles $n$ solitons approaching 
each other, interacting, and then separating again.  This complicated 
interaction displays a remarkable degree of structure at different scales 
as $n$ increases.  These distinguished scales include:

\vspace{.1in}

\noindent
{\bf The near-field limit.}  
The scaling $X:=nx$, $T:=n^2t$ is appropriate for studying the 
rogue-wave-type behavior near the origin.  Here the key feature is a single 
peak with amplitude of order $n$.  Locally the solution satisfies for each fixed $T$ 
a certain differential equation in the Painlev\'e-III hierarchy.  This regime was 
analyzed by two of the authors in \cite{BilmanB:2019}, the first large-$n$ analysis 
of $n$th-order pole solitons. The asymptotic solution seems to be a type of  
universal behavior, also appearing in the study of high-order Peregrine 
breathers for the NLS equation with constant, non-zero boundary conditions 
\cite{BilmanLM:2018}.  

\vspace{.1in}

\noindent
{\bf The far-field limit.}  Define 
\eq
\label{scale}
\chi:=\frac{x}{n}, \quad  \tau:=\frac{t}{n}.
\endeq
As the pole order $n\to\infty$, then the ($\chi$,$\tau$)-plane can be partitioned 
into $n$-independent regions in which the {multiple-pole} soliton has distinct 
behaviors, such as rapid oscillations of frequency $n$ or decay to zero.  
This scaling was previously studied in \cite{BilmanB:2019} and is the focus 
of the current work.

\vspace{.1in}

\noindent
{\bf The long-time limit.}  If $x$ and $t$ are unscaled, then as $t\to {\pm}\infty$ 
the $n$th-order pole soliton asymptotically resembles a train of $n$ 
distinct one-solitons.  
{
Asymptotics as $t\to \pm \infty$ were obtained by Olmedilla in \cite{Olmedilla:1987} for $n$th-order pole solitons for fixed order $n=2$ and $n=3$ by solving Gel'fand-Levitan-Marchenko equations with an appropriate kernel and arriving at a representation for the $n$th-order pole soliton that involves determinants of size $n$ via Cramer's rule.
Large-$t$ asymptotics for multiple-pole solutions of arbitrary but finite and fixed order $n$ were obtained by Schiebold in \cite{Schiebold:2017} using the earlier algebraic results \cite{Schiebold:2010b} by the same author.
}


\begin{figure}[h]
\begin{center}
\hspace{0.2in} $n=2$ \hspace{1.6in} $n=4$ \hspace{1.6in} $n=8$ \\
\vspace{.05in}
\includegraphics[height=2in]{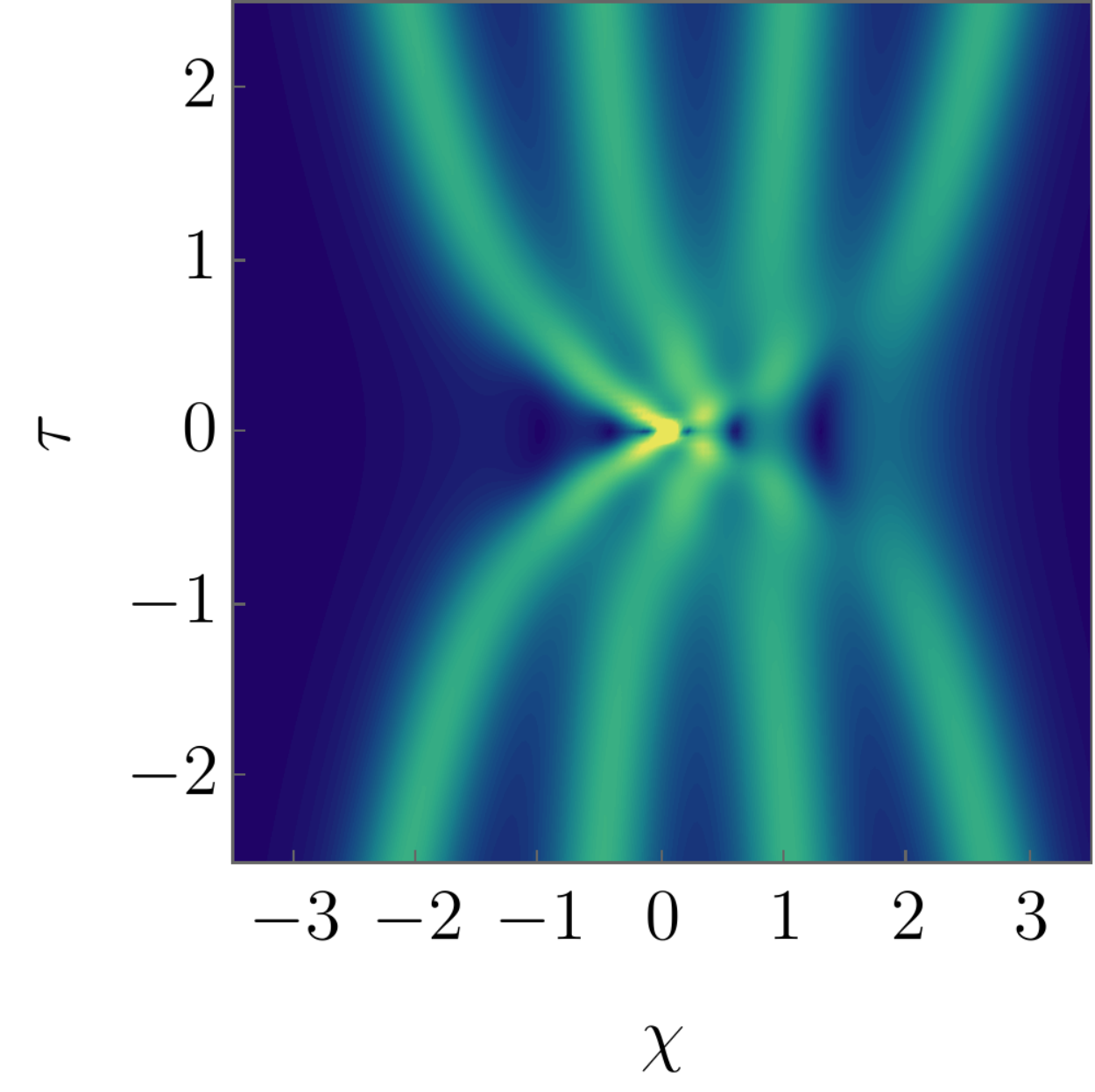}
\includegraphics[height=2in]{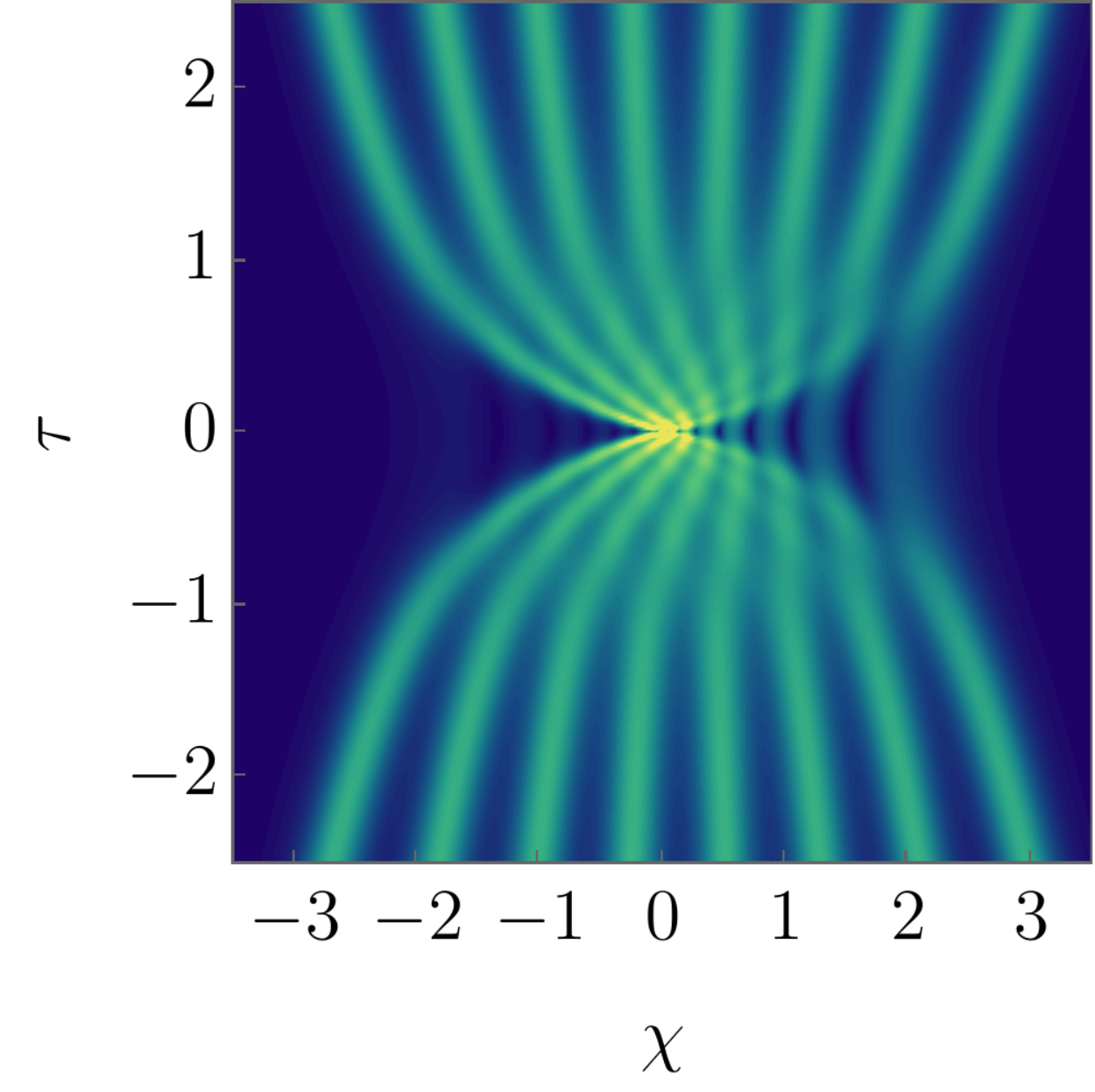}
\includegraphics[height=2in]{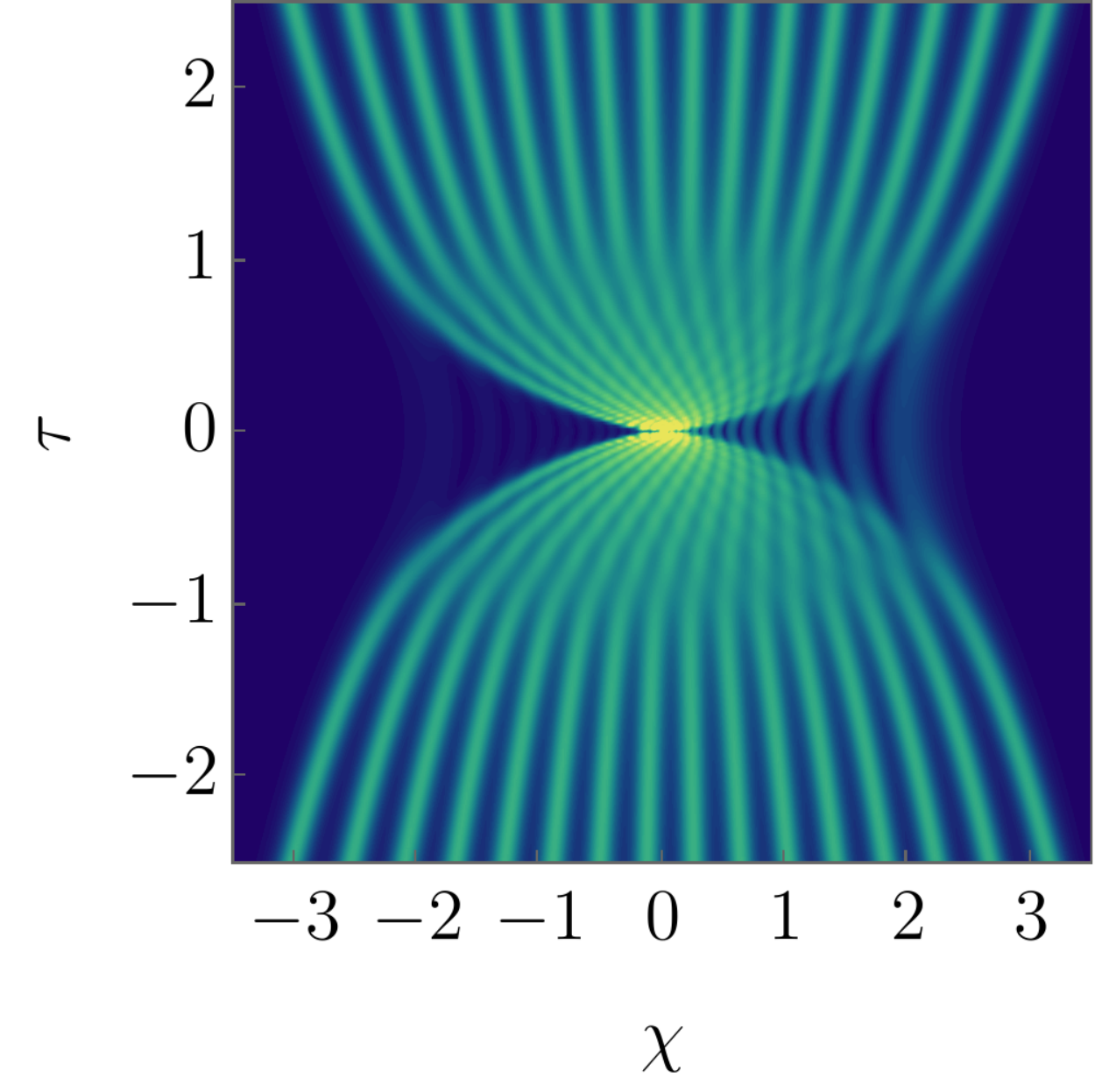}
\caption{The far-field scaling.  Plots of $|\psi^{[2n]}(n\chi,n\tau;(1,3),i)|$ for $-3.5\leq \chi\leq 3.5$ and $-2.5\leq \tau\leq 2.5$, where $\psi^{[2n]}(n\chi,n\tau;(1,3),i)$ is a multiple-pole soliton solution of the nonlinear Schr\"odinger equation \eqref{nls}.  In each plot $c_1=1$, $c_2=3$, and $\xi=i$.  \emph{Left:} $n=2$, \emph{Center:} $n=4$. \emph{Right:} $n=8$.}
\label{fig-far-field-2d}
\end{center}
\end{figure}

The generic $n$th-order pole soliton depends on a complex parameter $\xi$ 
(the spectral pole in the upper half-plane) and $n$ constant nonzero row vectors 
$(d_{1,j},d_{2,j})\in\mathbb{C}^2$, $j=1,...,n$ (higher-order analogues of the 
norming constants).  This function can be constructed via $n$ iterated Darboux 
transformations as described in \cite[\S 2]{BilmanB:2019}.  Working directly with a 
Riemann-Hilbert problem characterization in the context of the robust 
inverse-scattering transform framework provides fundamental eigenfunction matrices 
that are analytic at $\xi$ after each iteration by encoding the effect of the 
Darboux transformation in the form of a jump condition instead of a singularity in 
the spectral plane.  In order to obtain well-defined limits as $n\to\infty$, we 
first fix nonzero complex numbers $c_1$ and $c_2$  and set 
${\bf c}:=(c_1,c_2)\in(\mathbb{C}^*)^2$ (here 
$\mathbb{C}^*:=\mathbb{C}\setminus\{0\}$). We then take 
$(d_{1,j},d_{2,j}):=(\epsilon^{-1}c_1,\epsilon^{-1}c_2)$ for $j=1,...,n$ and take 
the limit $\epsilon\to 0^+$.  See Figure \ref{fig-far-field-2d} for plots of 
representative multiple-pole solitons in the far-field scaling.  This construction 
procedure is given in Appendix~\ref{app-Darboux} for completeness of our work, and 
it yields a representation of these multiple-pole solitons 
$\psi^{[2n]}(x,t;{\bf c},\xi)$ given in Riemann-Hilbert Problem \ref{rhp-M} below, which 
is convenient for our purposes of asymptotic analysis.

{
A related avenue of research pioneered by the work of Gesztesy, Karwowski, and Zhao in \cite{GesztesyKZ:1992} is the so-called \emph{countable superposition of solitons}. The authors considered a sequence of distinct eigenvalues $\{-\kappa_j^2\}_{j=1}^\infty$ along with associated norming constants $\{c_j\}_{j=1}^\infty$ and zero reflection coefficient for the Schr\"odinger operator. For each finite $N\in\mathbb{N}$, the scattering data $\{\kappa_j, c_j\}_{j=1}^N$ defines a reflectionless $N$-soliton solution $V_N(x,t)$ of the Korteweg-de Vries equation. Under certain summability and growth conditions on $\{\kappa_j, c_j\}$ as $N\to +\infty$, the authors established a limiting solution $V_\infty(x,t)$ of the Korteweg-de Vries equation that is reflectionless, global, and smooth. The study of countable superposition of solitons was extended to the focusing NLS equation \eqref{nls} later by Schiebold in \cite{Schiebold:2008} and \cite{Schiebold:2010a} for a sequence of distinct eigenvalues $\{ \lambda_j\}_{j=1}^{\infty}$ of the Zakharov-Shabat problem in the upper half-plane along with the associated norming constants again subject to appropriate growth conditions.
Drawing a comparison, the solutions we study can be thought of as a countable superposition as $n\to+\infty$ over $\mathbb{N}$, albeit with $\lambda_j\equiv \xi$ for \emph{all} $j\in \mathbb{N}$. Due to the repeated choice of the exceptional points $\lambda_j$, however, the family of solutions we study fall outside of the classes studied in these works. Indeed, following the proof of \cite[Lemma 1]{BilmanB:2019}, it is easy to see that $\psi^{[2n]}(0,0; \mathbf{c}, \xi) = 8\Im(\xi) c_1 c_2^* |\mathbf{c}|^{-2} n$, and hence the amplitudes of the solutions $\psi^{[2n]}(x,t; \mathbf{c}, \xi)$ explode as $n\to+\infty$. Therefore, there is not a limiting profile in the unscaled $(x,t)$-plane as $n\to+\infty$, contrary to the case in \cite{GesztesyKZ:1992,Schiebold:2008,Schiebold:2010a}. On the other hand, for each $n\in\mathbb{N}$, $\psi^{[2n]}(x,t; \mathbf{c}, \xi)$ defines a global classical solution (in fact, real-analytic in $(x,t)$) of the focusing NLS equation \eqref{nls}. This is a consequence of analytic Fredholm theory applied to the Riemann-Hilbert Problem~\ref{rhp-M}, which has analytic dependence on $(x,t)$ with a compact jump contour (see \cite[Proposition 3]{BilmanLM:2018} for details). Regularity properties of these solutions for fixed $n\in\mathbb{N}$ were also recently established using determinant representations \cite{ZhangTTH:2020}.
 }

In the present work we show that in the far-field scaling 
$\psi^{[2n]}(n\chi,n\tau;{\bf c},\xi)$ has 
four qualitatively different behaviors depending on the values of $\chi$ and 
$\tau$, and we give the leading-order large-$n$ asymptotic behavior for all 
$\chi$ and $\tau$ off the boundary curves.  As $n\to\infty$, 
$\psi^{[2n]}(n\chi,n\tau;{\bf c},\xi)$ exhibits the following four behaviors:

\vspace{.1in}

\noindent
{\bf The exponential-decay region.}  In this region the solution decays 
exponentially fast to zero as $n\to\infty$.  This was proven in 
\cite{BilmanB:2019}.  In the Riemann-Hilbert analysis the model problem has 
no bands (indicating no order-one contributions) and no parametrices 
(indicating no algebraically decaying contributions).  

\vspace{.1in}

\noindent
{\bf The algebraic-decay region.}  Here the leading-order solution decays 
as $n^{-1/2}$ and is given explicitly in terms of elementary functions.  The 
Riemann-Hilbert model problem consists of no bands and two parabolic-cylinder 
parametrices giving the leading-order contribution to the solution.  

\vspace{.1in}

\noindent
{\bf The non-oscillatory region.} In this region the leading-order solution 
is independent of $n$ and can be written explicitly up to the solution of a 
septic equation.  The model Riemann-Hilbert problem has a single band.  

\vspace{.1in}

\noindent
{\bf The oscillatory region.} In the final region the solution exhibits 
rapid oscillations with frequency of order $n$ within an amplitude envelope of 
order one.  The leading-order behavior is written in terms of genus-one 
Riemann-theta functions.  The corresponding Riemann-Hilbert model problem has 
two bands.  

\vspace{.1in}

The four far-field regions depend on $\xi$ but are independent of ${\bf c}$.  
The regions are illustrated for $\xi=i$ in Figure \ref{fig-far-field-2d-boundary}.
\begin{figure}[ht]
\begin{center}
\includegraphics[height=2.8in]{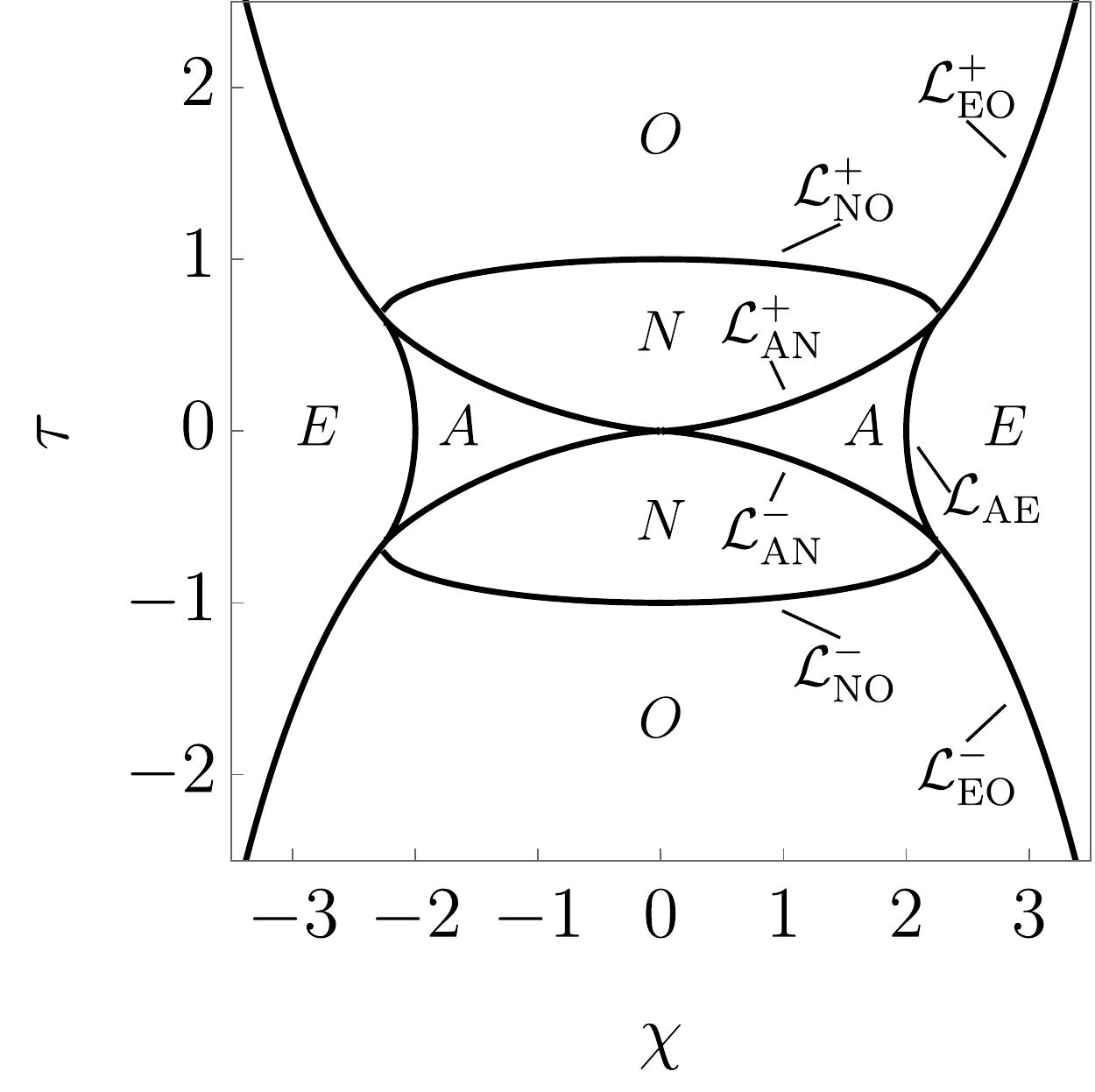}
\includegraphics[height=2.8in]{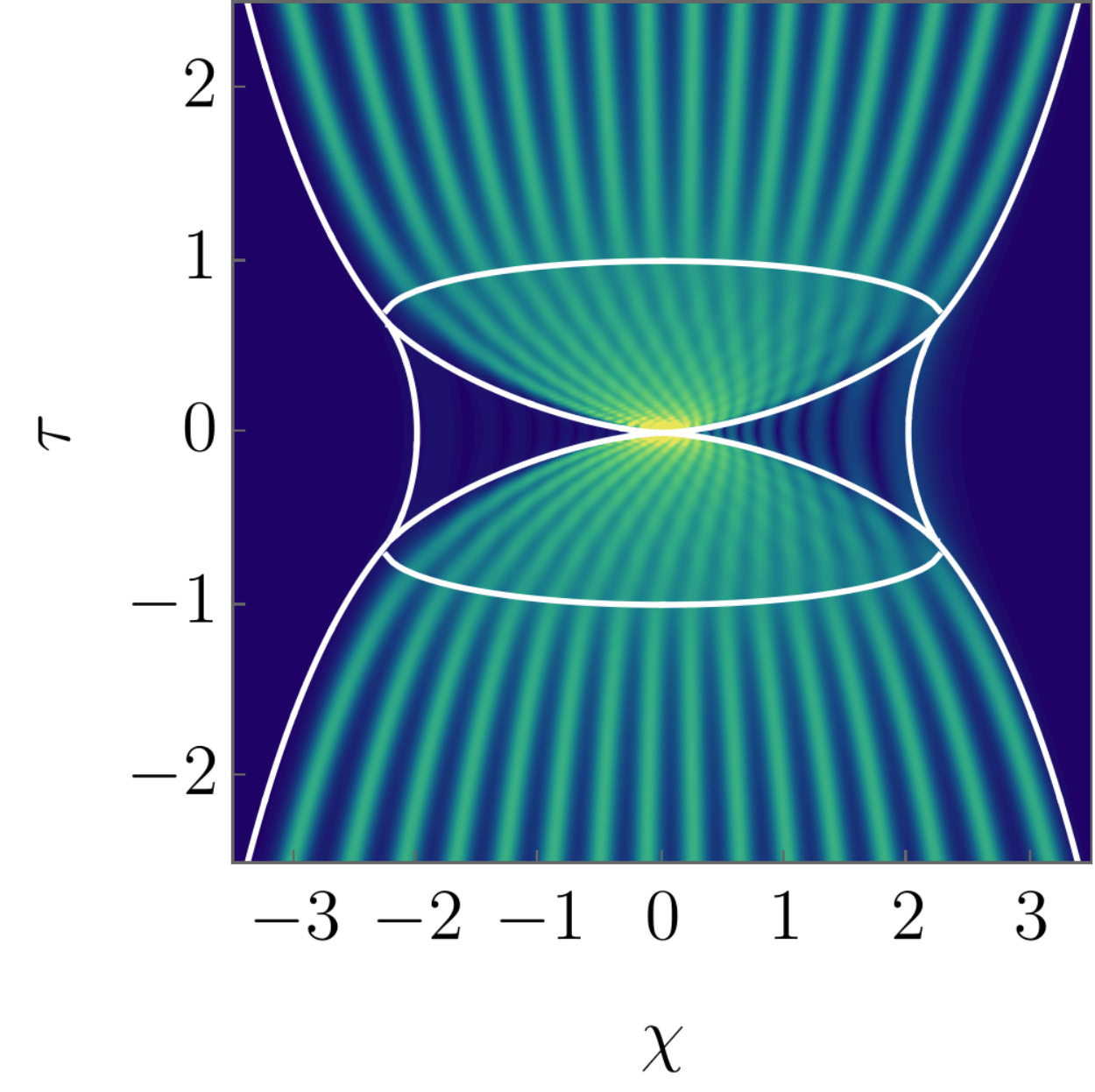}
\caption{The boundaries of the far-field regions.  \emph{Left}:  The algebraic-decay, exponential-decay, non-oscillatory, and oscillatory regions (denoted by A, E, N, and O, respectively), along with the various boundary curves for $\xi=i$.  \emph{Right}: The boundary curves superimposed on $|\psi^{[2n]}(n\chi,n\tau;(1,3),i)|$ with $c_1=1$, $c_2=3$, and $\xi=i$ for $-3.5\leq \chi\leq 3.5$ and $-2.5\leq \tau\leq 2.5$.}
\label{fig-far-field-2d-boundary}
\end{center}
\end{figure}

\subsection{The far-field regions.}
\label{subsec-regions}
In order to give our exact results we start by defining the region boundaries.  
We write $\xi=\alpha+i\beta$, $\alpha\in\mathbb{R}$, $\beta>0$.  

\vspace{.1in}

\noindent
{\bf Definition of the boundaries of the algebraic-decay region.}  Define 
\eq
\label{phi-def}
\varphi(\lambda;\chi,\tau;\xi) :=i(\lambda \chi+\lambda^2 \tau) + \log\left(\frac{\lambda-\xi^*}{\lambda-\xi}\right).
\endeq
This is the controlling phase function in the exponential-decay and 
algebraic-decay regions.  The critical points of $\varphi(\lambda)$ satisfy
\eq
2\tau(\lambda-\alpha)^3 + (\chi+2\alpha\tau)(\lambda-\alpha)^2 + 2\beta^2\tau(\lambda-\alpha) + (\beta^2\chi-2\beta+2\alpha\beta^2\tau) = 0.  
\endeq
First, set $\tau=0$ and $0<\chi<\frac{2}{\beta}$.  Then $\varphi(\lambda)$ has 
two real distinct critical points $\lambda^{(1)}$ and $\lambda^{(2)}$, where we choose  
$\lambda^{(1)}<\lambda^{(2)}$ (the third critical point is at infinity).  See Figure 
\ref{algebraic-phase-plots}.  The 
algebraic-decay region (with $\chi>0$) consists of those $\chi$ and $\tau$ 
values that can be reached by continuously varying $\chi$ and $\tau$ with no 
two critical points coinciding.  In this region if $\tau\neq 0$ 
then $\varphi(\lambda)$ has three distinct real critical points, which we 
label $\lambda^{(0)}<\lambda^{(1)}<\lambda^{(2)}$ if $\tau>0$ and 
$\lambda^{(1)}<\lambda^{(2)}<\lambda^{(0)}$ if $\tau<0$.  The region is bounded by the 
locus of points in the ($\chi$,$\tau$)-plane satisfying
\eq
\label{algebraic-boundary}
\begin{split}
& (16\alpha^4\beta+32\alpha^2\beta^3+16\beta^5)\tau^4 + (32\alpha^3\beta \chi - 16\alpha^3 + 32\alpha\beta^3 \chi - 144\alpha\beta^2)\tau^3 \\
& + (24\alpha^2\beta\chi^2 - 24\alpha^2 \chi + 8\beta^3 \chi^2 - 72\beta^2 \chi + 108\beta)\tau^2 + (8\alpha\beta \chi^3 - 12\alpha \chi^2)\tau + (\beta\chi^4 - 2\chi^3) = 0.
\end{split}
\endeq
For real $\alpha$ and positive $\beta$, this algebraic curve consists of three 
arcs in the ($\chi$,$\tau$)-plane that intersect pairwise at the three points 
\eq
P^0:=(0,0), \quad P^+:=\left(\frac{-3\sqrt{3}\alpha+9\beta}{4\beta^2},\frac{3\sqrt{3}}{8\beta^2}\right), \quad P^-:=\left(\frac{3\sqrt{3}\alpha+9\beta}{4\beta^2},\frac{-3\sqrt{3}}{8\beta^2}\right)
\endeq
(each of these three points corresponds to $\lambda^{(1)}=\lambda^{(2)}=\lambda^{(0)}$).  
The arc with endpoints $P^-$ and $P^+$ passes through the point 
$\left(\frac{2}{\beta},0\right)$ on the $\chi$-axis and is denoted by 
$\mathcal{L}_\text{AE}$.  This arc is a boundary between the algebraic-decay 
and the exponential-decay regions and corresponds to $\lambda^{(1)}=\lambda^{(2)}$.  
The arc from $P^0$ to $P^+$ is denoted by $\mathcal{L}_\text{AN}^+$ (and 
corresponds to $\lambda^{(1)}=\lambda^{(0)}$), while that from $P^0$ to $P^-$ is 
denoted by $\mathcal{L}_\text{AN}^-$ (and corresponds to 
$\lambda^{(2)}=\lambda^{(0)}$).  Both of these arcs form boundaries between the 
algebraic-decay region and the non-oscillatory region.  Note that if $\xi=i$, 
the defining condition \eqref{algebraic-boundary} for the boundary of the 
algebraic-decay region simplifies to 
\eq
\label{xi-i-boundary-quadratic}
16 \tau^4 + (8 \chi^2 - 72 \chi + 108)\tau^2 + (\chi^4 - 2\chi^3) = 0.
\endeq

\vspace{.1in}

\noindent
{\bf Definition of the exponential-decay / oscillatory boundary.}
We now define $\mathcal{L}_\text{EO}^\pm$, the boundaries between the 
exponential-decay and oscillatory regions when $\chi>0$.  Set $\tau=0$ and 
choose $\chi>\frac{2}{\beta}$.  Then $\varphi(\lambda)$ has a complex-conjugate 
pair of critical points $\lambda^+$ and $\lambda^-$, where we choose 
$\lambda^+$ to be in the upper half-plane.  See Figure \ref{exponential-phase-plots}. 
Here we have that 
$\Re(\varphi(\lambda^\pm))\neq 0$.  The exponential-decay region consists of 
those $(\chi,\tau)$ pairs we can reach by continuously varying $\chi$ and 
$\tau$ such that no two critical points coincide \emph{and} such that the 
level lines $\Re(\varphi(\lambda))=0$ never intersect either of the two 
critical points with nonzero imaginary part (which we continue to label as 
$\lambda^\pm$).  In this region if $\tau\neq 0$ then there is a third finite 
critical point which is real and that we label as $\lambda^{(0)}$.  The curve 
$\mathcal{L}_\text{AE}$ corresponds to $\lambda^+=\lambda^-$.  The curve 
$\mathcal{L}_\text{EO}^+$ (respectively, $\mathcal{L}_\text{EO}^-$) is defined 
as those points with $\tau>0$ (respectively, $\tau<0$) such that 
$\Re(\varphi(\lambda^+))=\Re(\varphi(\lambda^-))=0$.  Both 
$\mathcal{L}_\text{EO}^+$ and $\mathcal{L}_\text{EO}^-$ are simple, 
semi-infinite curves with endpoints $P^+$ and $P^-$, respectively.  

\vspace{.1in}

\noindent
{\bf Definition of the oscillatory / non-oscillatory boundary.}
Finally, we define $\mathcal{L}_\text{NO}^+$, the boundary between the 
oscillatory and non-oscillatory regions when $\tau>0$.  Given a complex 
number $a=a(\chi,\tau)$, define 
\eq
\label{r-def}
R(\lambda)\equiv R(\lambda;\chi,\tau):=((\lambda-a(\chi,\tau))(\lambda-a(\chi,\tau)^*))^{1/2}
\endeq
with asymptotic behavior $R(\lambda)=\lambda+\mathcal{O}(1)$ as 
$\lambda\to\infty$ and branch cut from $a^*$ to $a$ (we will completely 
specify the branch cut momentarily).  Set 
\eq 
\label{gprime-def}
g'(\lambda):=\frac{R(\lambda)}{R(\xi^{*})(\xi^{*}-\lambda)}-\frac{R(\lambda)}{R(\xi)(\xi-\lambda)}-2i\tau R(\lambda)+i\chi+2i\tau\lambda+\frac{1}{\lambda-\xi^*}-\frac{1}{\lambda-\xi}.
\endeq
Then $a(\chi,\tau)$ is chosen so that $g'(\lambda)=\mathcal{O}(\lambda^{-2})$ 
as $\lambda\to\infty$.  The function $\varphi'(\lambda)-g'(\varphi)$ (which 
will turn out to be the derivative of the controlling phase function in the 
non-oscillatory region) has two real zeros if 
$(\chi,\tau)\in\mathcal{L}_\text{AN}^+$.  One  zero is simple (corresponding to 
$\lambda^{(2)}$ from the algebraic-decay region) and one zero is double 
(corresponding to $\lambda^{(0)}=\lambda^{(1)}$ from the algebraic-decay 
region).  See Figure \ref{nonoscillatory-phase-plots}.  Keeping $\chi$ fixed and 
increasing $\tau$, the double zero splits 
into one real zero (denoted by $\lambda^{(1)}$) and two square-root branch 
points at $a$ and $a^*$.  The simple real zero persists and is again denoted by 
$\lambda^{(2)}$.  See Figure \ref{oscillatory-phase-plots}.  We now choose the branch 
cut for $R(\lambda)$ (and thus the 
cut for $g'(\lambda)$ as well) to run from $a^*$ to $\lambda^{(1)}$ to $a$.  
As $\chi$ increases, the non-oscillatory region continues until the two real 
zeros coincide:  $\lambda^{(1)}=\lambda^{(2)}$.  This is the condition for the 
contour $\mathcal{L}_\text{NO}$ {separating} the non-oscillatory and oscillatory 
regions.  

The exponential-decay, algebraic-decay, non-oscillatory, and oscillatory regions 
are now defined by these boundary curves as illustrated in Figure 
\ref{fig-far-field-2d-boundary}.  

\subsection{Results.}  We now give our main results, the leading-order 
asymptotic behavior in each of the four regions. 
{The symmetry properties of $\psi^{[2n]}(x,t)$ stated in Proposition~\ref{prop:symmetries} allow us to restrict our analysis to the first quadrant of the $(\chi,\tau)$ plane without loss of generality.}

\begin{theorem}
\label{exp-decay-thm}
{\rm (The exponential-decay region).}  
{ Fix $\chi>0$ and $\tau\geq 0$ so that} $(\chi,\tau)$ is in the exponential-decay region.  Then 
\eq
\psi^{[2n]}(n\chi,n\tau) = \mathcal{O}(e^{- \delta n}),\quad { n\to +\infty},
\endeq
for some constant $\delta>0$.  
\end{theorem}
Theorem \ref{exp-decay-thm} was proven in \cite[\S 3]{BilmanB:2019}.

\begin{theorem}
\label{alg-decay-thm}
{\rm (The algebraic-decay region).}  
{ Fix $\chi>0$ and $\tau\geq 0$ so that} $(\chi,\tau)$ is in the algebraic-decay region.  
Let $\lambda^{(1)}$, $\lambda^{(2)}$, and $\lambda^{(0)}$ be 
the real critical points of $\varphi(\lambda)$ as defined in 
\S\ref{subsec-regions} 
{ with $\lambda^{(0)}<\lambda^{(1)}<\lambda^{(2)}$ if $\tau>0$ and 
$\lambda^{(1)}<\lambda^{(2)}$ (and $\lambda^{(0)}=\infty$) if $\tau=0$}.  
Define 
\eq
p:=\frac{1}{2\pi}\log\left(1+\left|\frac{c_2}{c_1}\right|^2\right) \quad \text{and} \quad \nu:=\arg\left(\frac{c_2}{c_1}\right),
\endeq
where $\log(\cdot)$ and $\arg(\cdot)$ each have the principal branch.  
Also introduce 
\eq
\label{theta-def}
\theta(\lambda;\chi,\tau):=-i\varphi(\lambda;\chi,\tau)
\endeq
and 
\begin{equation}
\phi^{[n]}(\chi,\tau):= p \log(n) + 2 p \log\left(\lambda^{(2)}(\chi,\tau) - \lambda^{(1)}(\chi,\tau)\right) + \frac{\pi}{4} + p \log(2) - \arg(\Gamma(i p)),
\end{equation}
where $\Gamma(\cdot)$ is the standard gamma function.  Then 
\eq
\label{psi2n-alg-result}
\begin{split}
\psi^{[2n]}(n\chi,n\tau) = \frac{ \sqrt{2p}\, e^{-i \nu}}{n^{1/2}}&\left(\frac{ e^{-2i {n}\theta(\lambda^{(1)};\chi, \tau)}(-\theta''(\lambda^{(1)};\chi, \tau))^{-i p} }{\sqrt{-\theta''(\lambda^{(1)};\chi, \tau)}}  e^{-i\phi^{[n]}(\chi,\tau)}\right. \\
&\hspace{.2in}\left.  + \frac{ e^{-2i {n} \theta(\lambda^{(2)};\chi, \tau)}\theta''(\lambda^{(2)};\chi, \tau)^{i p} }{\sqrt{\theta''(\lambda^{(2)};\chi, \tau)}}  e^{i\phi^{[n]}(\chi,\tau)} \right)+\mathcal{O}(n^{-1}), \quad { n\to +\infty}.
\end{split}
\endeq
\end{theorem}
Theorem \ref{alg-decay-thm} is proven in \S\ref{sec-alg-decay}.  Figure
\ref{algebraic-sol-plots-c13} compares the 
exact solution to the leading-order behavior for various values of $n$.  
\begin{figure}[ht]
\begin{center}
\hspace{.2in} $n=2$ \hspace{1.4in} $n=4$ \hspace{1.6in} $n=8$ \\
\includegraphics[height=1.4in,width=2in]{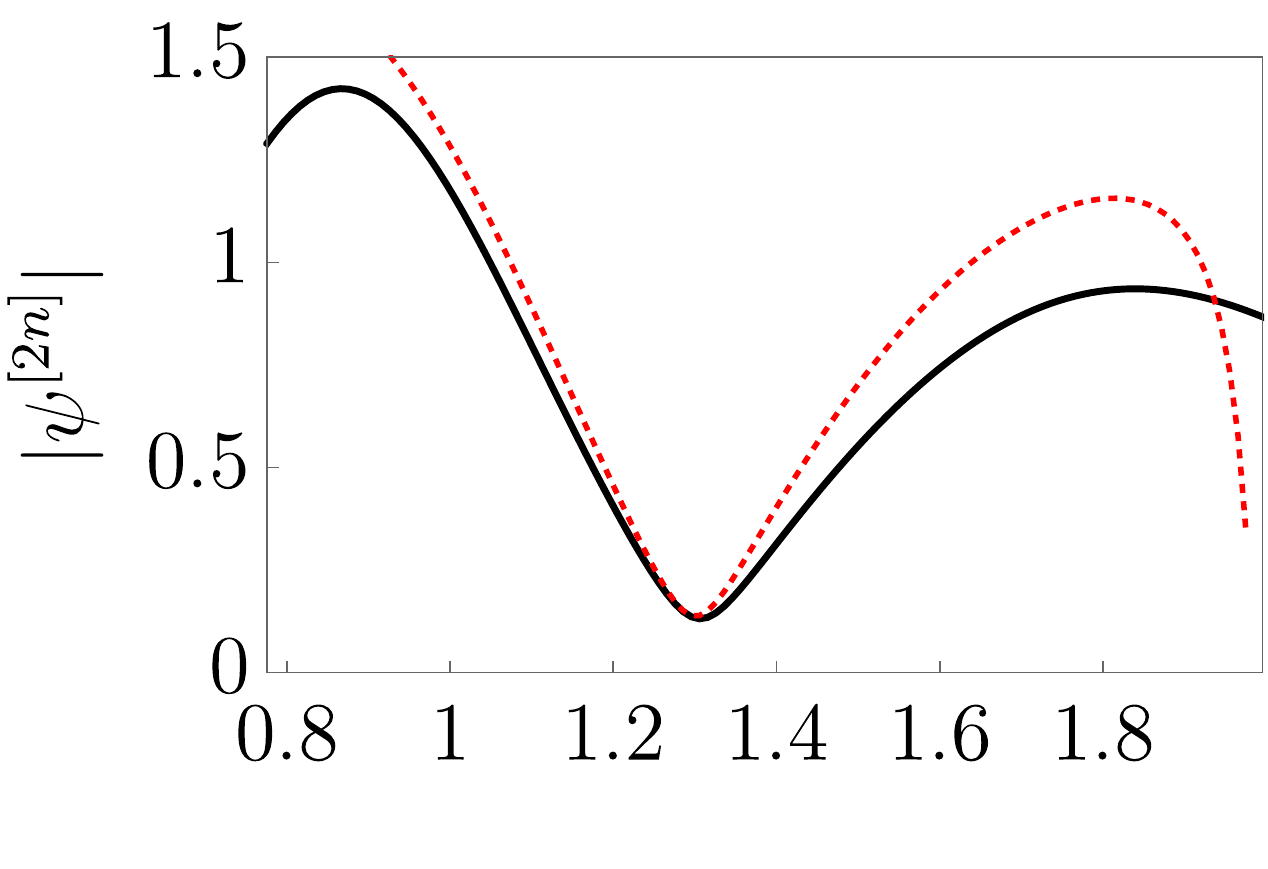}
\includegraphics[height=1.4in,width=2in]{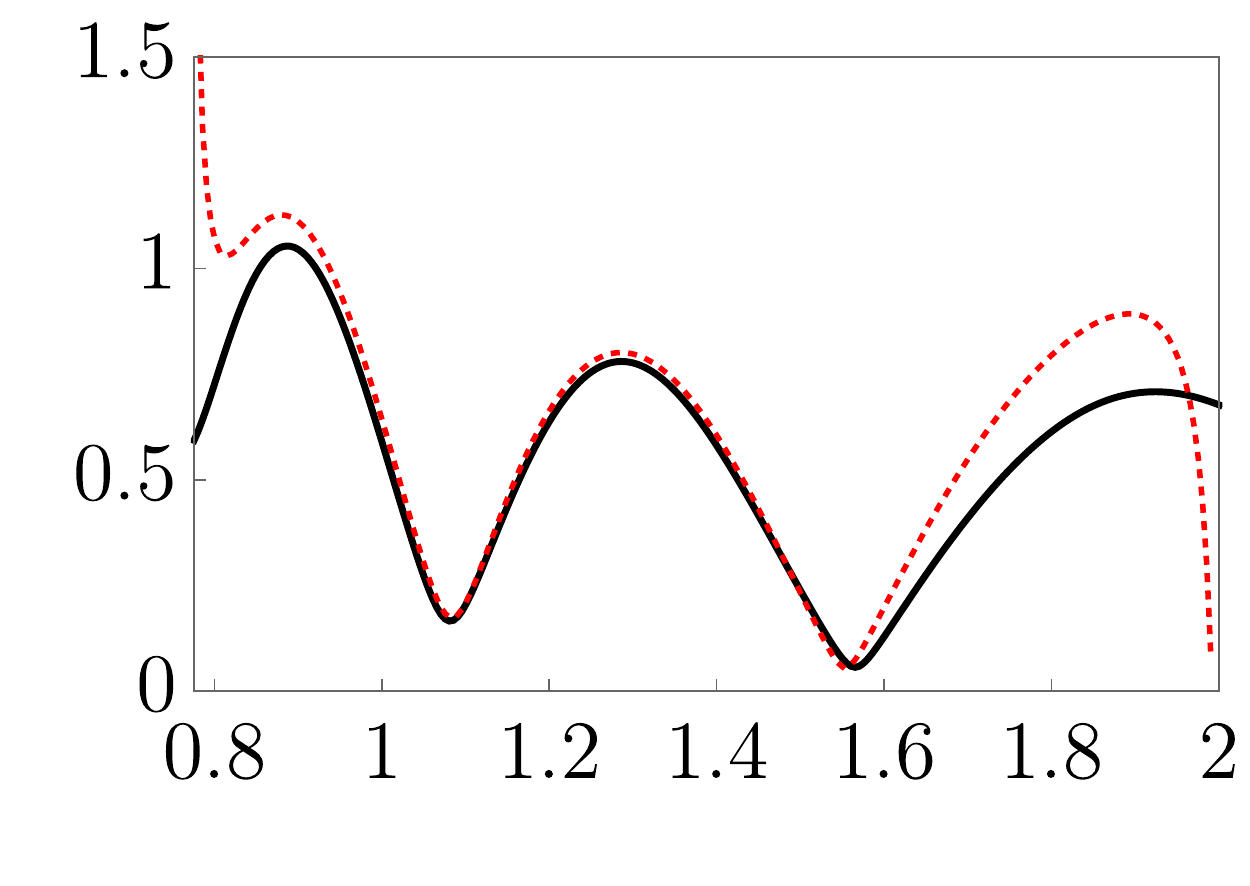}
\includegraphics[height=1.4in,width=2in]{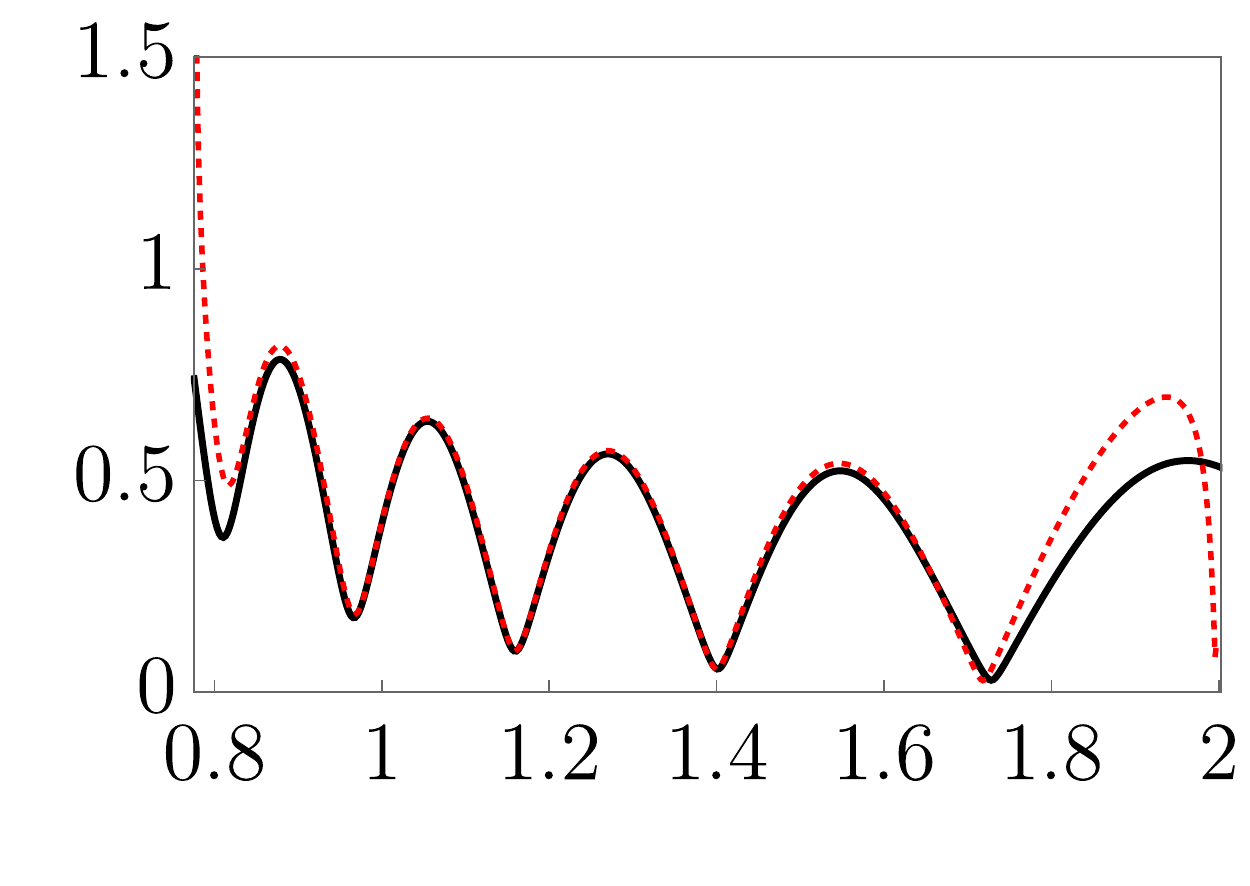}\\
\includegraphics[height=1.4in,width=2in]{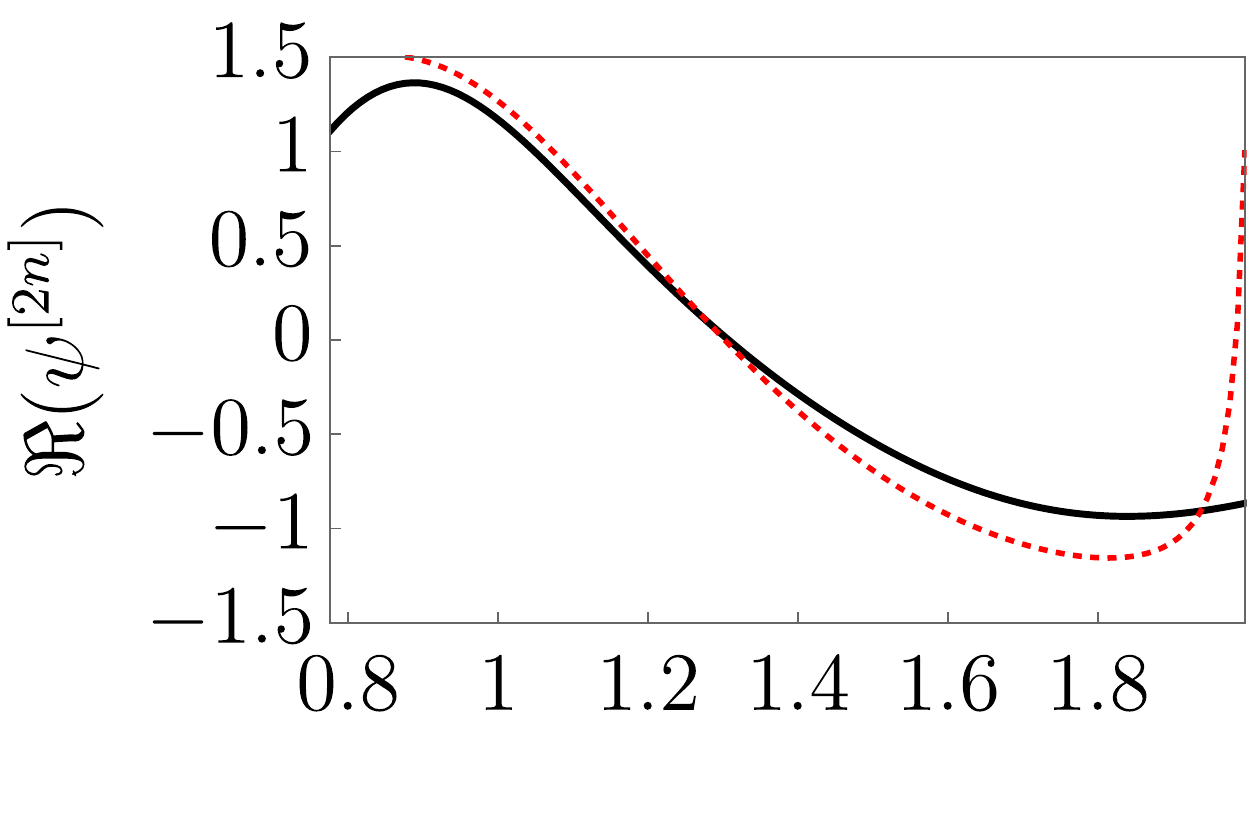}
\includegraphics[height=1.4in,width=2in]{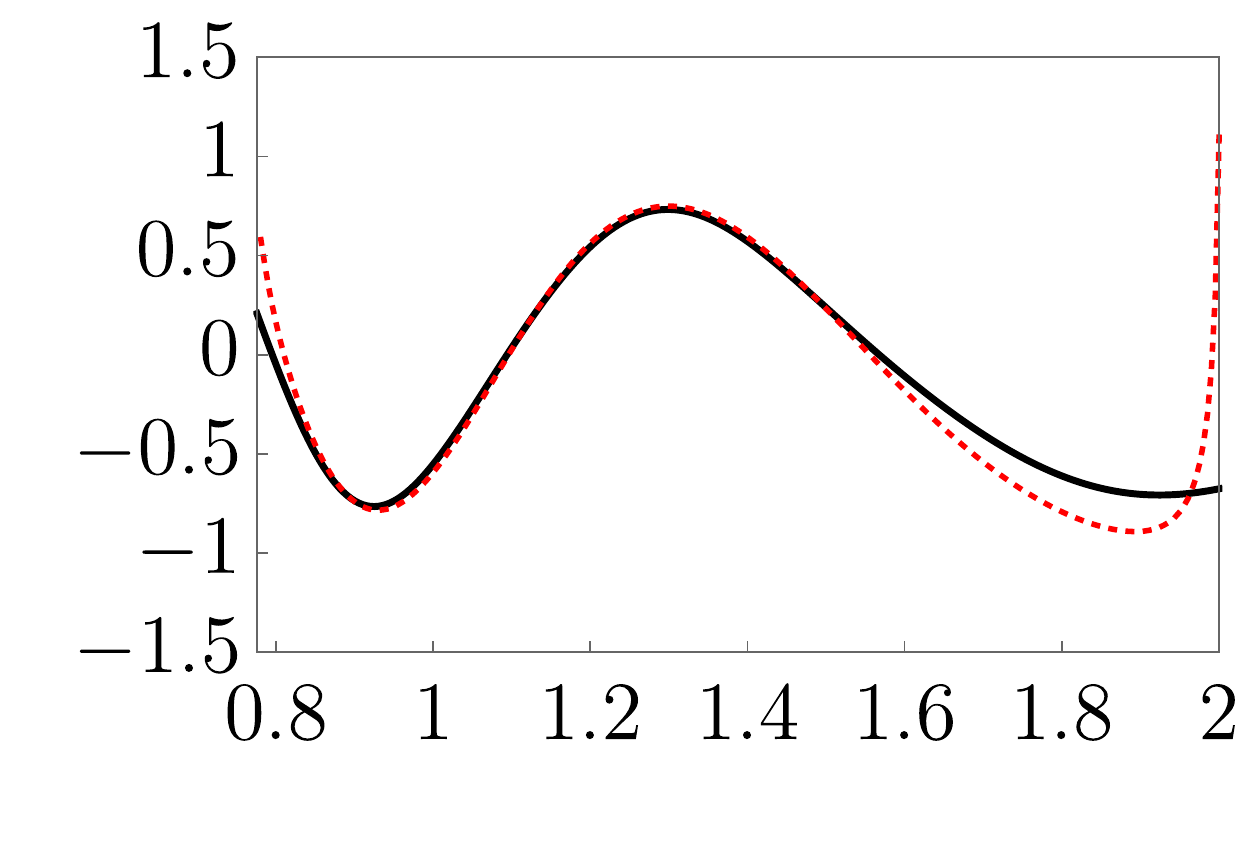}
\includegraphics[height=1.4in,width=2in]{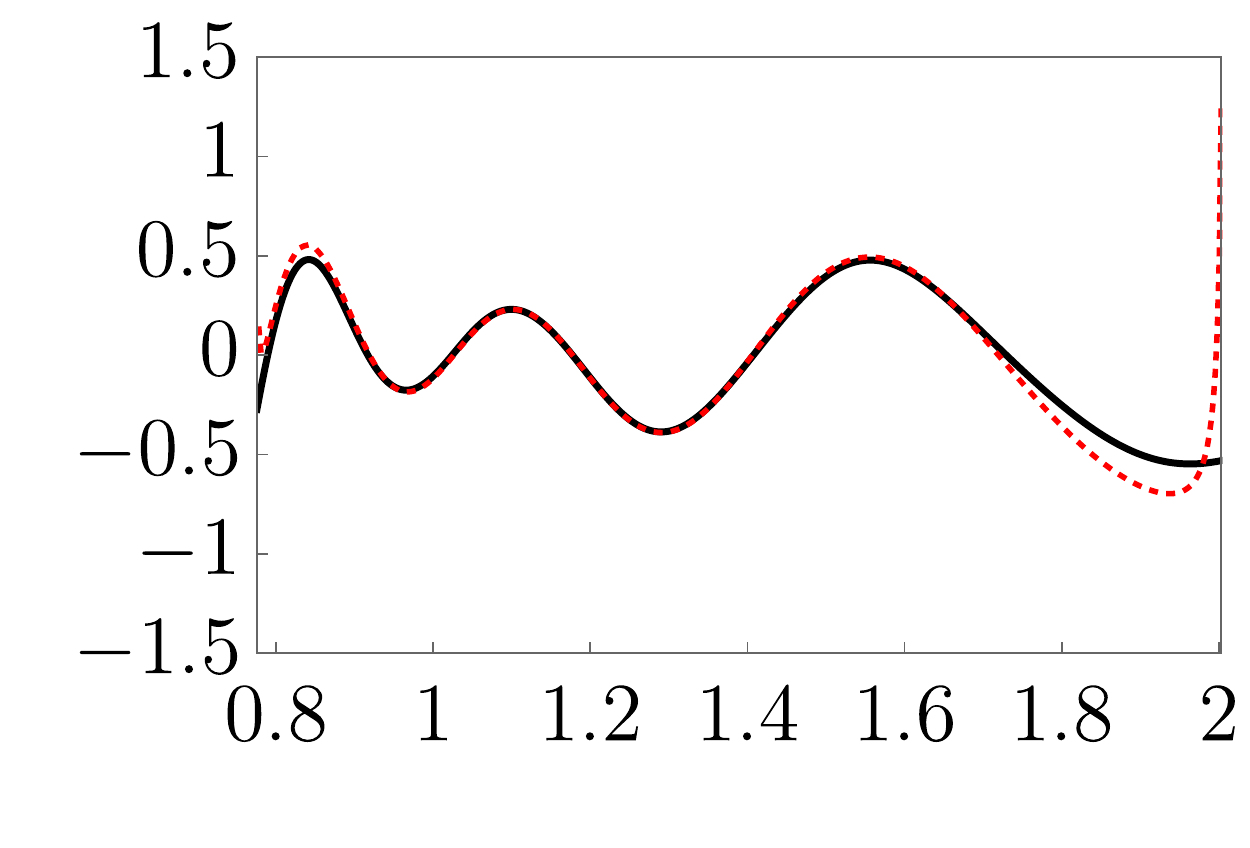}\\
\includegraphics[height=1.5in,width=2in]{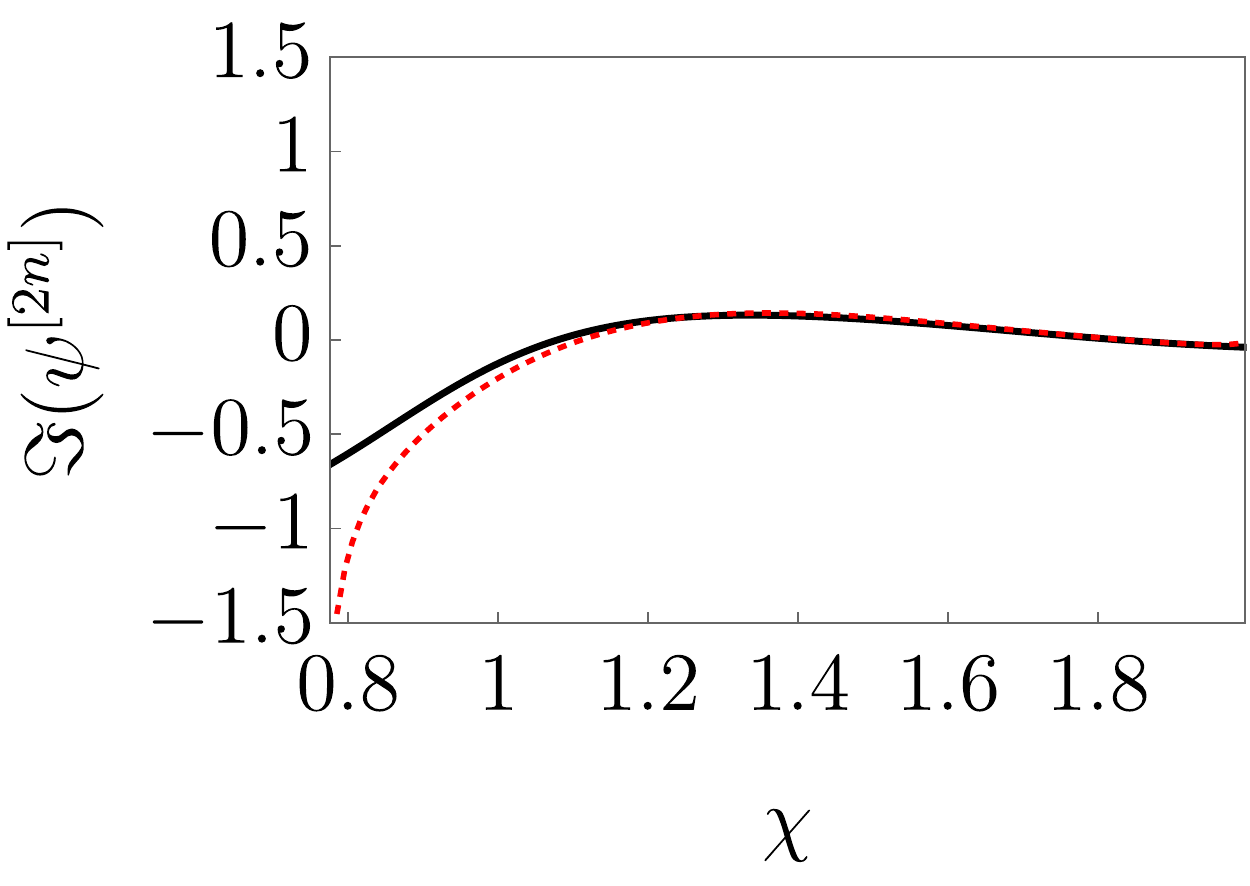}
\includegraphics[height=1.5in,width=2in]{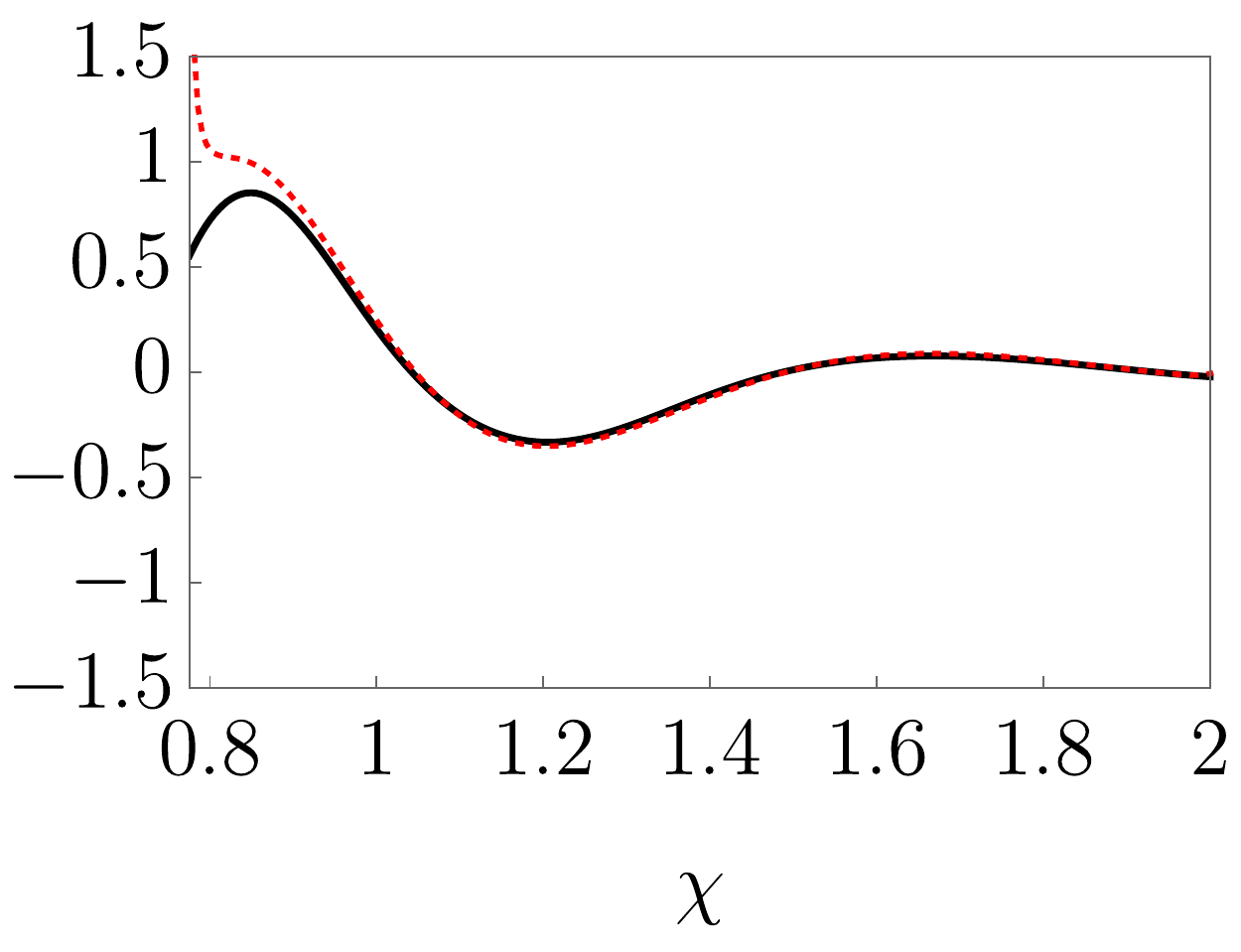}
\includegraphics[height=1.5in,width=2in]{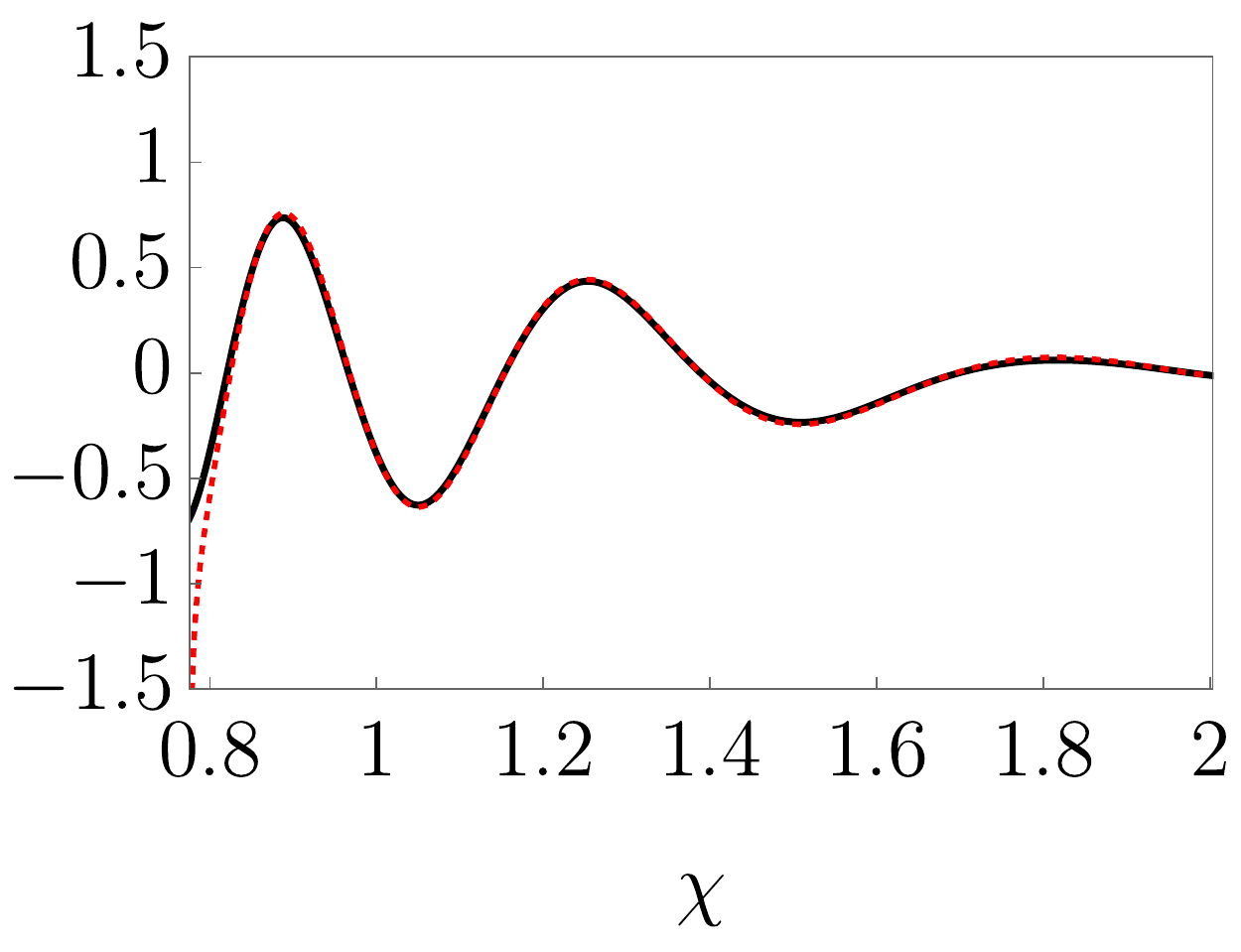}
\caption{Convergence of the leading-order asymptotic approximation in the 
algebraic-decay region for $\xi=i$ and ${\bf c}=(1,3)$ at $\tau=\frac{1}{10}$. 
Solid black curves are for the exact solution 
$\psi^{[2n]}(n\chi,n\frac{1}{10};(1,3),i)$ while 
dashed red curves are for the leading-order approximation given by Theorem 
\ref{alg-decay-thm}.  For this time slice the algebraic-decay region 
(with $\chi\geq 0$) is approximately $0.7756<\chi<2.0050$.  
\emph{Left-to-right}:  $n=2$, $n=4$, $n=8$.  \emph{Top-to-bottom}:  The 
absolute value, real part, and imaginary part.}
\label{algebraic-sol-plots-c13}
\end{center}
\end{figure}
\begin{theorem}
\label{nonosc-thm}
{\rm (The non-oscillatory region).}  
{ Fix $\chi\geq 0$ and $\tau>0$ so that} $(\chi,\tau)$ is in the non-oscillatory region.  Recall that in this region 
$R(\lambda)$ and $g'(\lambda)$ are defined in \eqref{r-def} and 
\eqref{gprime-def}, respectively.  Let $a(\chi,\tau)$ be defined as before so 
that $g'(\lambda)=\mathcal{O}(\lambda^{-2})$ as $\lambda\to\infty$, and 
define $K(\chi,\tau)$ by \eqref{K-def} and $f(\infty;\chi,\tau)$ by 
\eqref{finf-def}.  Then 
\eq
\psi^{[2n]}(n\chi,n\tau) = -i\Im(a(\chi,\tau))e^{-2f(\infty;\chi,\tau)}{e^{-2 i n K(\chi,\tau)}} + \mathcal{O}\left(\frac{1}{n^{1/2}}\right),\quad {n\to+\infty}.
\endeq
\end{theorem}
Theorem \ref{nonosc-thm} is proven in \S\ref{sec-nonosc}.  Figure
\ref{nonosc-sol-plots-c13} compares the exact solution to the 
leading-order behavior for various values of $n$.  
\begin{figure}[t]
\begin{center}
\hspace{.2in} $n=2$ \hspace{1.4in} $n=4$ \hspace{1.6in} $n=8$ \\
\includegraphics[height=1.4in,width=2in]{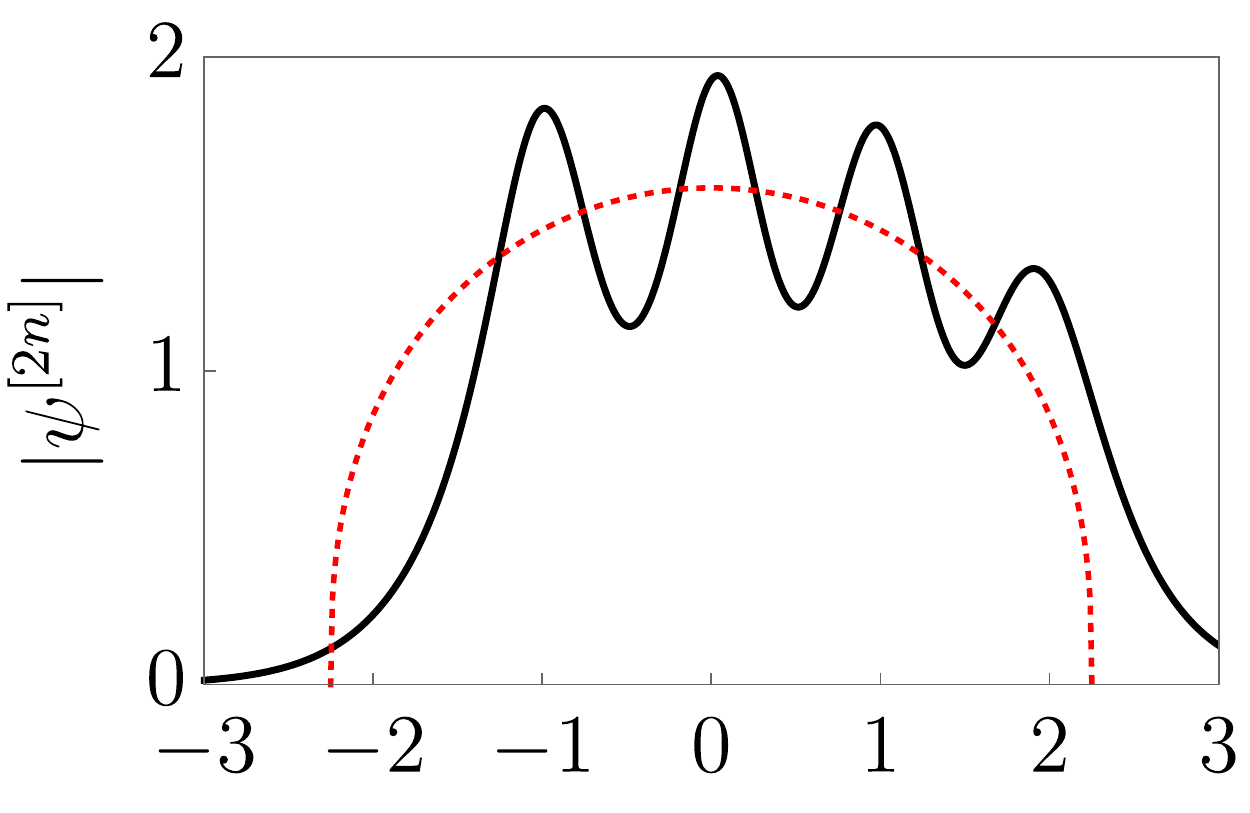}
\includegraphics[height=1.4in,width=2in]{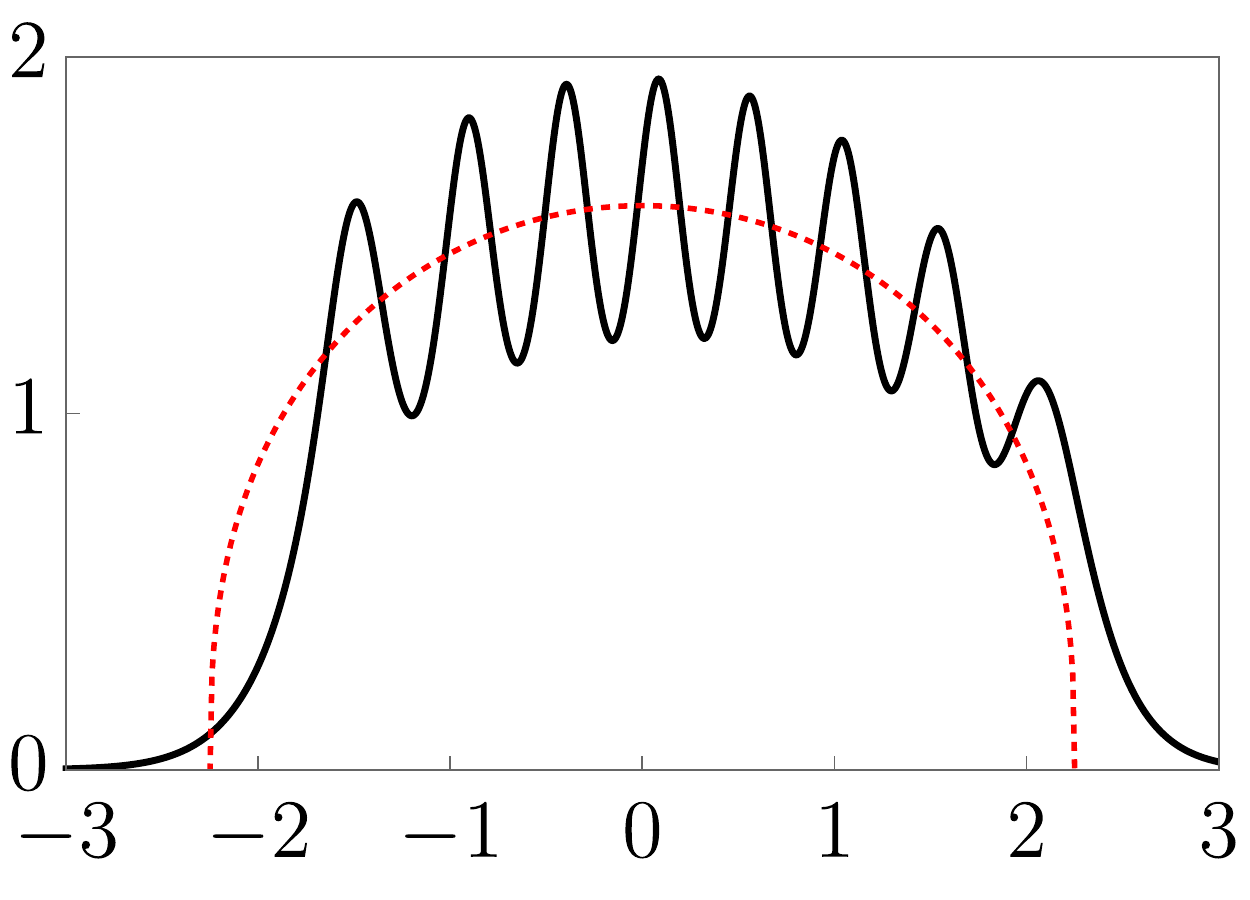}
\includegraphics[height=1.4in,width=2in]{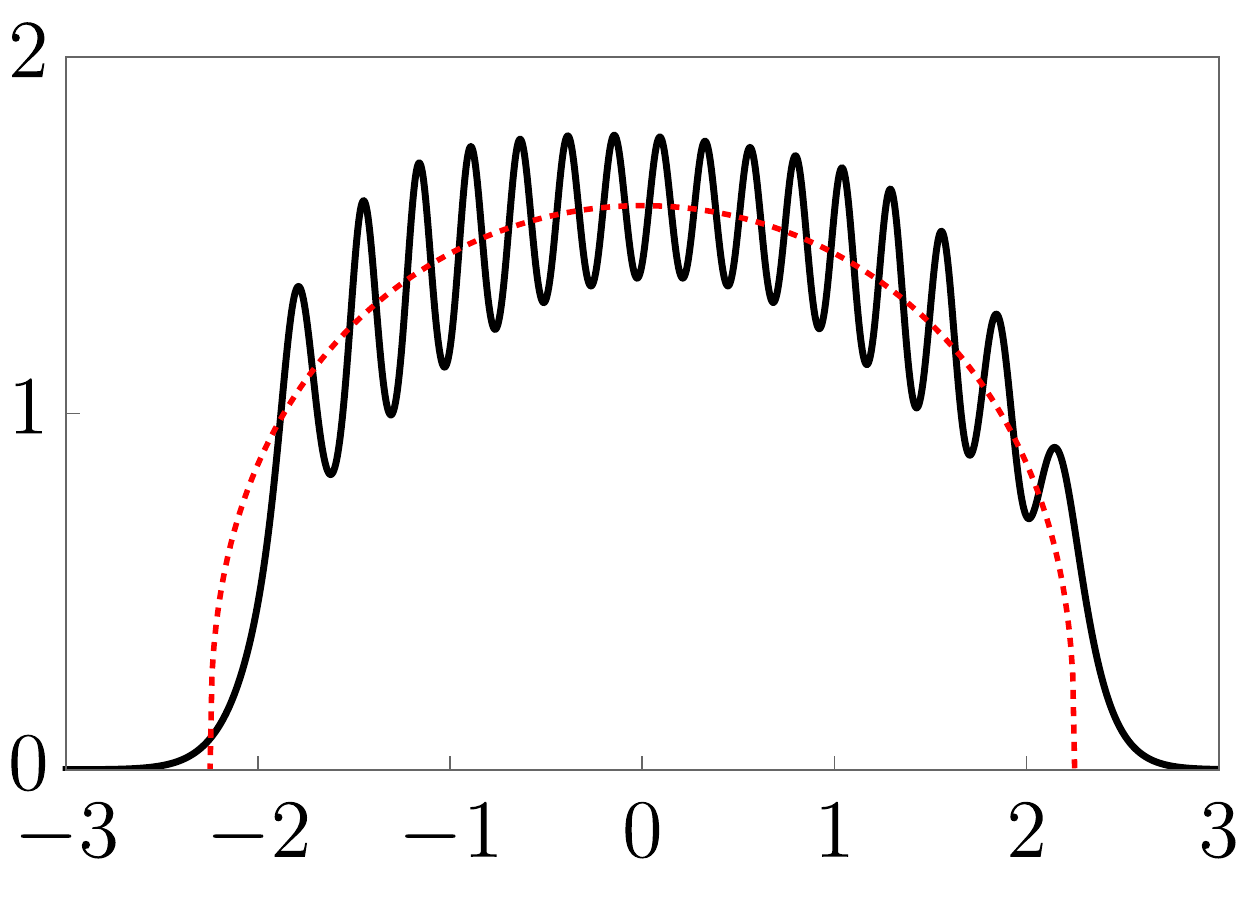}\\
\includegraphics[height=1.4in,width=2in]{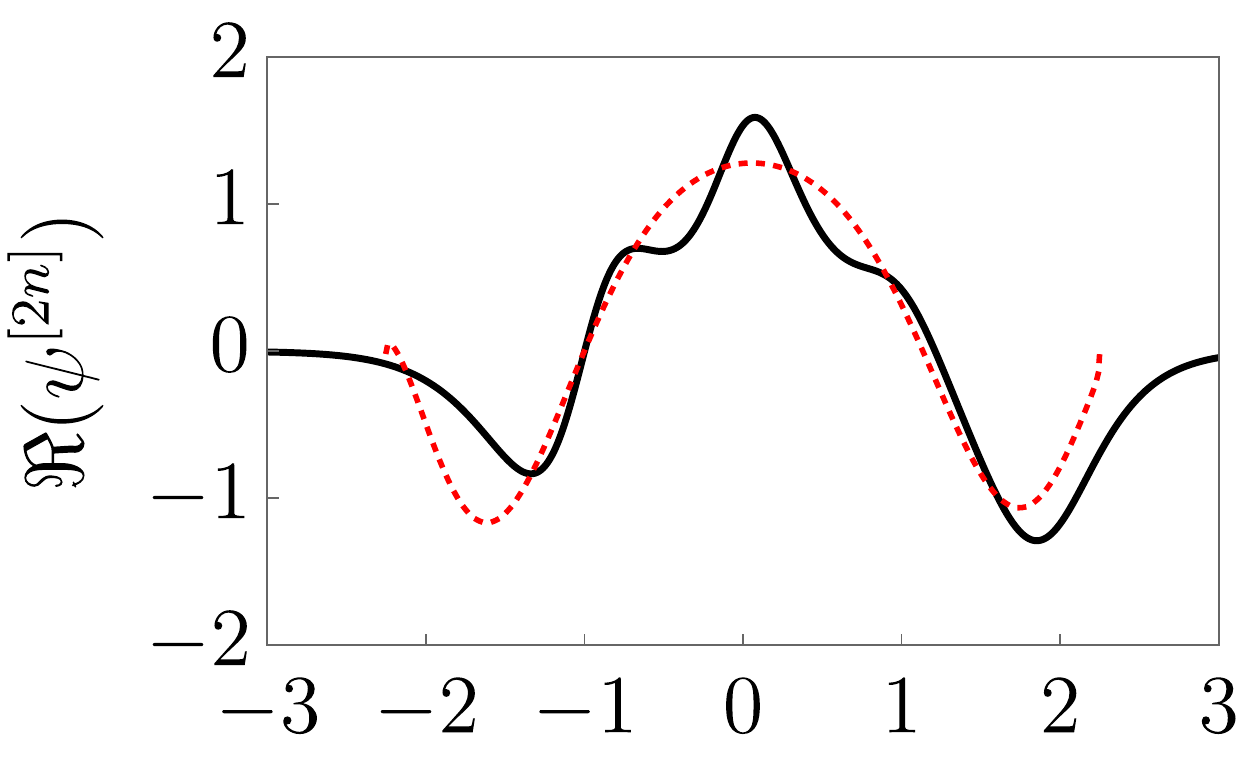}
\includegraphics[height=1.4in,width=2in]{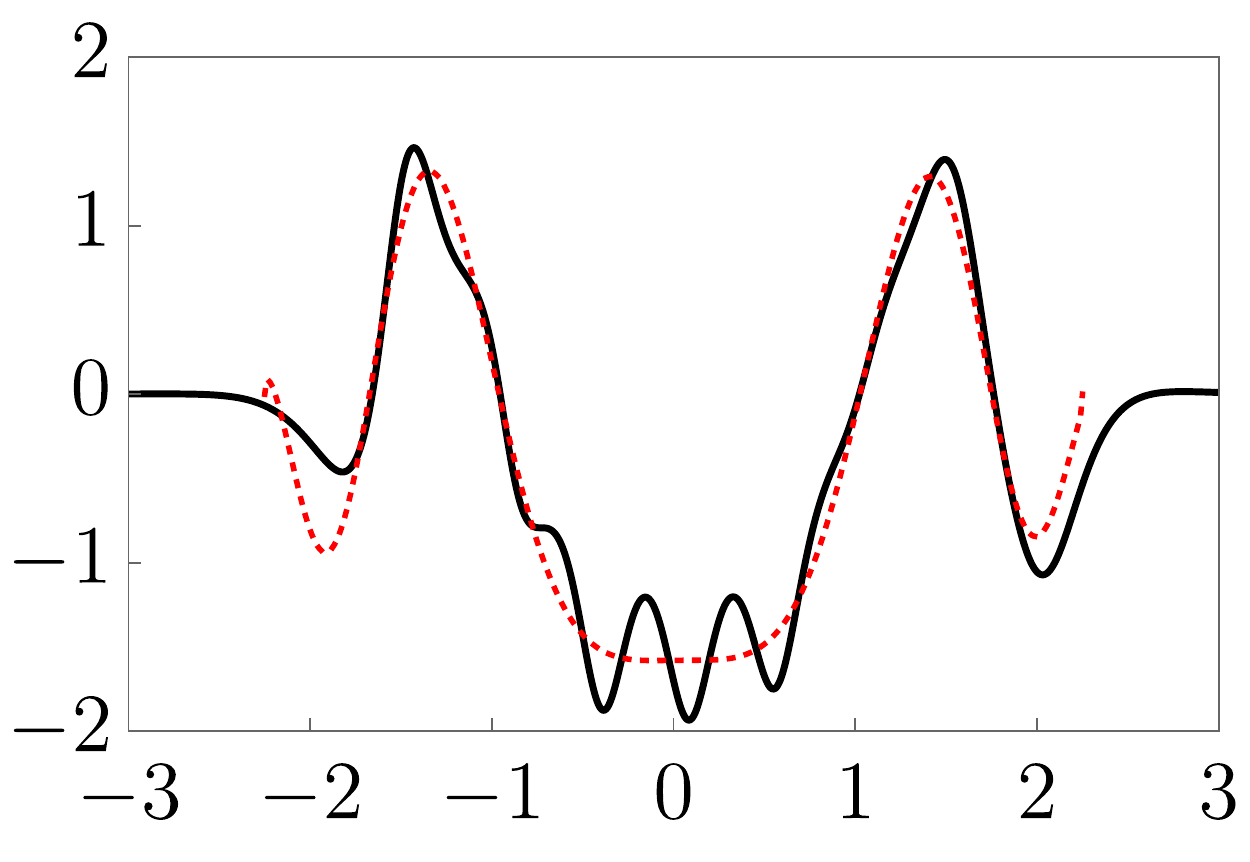}
\includegraphics[height=1.4in,width=2in]{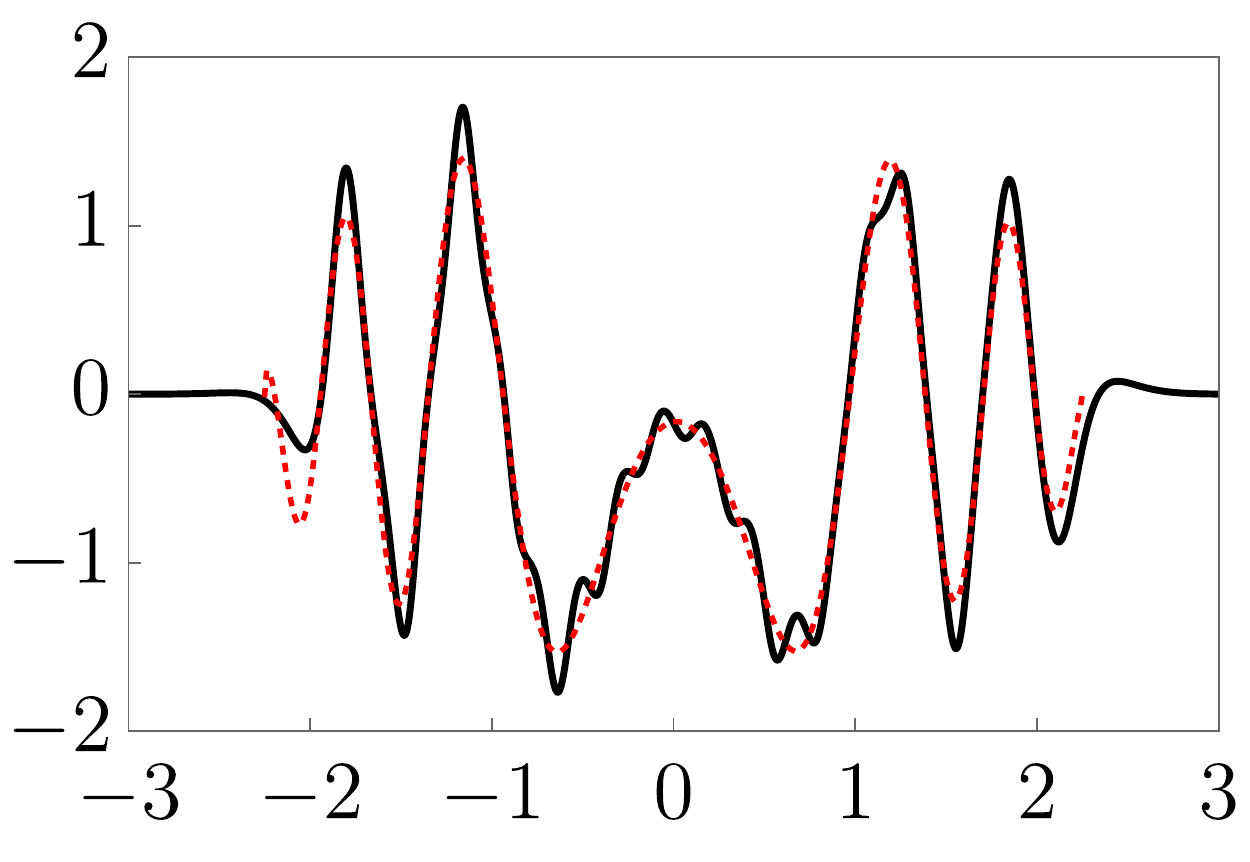}\\
\includegraphics[height=1.5in,width=2in]{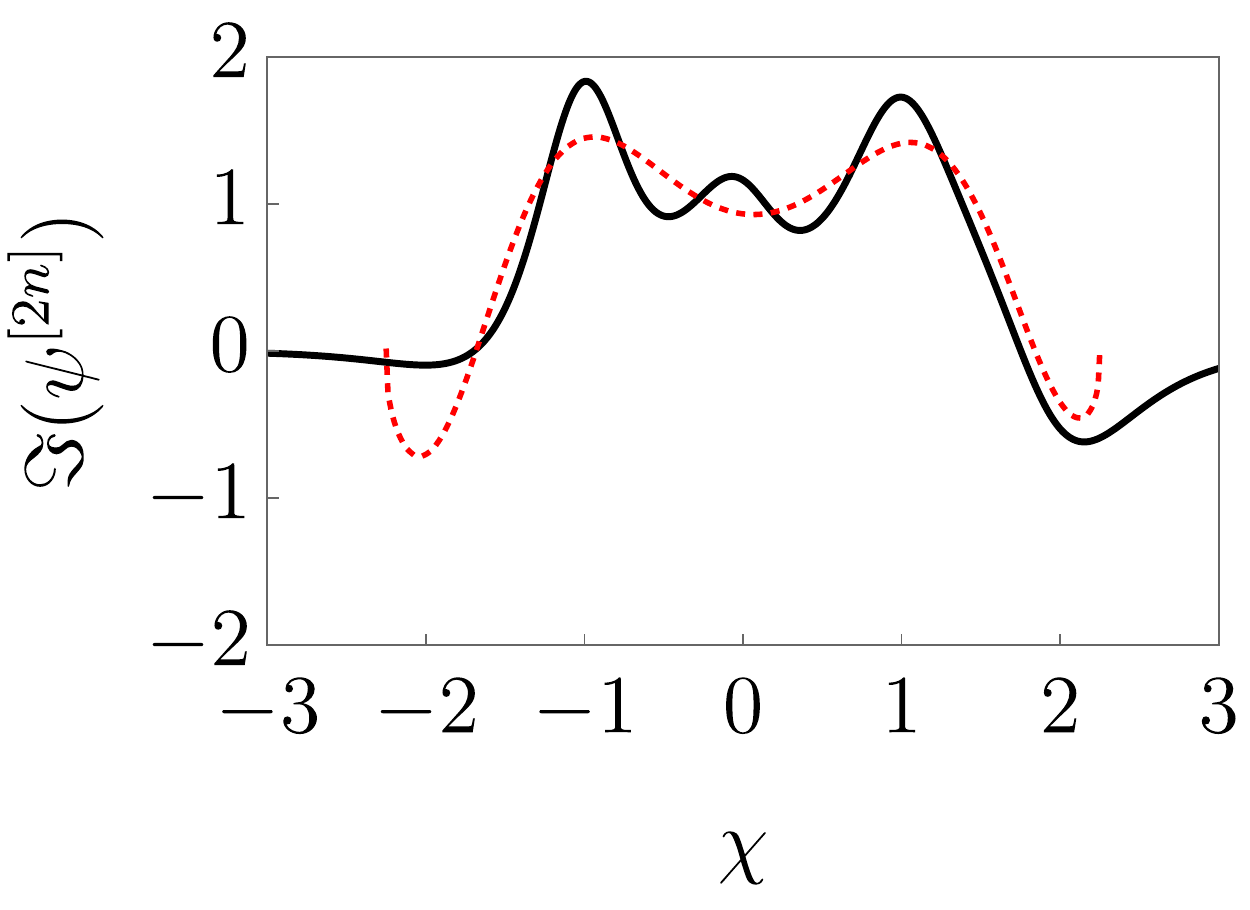}
\includegraphics[height=1.5in,width=2in]{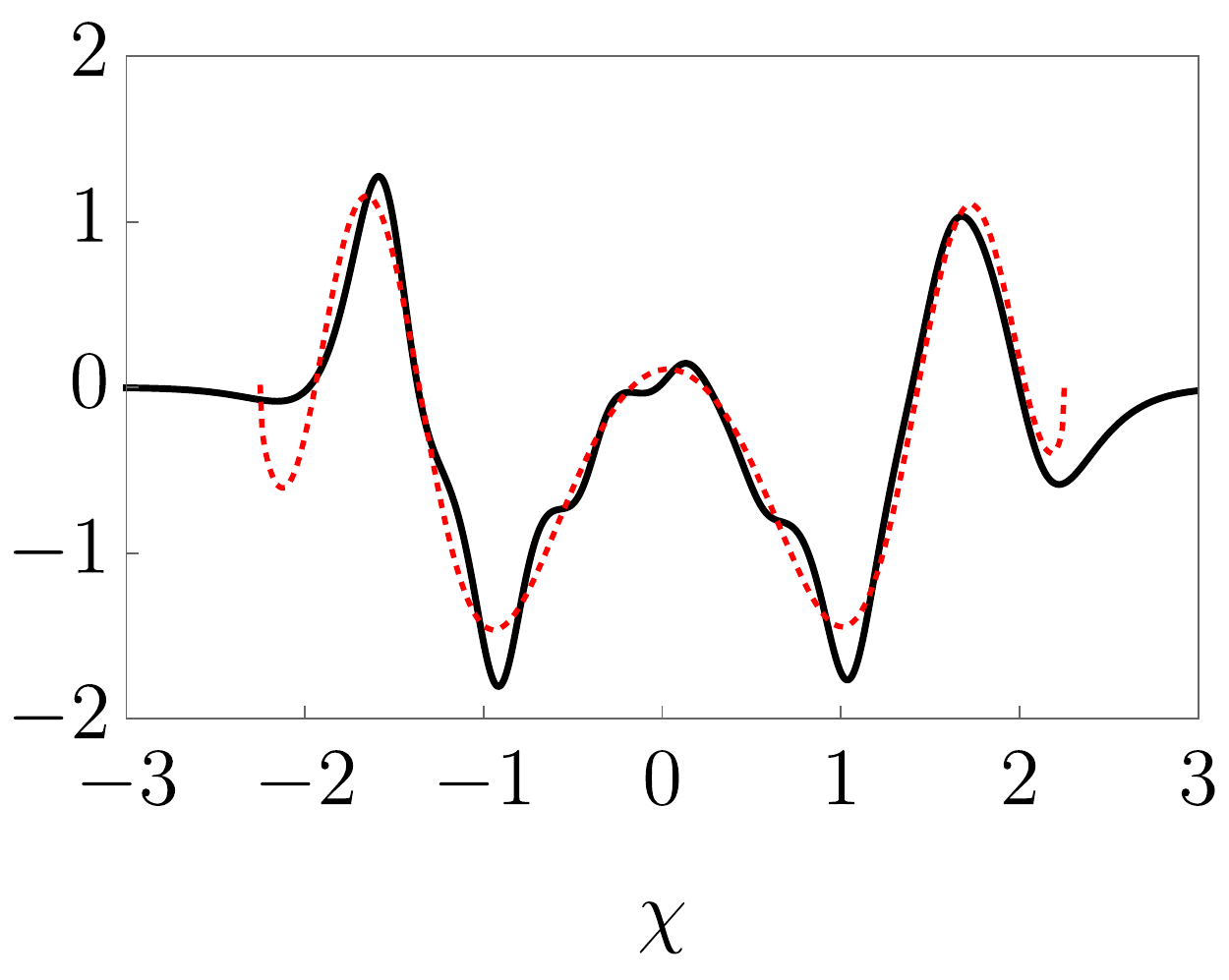}
\includegraphics[height=1.5in,width=2in]{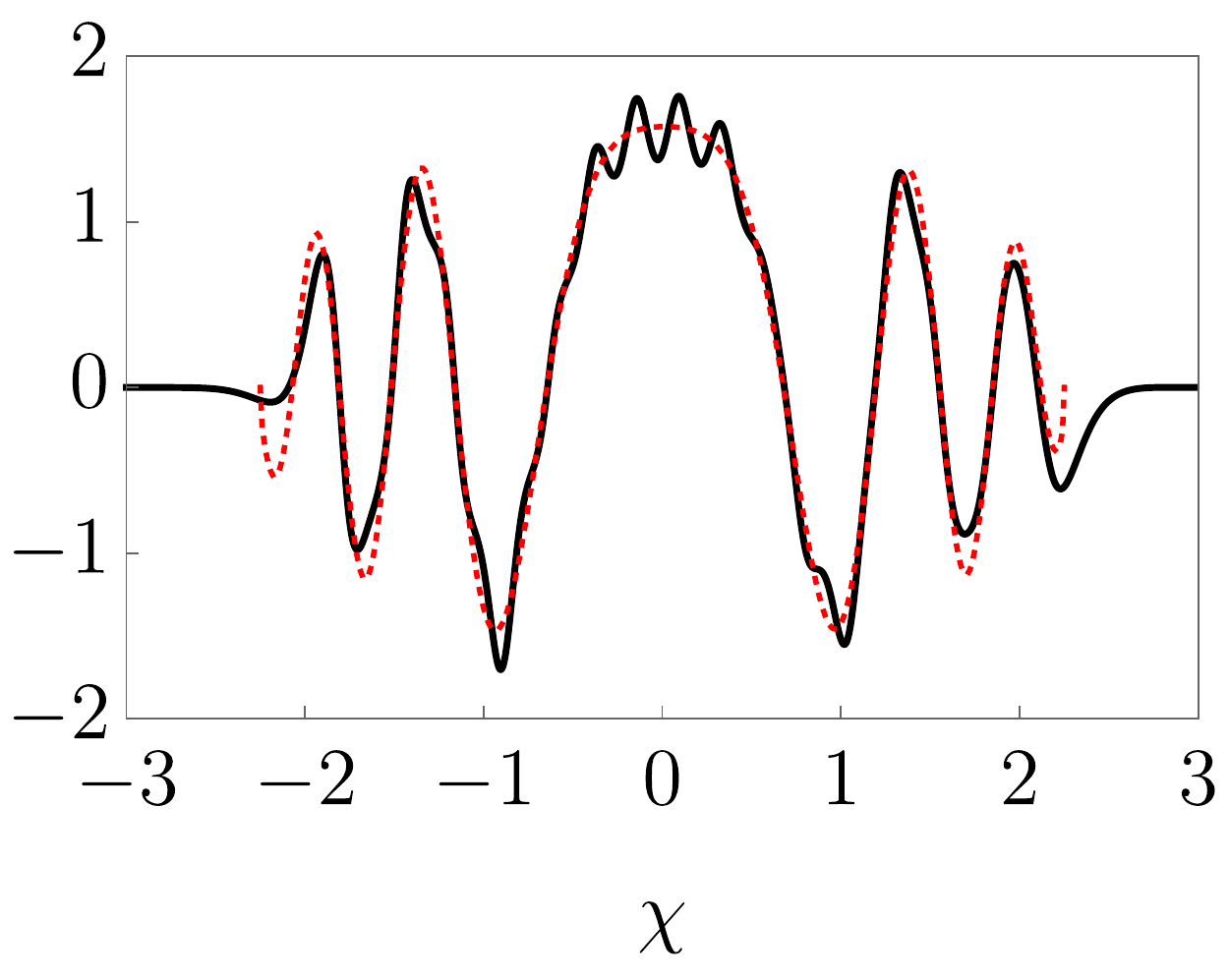}
\caption{Convergence of the leading-order asymptotic approximation in the 
non-oscillatory region for $\xi=i$ and ${\bf c}=(1,3)$ at 
$\tau=\frac{3\sqrt{3}}{8}$.  Solid black curves are for the exact solution 
$\psi^{[2n]}(n\chi,n\frac{3\sqrt{3}}{8};(1,3),i)$ while dashed red curves are 
for the leading-order approximation given by Theorem \ref{nonosc-thm}.  For 
this time slice the non-oscillatory region is exactly 
$-\frac{9}{4}\leq\chi\leq\frac{9}{4}$.  \emph{Left-to-right}:  $n=2$, $n=4$, 
$n=8$.  \emph{Top-to-bottom}:  The absolute value, real part, and imaginary 
part.}
\label{nonosc-sol-plots-c13}
\end{center}
\end{figure}
\begin{theorem}
\label{osc-thm}
{\rm (The oscillatory region).}  
{ Fix $\chi\geq 0$ and $\tau>0$ so that} $(\chi,\tau)$ is in the oscillatory region.  Define 
$a\equiv a(\chi,\tau)$ and $b\equiv b(\chi,\tau)$ by \eqref{a-b-def}, 
$F_1\equiv F_1(\chi,\tau)$ by \eqref{F1-def}, 
$F_0\equiv F_0(\chi,\tau)$ by \eqref{F0-def}, 
$A(\lambda)\equiv A(\lambda;\chi,\tau)$ by \eqref{Abel-map}, 
$B\equiv B(\chi,\tau)$ by \eqref{B-cycle}, 
$J\equiv J(\chi,\tau)$ by \eqref{J-def},
$U\equiv U(\chi,\tau)$ by \eqref{U-def},
and 
$Q\equiv Q(\chi,\tau)$ by \eqref{Q-def}.
Introduce the genus-one Riemann-theta function 
\eq
\Theta(\lambda) \equiv \Theta(\lambda;B) := \sum_{k\in \mathbb{Z}}e^{k\lambda+\frac{1}{2}Bk^2}.
\label{Riemann-theta}
\endeq
Then 
\eq
\begin{split}
\psi^{[2n]}(n\chi,n\tau) = \frac{\Theta(A(\infty)-A(Q)-i\pi-\frac{B}{2}+F_1U)\Theta(A(\infty)+A(Q)+i\pi+\frac{B}{2})}{\Theta(A(\infty)-A(Q)-i\pi-\frac{B}{2})\Theta(A(\infty)+A(Q)+i\pi+\frac{B}{2}-F_1U)} \hspace{.4in}& \\
  \times i\Im(b-a)e^{-2F_1J-2F_0} +\mathcal{O}\left(\frac{1}{n}\right),\quad { n\to +\infty}.&
\end{split}
\endeq
\end{theorem}
Theorem \ref{osc-thm} is proven in \S\ref{sec-osc}.  Figure
\ref{osc-sol-plots-c13} compares the exact solution to the 
leading-order behavior for various values of $n$.  
\begin{figure}
\begin{center}
\hspace{.2in} $n=2$ \hspace{1.4in} $n=4$ \hspace{1.6in} $n=8$ \\
\includegraphics[height=1.4in,width=2in]{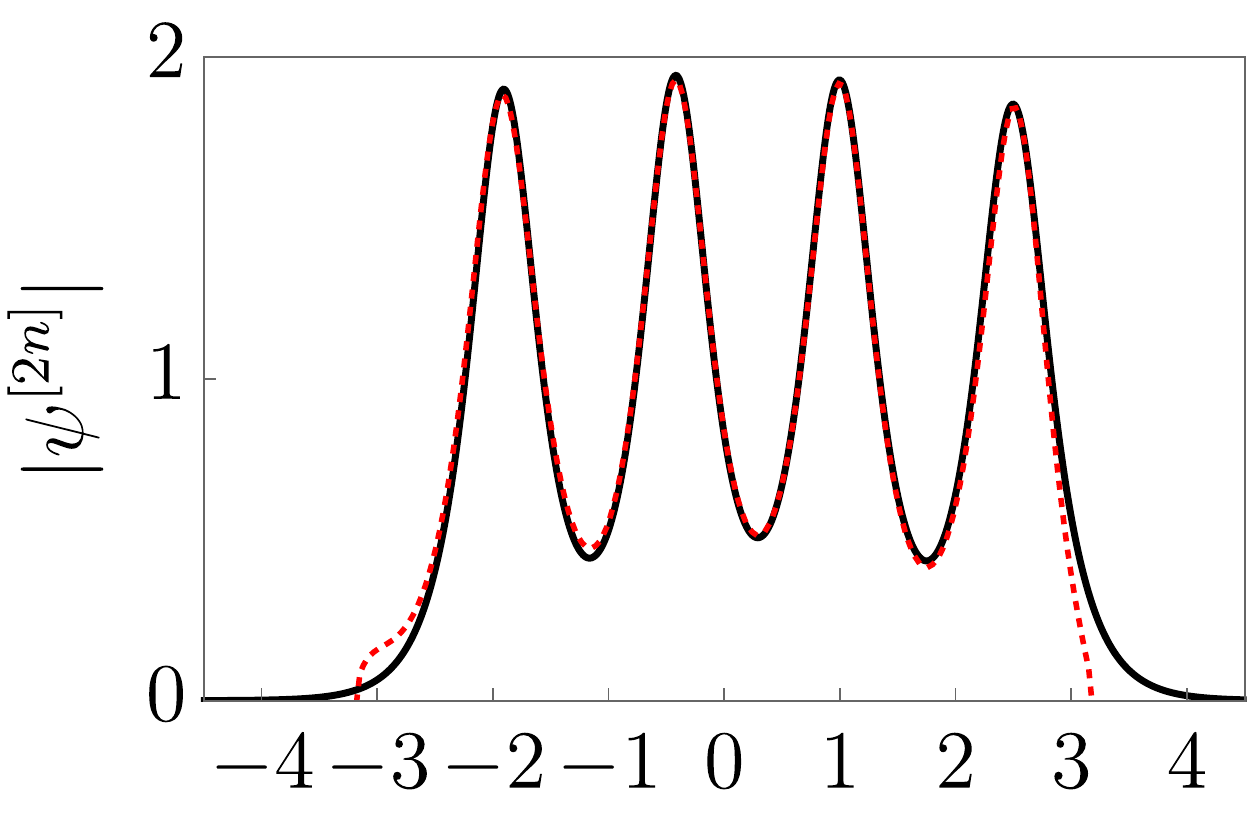}
\includegraphics[height=1.4in,width=2in]{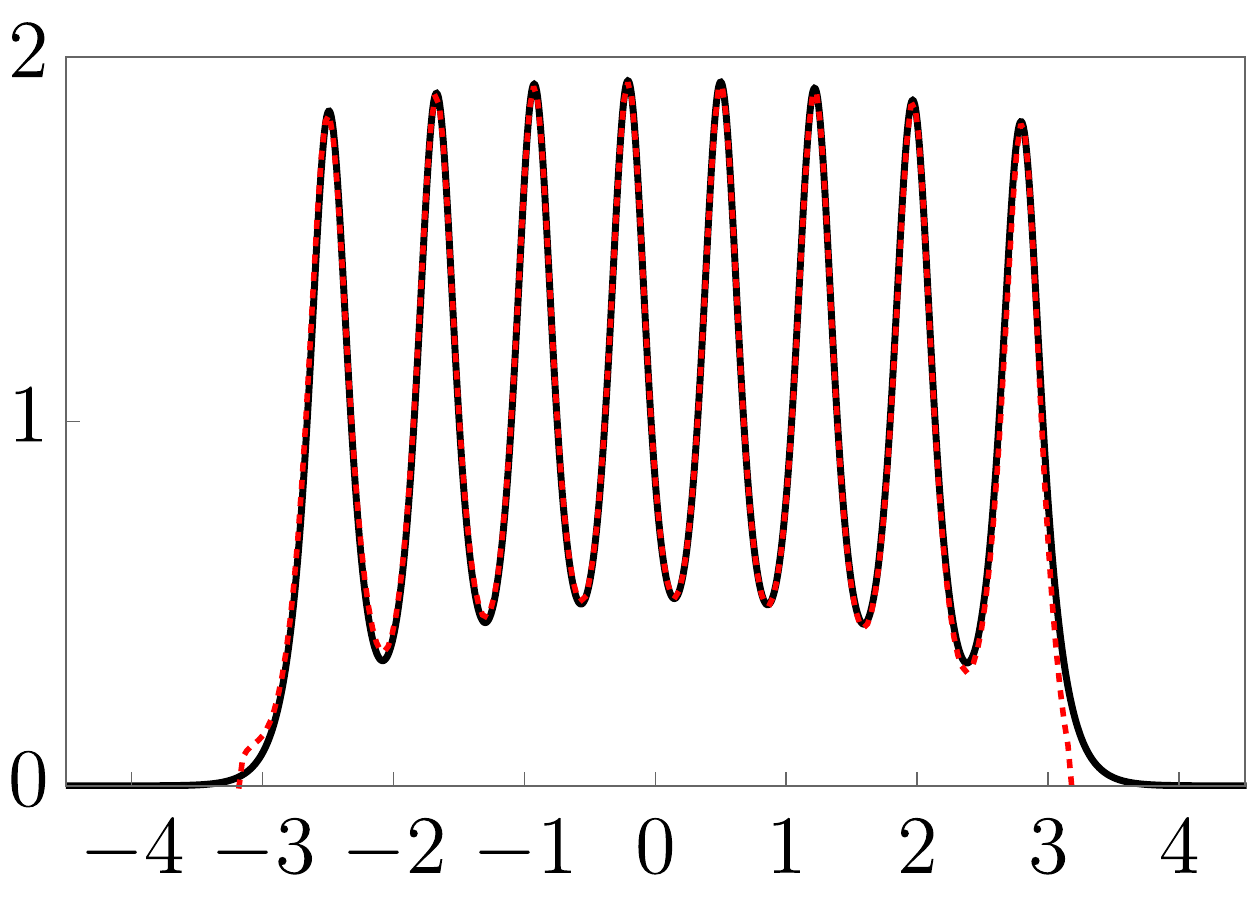}
\includegraphics[height=1.4in,width=2in]{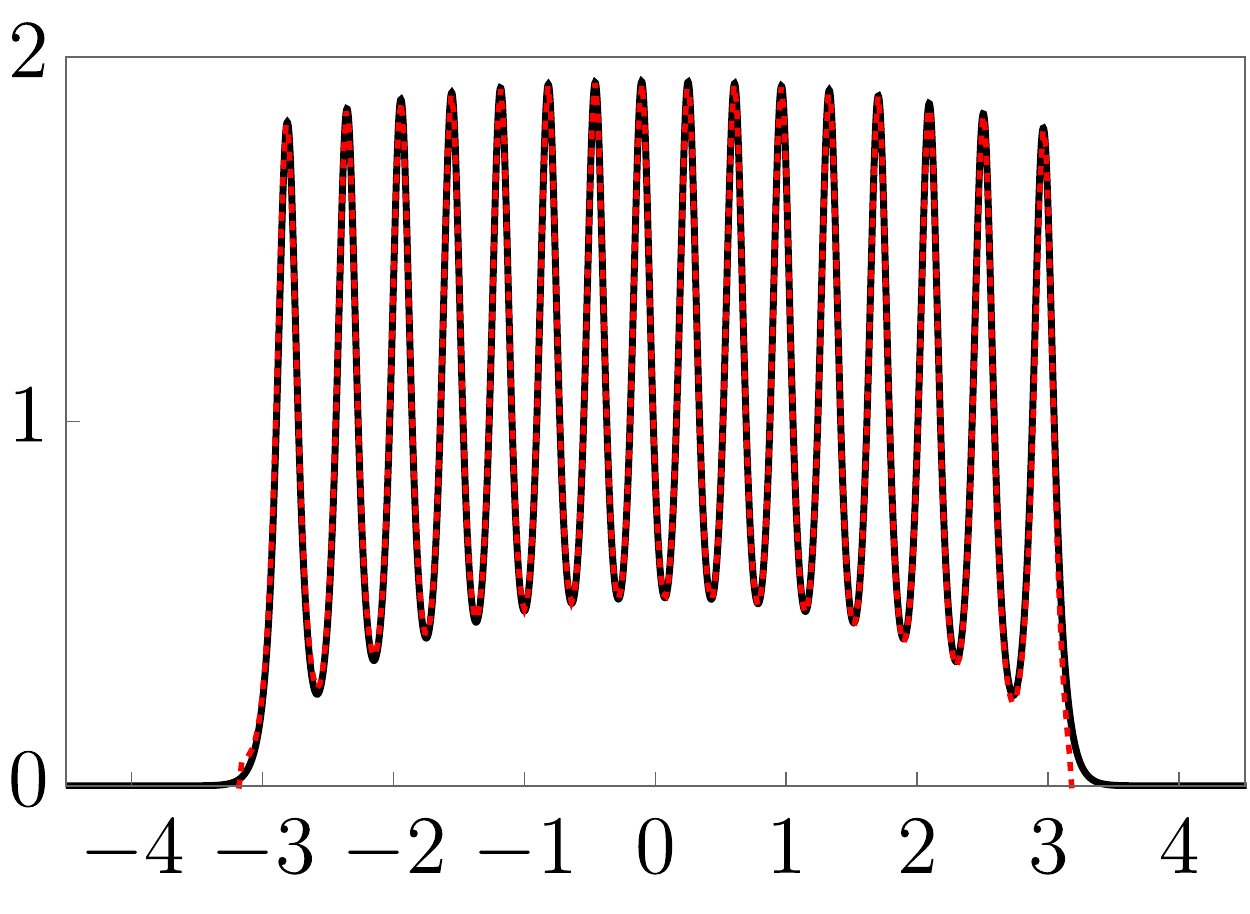}\\
\includegraphics[height=1.4in,width=2in]{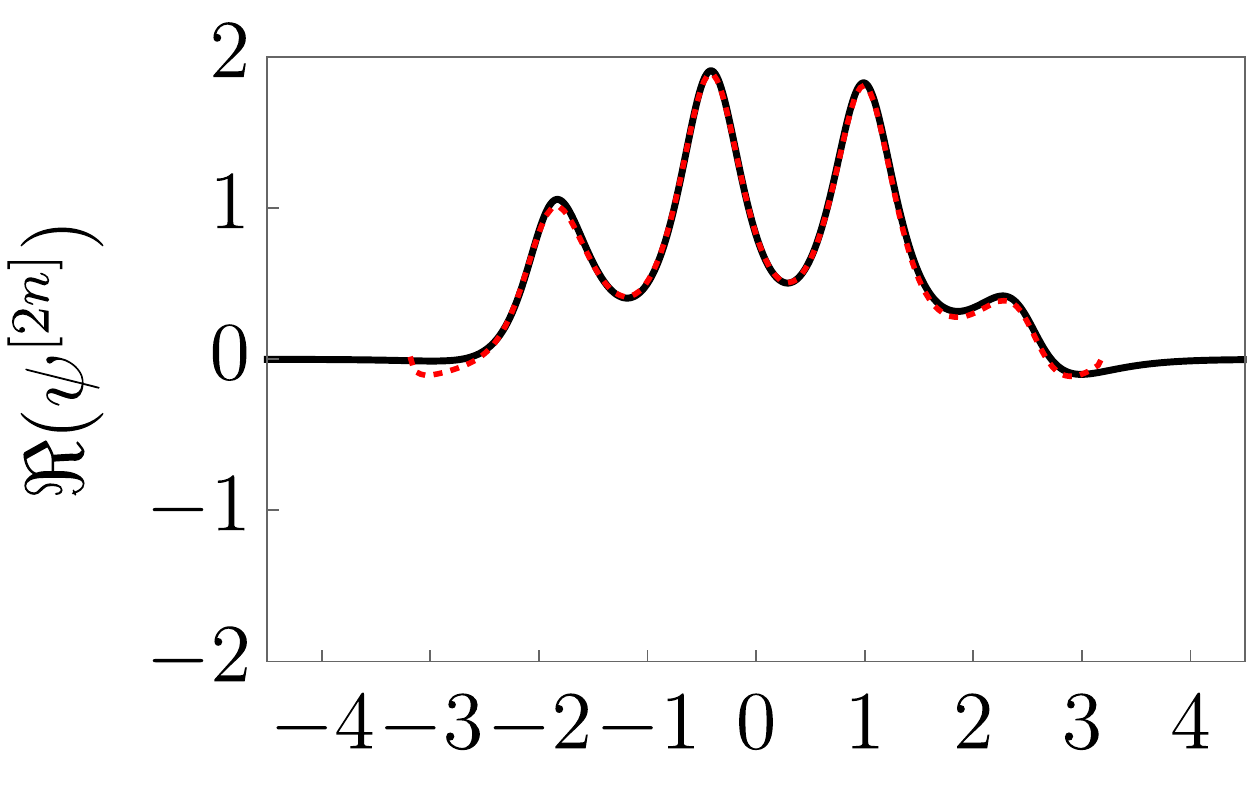}
\includegraphics[height=1.4in,width=2in]{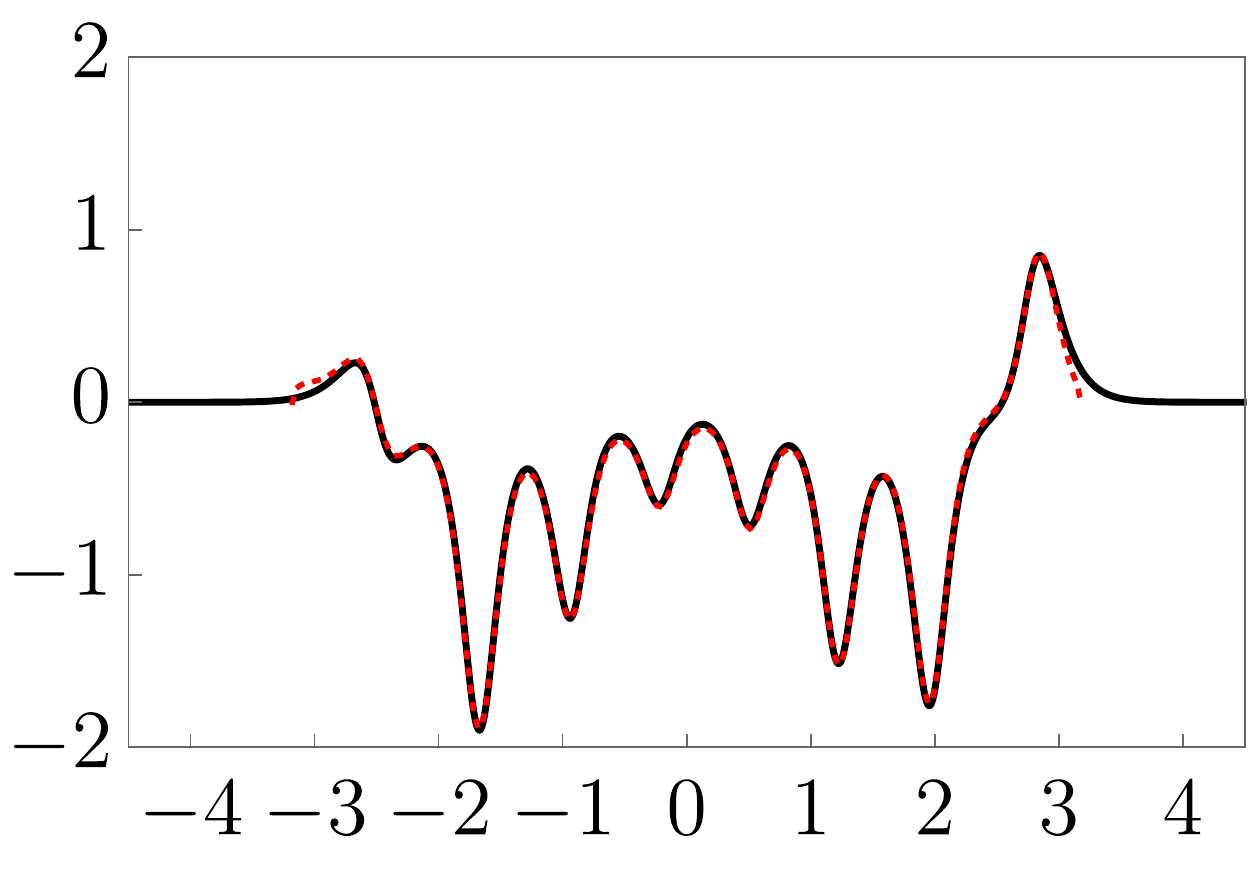}
\includegraphics[height=1.4in,width=2in]{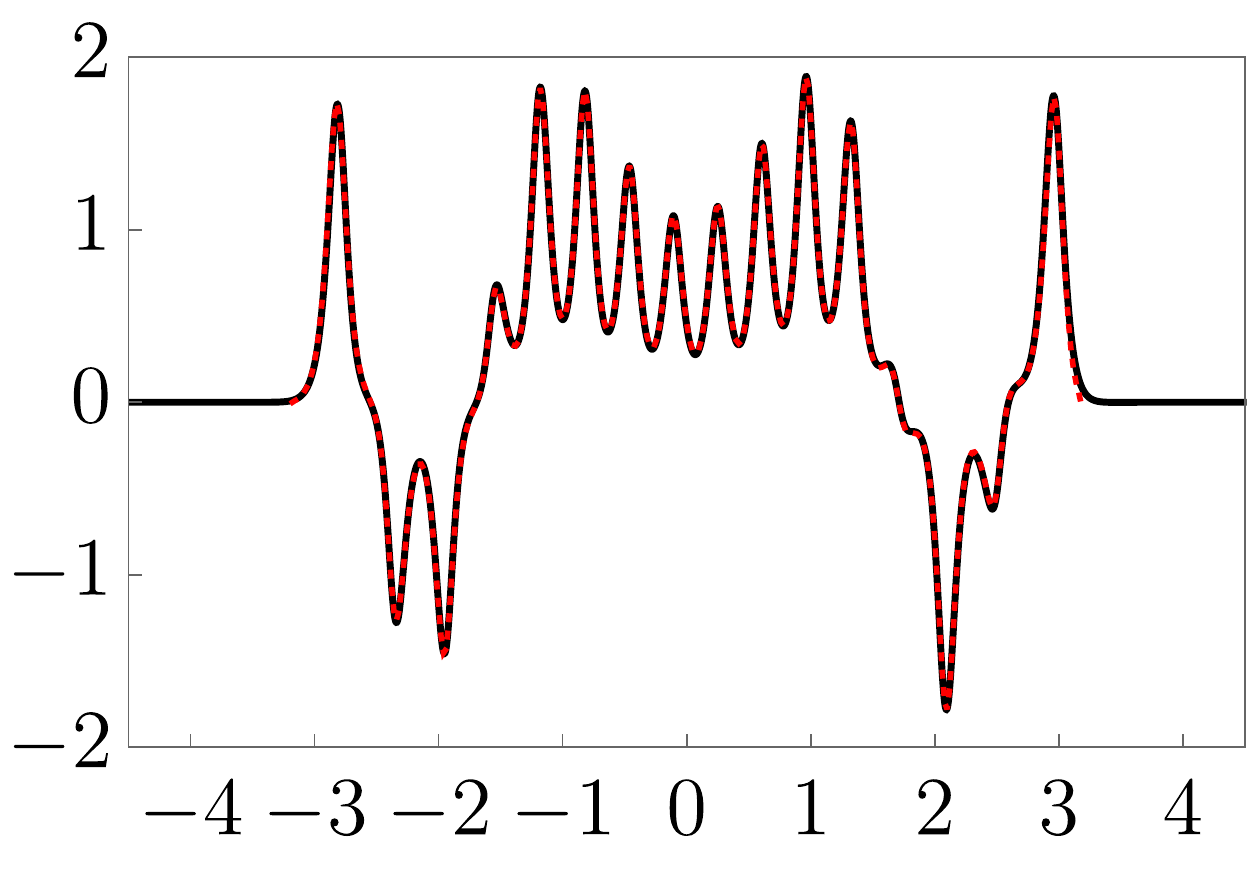}\\
\includegraphics[height=1.5in,width=2in]{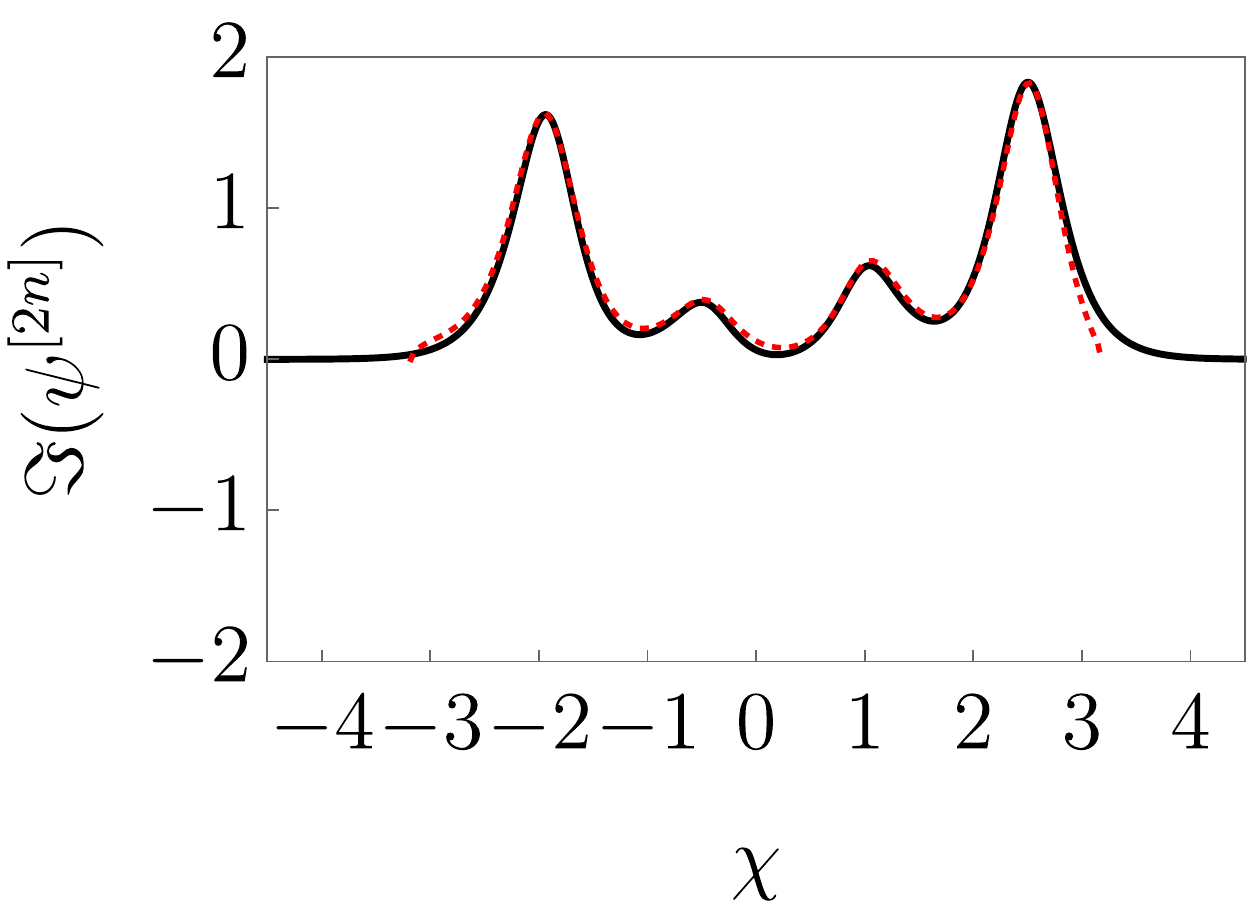}
\includegraphics[height=1.5in,width=2in]{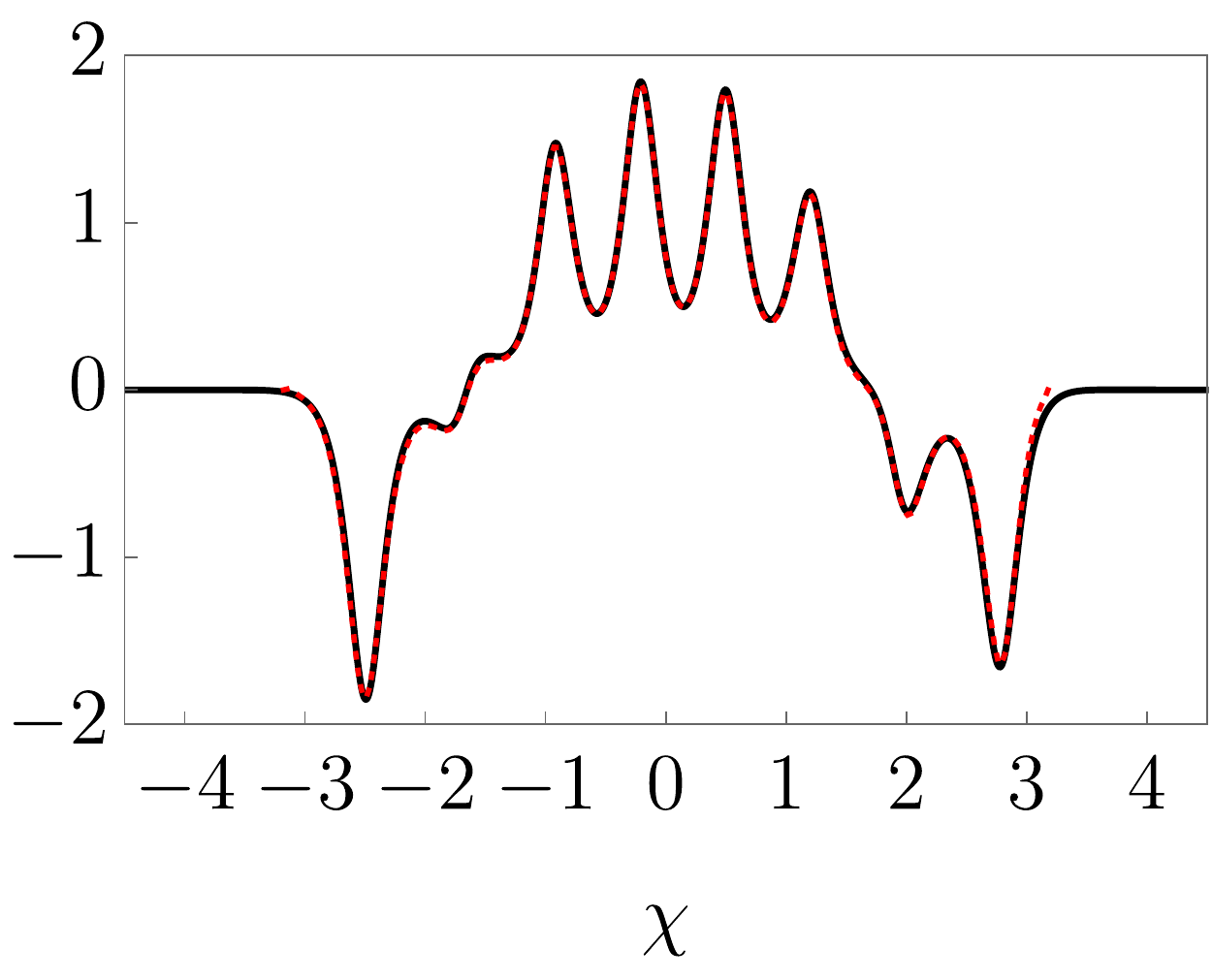}
\includegraphics[height=1.5in,width=2in]{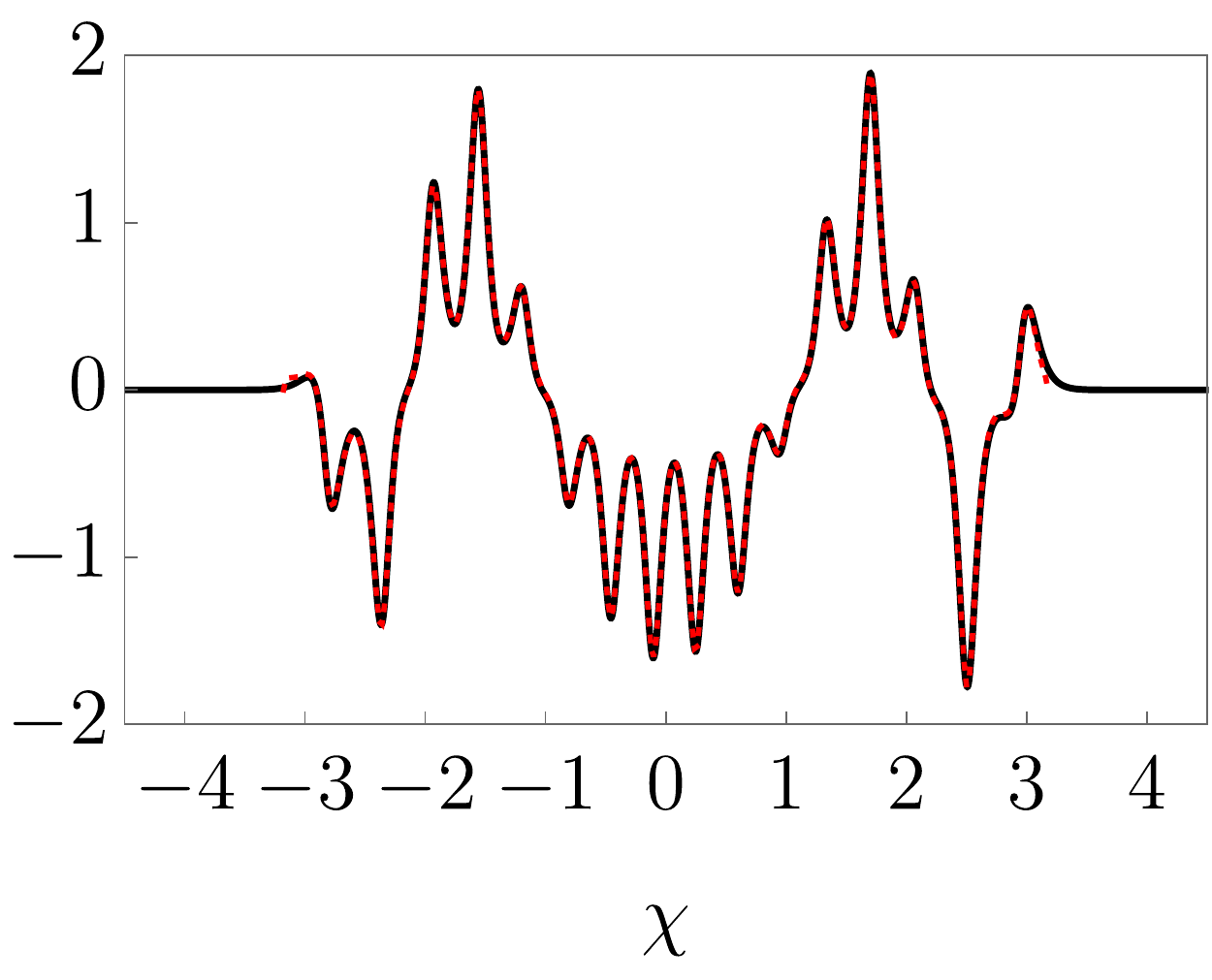}
\caption{Convergence of the leading-order asymptotic approximation in the 
oscillatory region for $\xi=i$ and ${\bf c}=(1,3)$ at $\tau=2$.  Solid black 
curves are for the exact solution 
$\psi^{[2n]}(n\chi,n{2};(1,3),i)$ 
while dashed red curves are for the leading-order 
approximation given by Theorem \ref{osc-thm}.  For this time slice the 
oscillatory region is approximately $-3.178<\chi<3.178$.
\emph{Left-to-right}:  $n=2$, $n=4$, $n=8$.  \emph{Top-to-bottom}:  The 
absolute value, real part, and imaginary part.}
\label{osc-sol-plots-c13}
\end{center}
\end{figure}

\subsection{The far-field Riemann-Hilbert problem}  We now introduce the basic 
Riemann-Hilbert problem used to define the {multiple-pole} solitons we study.  This 
representation was derived in \cite{BilmanB:2019} using the recently
introduced robust inverse-scattering transform \cite{BilmanM:2017}.
\begin{rhp}
\label{rhp-M}
{\rm (The unscaled Riemann-Hilbert problem)}.  Fix a pole location 
$\xi=\alpha+i\beta\in\mathbb{C}^+$, a vector of connection coefficients 
${\bf c}\equiv(c_1,c_2)\in(\mathbb{C}^*)^2$, and a non-negative integer $n$.  
Define $D_0\subset\mathbb{C}$ to be a circular disk centered at the origin 
containing $\xi$ in its interior.  Let $(x,t)\in\mathbb{R}^2$ be arbitrary 
parameters. Find the unique $2\times 2$ matrix-valued function 
$\mathbf{M}^{[n]}(\lambda;x,t)$ with the following properties:
\begin{itemize}
\item[]\textbf{Analyticity:} $\mathbf{M}^{[n]}(\lambda;x,t)$ is analytic for $\lambda\in\mathbb{C}\setminus \partial D_0$, and it takes continuous boundary values from the interior and exterior of $\partial D_0$.
\item[]\textbf{Jump condition:} The boundary values on the jump contour $\partial D_0$ (oriented clockwise) are related as
\begin{equation}
\mathbf{M}_{+}^{[n]}(\lambda;x,t) = \mathbf{M}_{-}^{[n]}(\lambda;x,t)e^{-i(\lambda x+\lambda^2 t)\sigma_3}\mathcal{S} \left( \frac{\lambda-\xi}{\lambda-\xi^*}\right)^{n\sigma_3} \mathcal{S}^{-1}e^{i(\lambda x+\lambda^2 t)\sigma_3},\quad\lambda\in\partial D_0,
\label{eq:jump-rhp-M}
\end{equation}
where
\eq
\mathcal{S}\equiv\mathcal{S}(c_1,c_2):=\frac{1}{|{\bf c}|}\bbm c_1 & -c_2^* \\ c_2 & c_1^* \ebm
\label{S-def}
\endeq
and $\sigma_3$ is the third Pauli matrix
\eq
\sigma_3 := \bbm 1 & 0 \\ 0 & -1 \ebm.
\endeq
\item[]\textbf{Normalization:} $\mathbf{M}^{[n]}(\lambda;x,t)=\mathbb{I}+\mathcal{O}(\lambda^{-1})$ as $\lambda\to\infty$.
\end{itemize}
\end{rhp}
\noindent
Given the solution $\mathbf{M}^{[n]}(\lambda;x,t)$, the function 
\eq
\label{psi-from-M}
\psi^{[2n]}(x,t;{\bf c},\xi) := 2i\lim_{\lambda\to\infty}\lambda[{\bf M}^{[n]}(\lambda;x,t;{\bf c},\xi)]_{12}
\endeq
is a $2n^\text{th}$-order pole soliton solution of \eqref{nls}. {We first present the following elementary symmetry properties of multiple-pole solitons of order $2n$.}
{
\begin{proposition} 
\label{prop:symmetries}
Let $\mathbf{c}=(c_1, c_2)\in \mathbb{C}^*$ and $\xi=\alpha + i \beta$ with $\alpha\in\mathbb{R}$ and $\beta>0$ be given. The multiple-pole solitons $\psi^{[2n]}(x,t; (c_1,c_2), \xi)$ enjoy the following symmetry properties:
\begin{align}
\psi^{[2n]}( - x, t; (c_1,c_2), \xi) &= \psi^{[2n]}( x,t; ( -c_2^*, -c_1^*), - \xi^*),\label{x-symmetry}\\
\psi^{[2n]}( x, -t; (c_1,c_2), \xi) &=  \psi^{[2n]}( x,t; ( c_1^*, c_2^*), - \xi^*)^*.\label{t-symmetry}
\end{align}
\end{proposition}
}
{A proof of based on the uniqueness of solutions of Riemann-Hilbert Problem~\ref{rhp-M} is given in Appendix~\ref{A:symmetries}.
}

We analyze Riemann-Hilbert Problem \ref{rhp-M} { in the large-$n$ regime} using the Deift-Zhou nonlinear 
steepest-descent method \cite{DeiftZ:1993}, which consists of making a series 
of invertible transformations in order to arrive at a problem that can be 
approximated in the large-$n$ limit.  The first transformation introduces the 
far-field scaling while simplifying the form of the jump matrix.  This 
Riemann-Hilbert problem for ${\bf N}^{[n]}(\lambda)$ will be our starting 
point for analysis in each of the far-field regions.  
{Define}
\eq
{\bf N}^{[n]}(\lambda;\chi,\tau) := \begin{cases} {\bf M}^{[n]}(\lambda;n\chi,n\tau)e^{-in(\lambda \chi+\lambda^2\tau)}\mathcal{S}e^{in(\lambda \chi+\lambda^2\tau)}, & \lambda\in D_0, \\ {\bf M}^{[n]}(\lambda;n\chi,n\tau)\left(\frac{\lambda-\xi^*}{\lambda-\xi}\right)^{n\sigma_3}, & \lambda\notin D_0. \end{cases}
\endeq
{As $\mathbf{N}^{[n]}(\lambda;\chi,\tau)$ is related to $\mathbf{M}^{[n]}(\lambda;n \chi,n\tau)$ outside $D_0$ via multiplication on the right by a diagonal matrix that tends to the identity matrix as $\lambda \to \infty$, the recovery formula remains unchanged:
\begin{equation}
\psi^{[2n]}(n\chi,n\tau;\mathbf{c},\xi) = 2i \lim _{\lambda \rightarrow \infty} \lambda\left[\mathbf{N}^{[n]}(\lambda ; \chi, \tau ; \mathbf{c},\xi)\right]_{12}.
\end{equation}
}
\begin{rhp}
\label{rhp-N}
{\rm (The far-field Riemann-Hilbert problem)}.  Fix a pole location 
$\xi=\alpha+i\beta\in\mathbb{C}^+$, a vector of connection 
coefficients ${\bf c}\equiv(c_1,c_2)\in(\mathbb{C}^*)^2$, and a non-negative 
integer $n$.  Define $D_0\subset\mathbb{C}$ to be a circular disk centered at 
the origin containing $\xi$ in its interior.  Let $(\chi,\tau)\in\mathbb{R}^2$ be 
arbitrary parameters. Find the unique $2\times 2$ matrix-valued function 
$\mathbf{N}^{[n]}(\lambda;\chi,\tau)$ with the following properties:
\begin{itemize}
\item[]\textbf{Analyticity:} $\mathbf{N}^{[n]}(\lambda;\chi,\tau)$ is analytic for 
$\lambda\in\mathbb{C}\setminus \partial D_0$, and it takes continuous boundary 
values from the interior and exterior of $\partial D_0$.
\item[]\textbf{Jump condition:} The boundary values on the jump contour 
$\partial D_0$ (oriented clockwise) are related as 
${\bf N}_+^{[n]}(\lambda;\chi,\tau)={\bf N}_-^{[n]}(\lambda;\chi,\tau){\bf V}_{\bf N}^{[n]}(\lambda;\chi,\tau)$, where 
\begin{equation}
{\bf V}_{\bf N}^{[n]}(\lambda;\chi,\tau) := e^{-n\varphi(\lambda;\chi,\tau)\sigma_3}\mathcal{S}^{-1}e^{n\varphi(\lambda;\chi,\tau)\sigma_3}.
\label{eq:jump-rhp-N}
\end{equation}
\item[]\textbf{Normalization:} $\mathbf{N}^{[n]}(\lambda;\chi,\tau)=\mathbb{I}+\mathcal{O}(\lambda^{-1})$ as $\lambda\to\infty$.
\end{itemize}
\end{rhp}
{ With Proposition~\ref{prop:symmetries} at hand, we restrict our attention to the first quadrant of the $(x,t)$-plane, hence that of the $(\chi,\tau)$-plane, for the remainder of this paper.}

\section{The algebraic-decay region}
\label{sec-alg-decay}

Pick $(\chi,\tau)$ in the algebraic-decay region.  Our first objective is to 
understand the signature chart of $\Re(\varphi(\lambda;\chi,\tau))$.  
\begin{lemma}
\label{algebraic-lemma}
In the algebraic-decay region, there is a domain $D_{\rm up}$ in the upper half-plane 
with the following properties:
\begin{itemize}
\item $D_{\rm up}$ contains $\xi$, is bounded by curves along which 
$\Re(\varphi(\lambda))=0$, and abuts the real axis along a single interval 
(denoted $(\lambda^{(1)},\lambda^{(2)})$).  
\item $\Re(\varphi(\lambda))>0$ for all $\lambda\in D_{\rm up}$.
\item $\Re(\varphi(\lambda))<0$ for all $\lambda$ in the upper half-plane 
in the complement of $\overline{D_{\rm up}}$ but sufficiently close to $D_{\rm up}$.  
\end{itemize}
Similarly, there is a domain $D_{\rm down}$ in the lower half-plane such that:
\begin{itemize}
\item $D_{\rm down}$ contains $\xi^*$, is bounded by curves along which 
$\Re(\varphi(\lambda))=0$, and abuts the real axis along the same interval as 
$D$.  
\item $\Re(\varphi(\lambda))<0$ for all $\lambda\in D_{\rm down}$.
\item $\Re(\varphi(\lambda))>0$ for all $\lambda$ in the lower half-plane 
in the complement of $\overline{D_{\rm down}}$ but sufficiently close to 
$D_{\rm down}$.  
\end{itemize}
\end{lemma}
\begin{proof}
It is 
instructive to compare with the signature chart in the exponential-decay region.  
In \cite{BilmanB:2019} it was proven that in the exponential-decay region there 
is a closed loop in the $\lambda$-plane surrounding $\xi$ on which 
$\Re(\varphi(\lambda))=0$.  Inside this curve $\Re(\varphi(\lambda))>0$, while 
outside the curve for $\lambda$ sufficiently close to the curve 
$\Re(\varphi(\lambda))<0$.  In the lower half-plane the signature chart is 
symmetric with the signs flipped.  If $\tau=0$ there are two critical points 
$\lambda^+$ and $\lambda^-$ that are complex conjugates;  if $\tau\neq 0$ there is 
an additional real critical point $\lambda^{(0)}$.  See Figure 
\ref{exponential-phase-plots}.
\begin{figure}
\begin{center}
\includegraphics[height=2in]{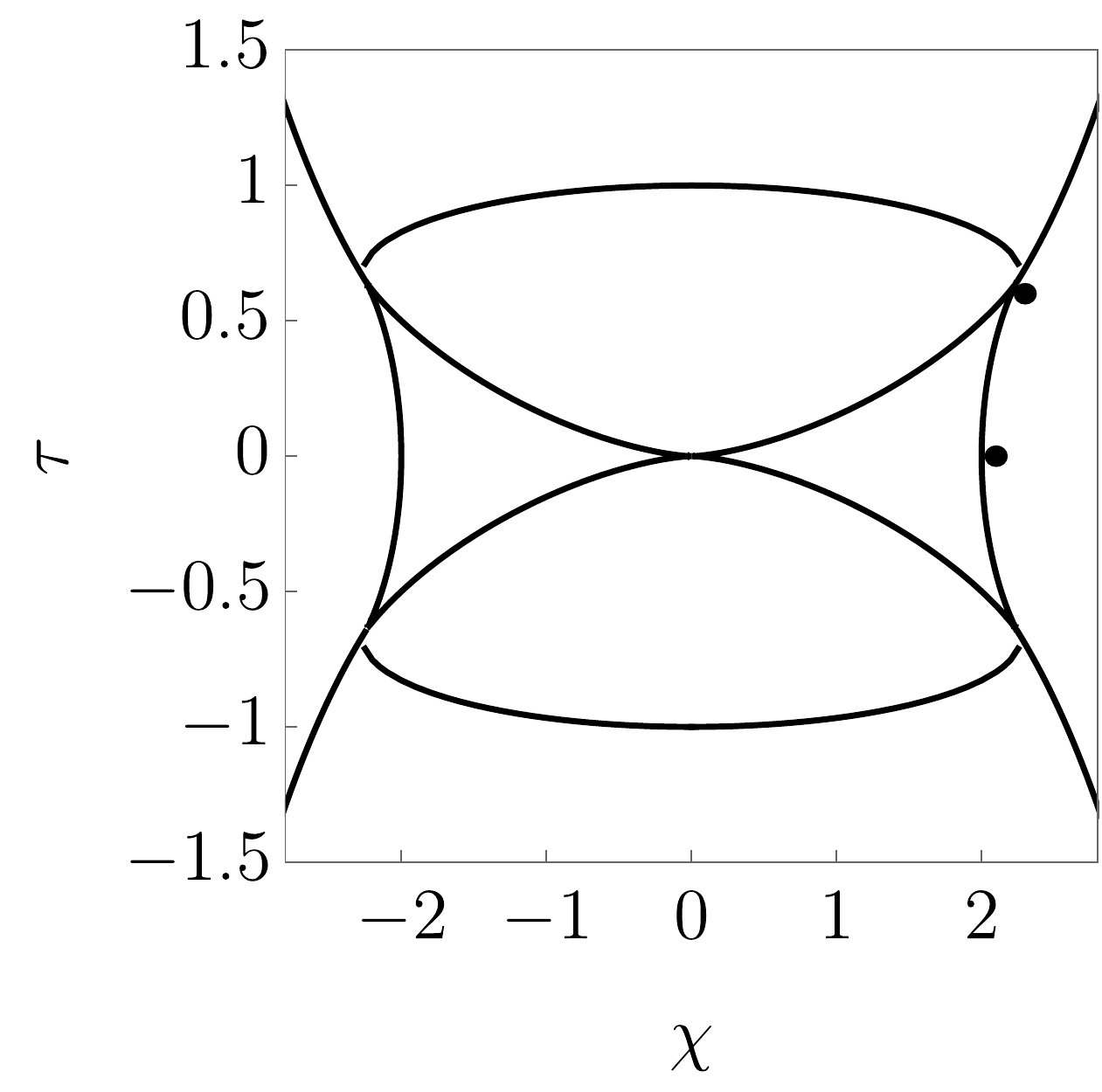} 
\includegraphics[height=2in]{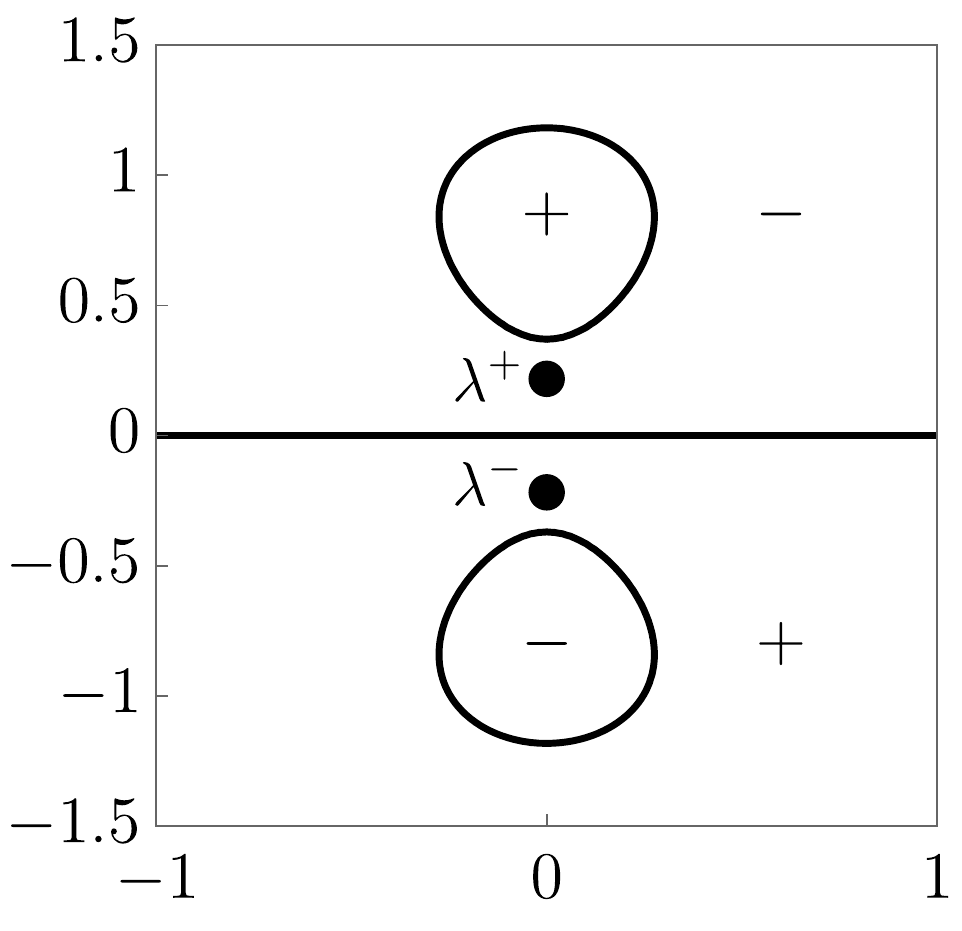} 
\includegraphics[height=2in]{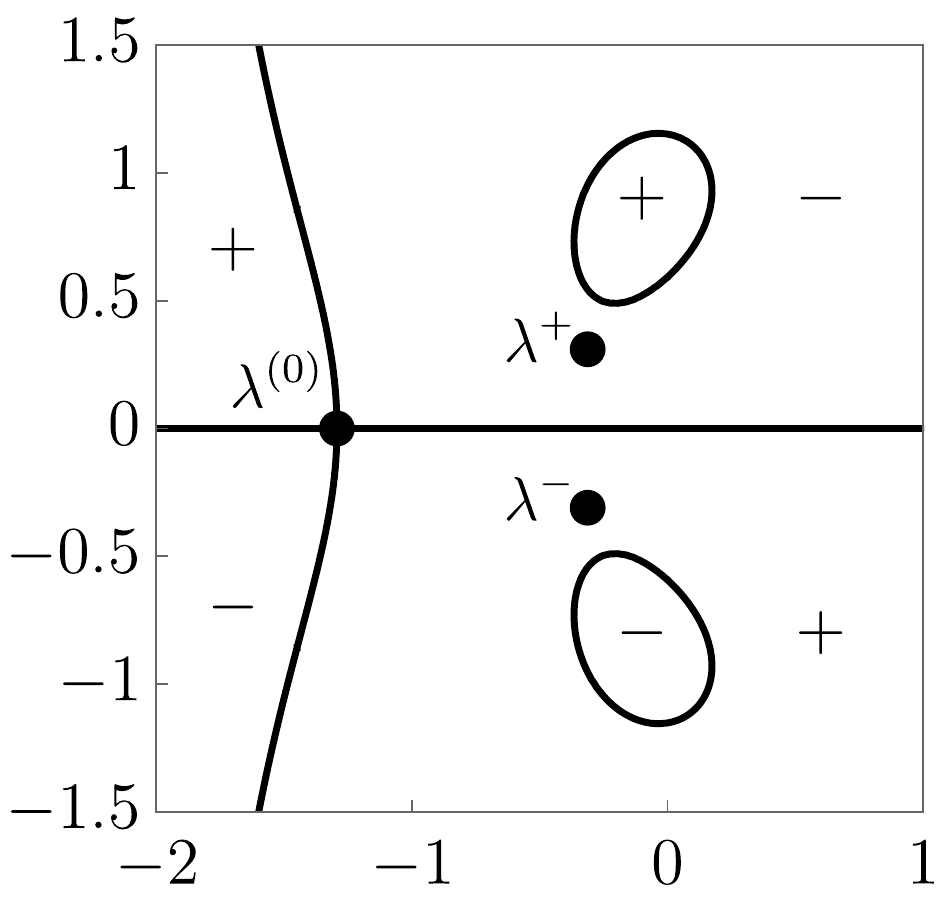} 
\caption{Signature charts of $\Re(\varphi(\lambda;\chi,\tau))$ for $\xi=i$ in the 
exponential-decay region, along with the critical points $\lambda^+$ and 
$\lambda^-$ (and, when it exists, $\lambda^{(0)}$).  
\emph{Left}:  Positions in the ($\chi$,$\tau$)-plane relative to the boundary 
curves.  \emph{Center}:  $\chi=2.1$, $\tau=0$.  \emph{Right}: $\chi=2.3$, 
$\tau=0.6$.}
\label{exponential-phase-plots}
\end{center}
\end{figure}
\begin{figure}
\begin{center}
\includegraphics[height=2in]{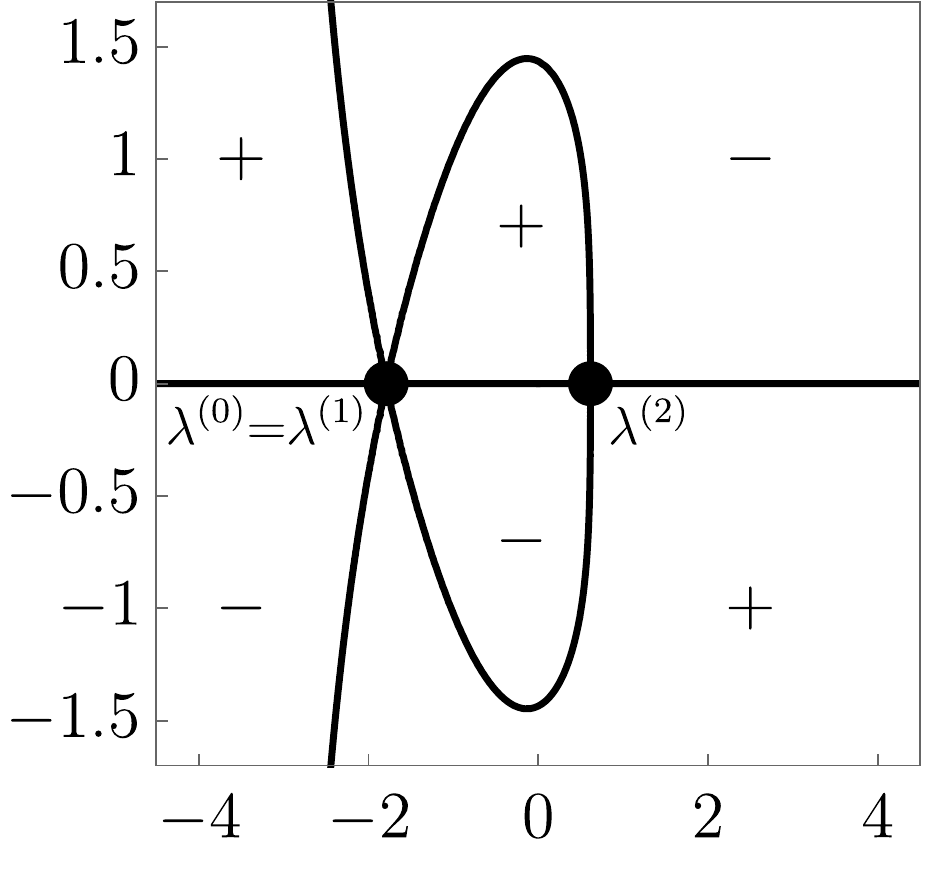}
\includegraphics[height=2in]{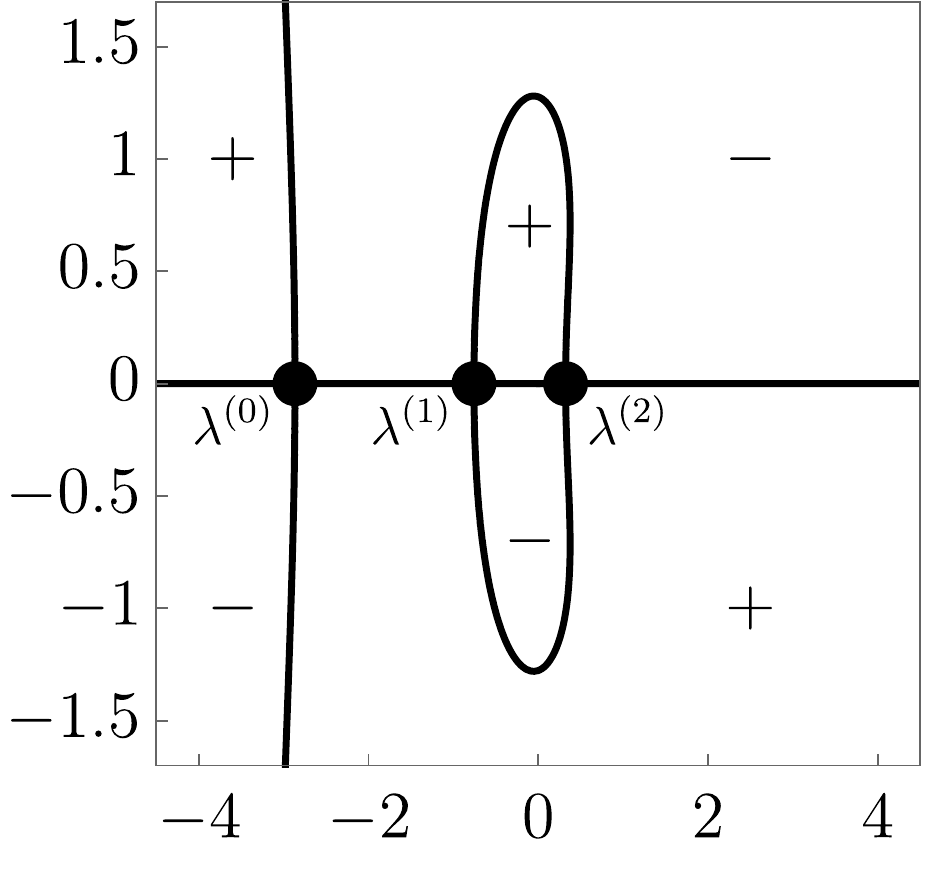}
\includegraphics[height=2in]{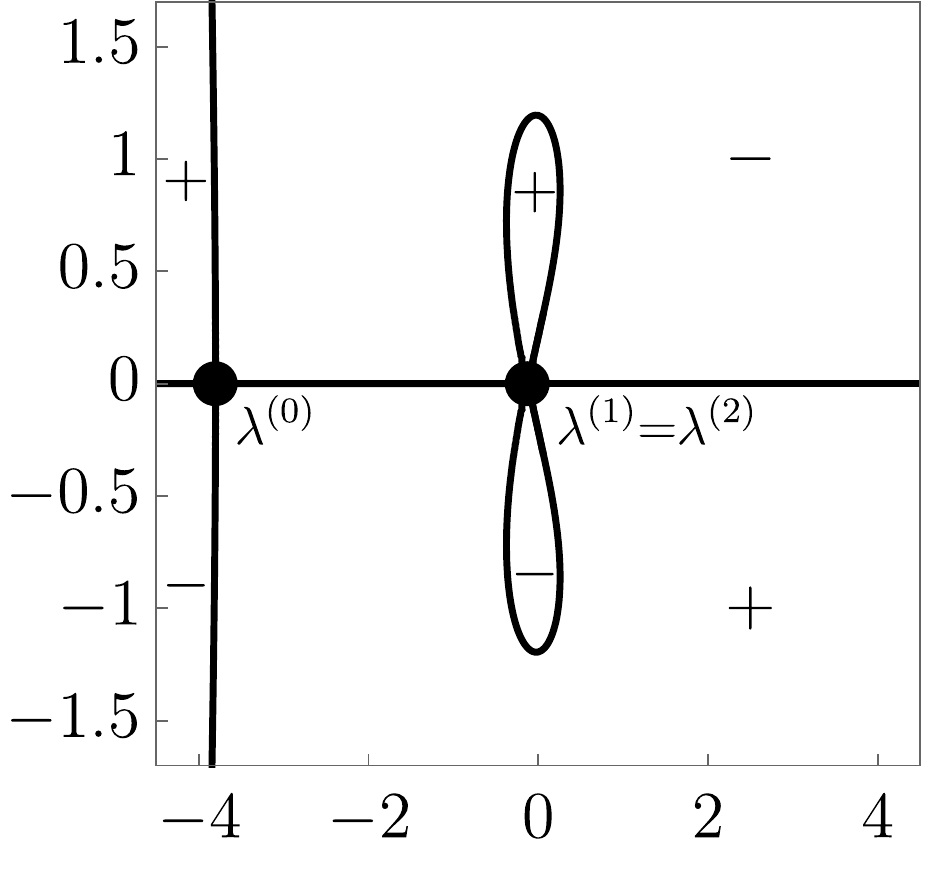} \\
\includegraphics[height=2.1in]{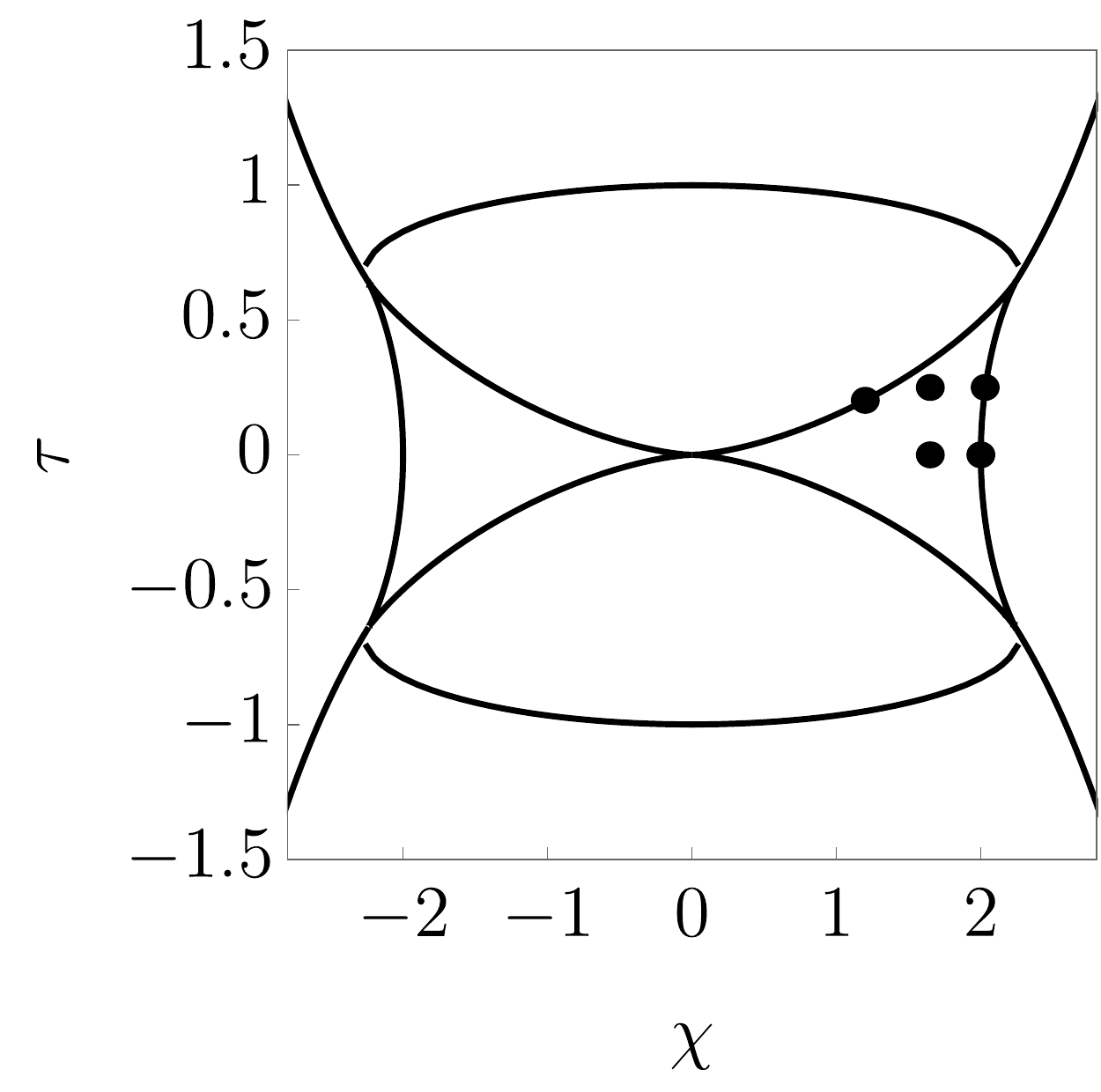} 
\includegraphics[height=2in]{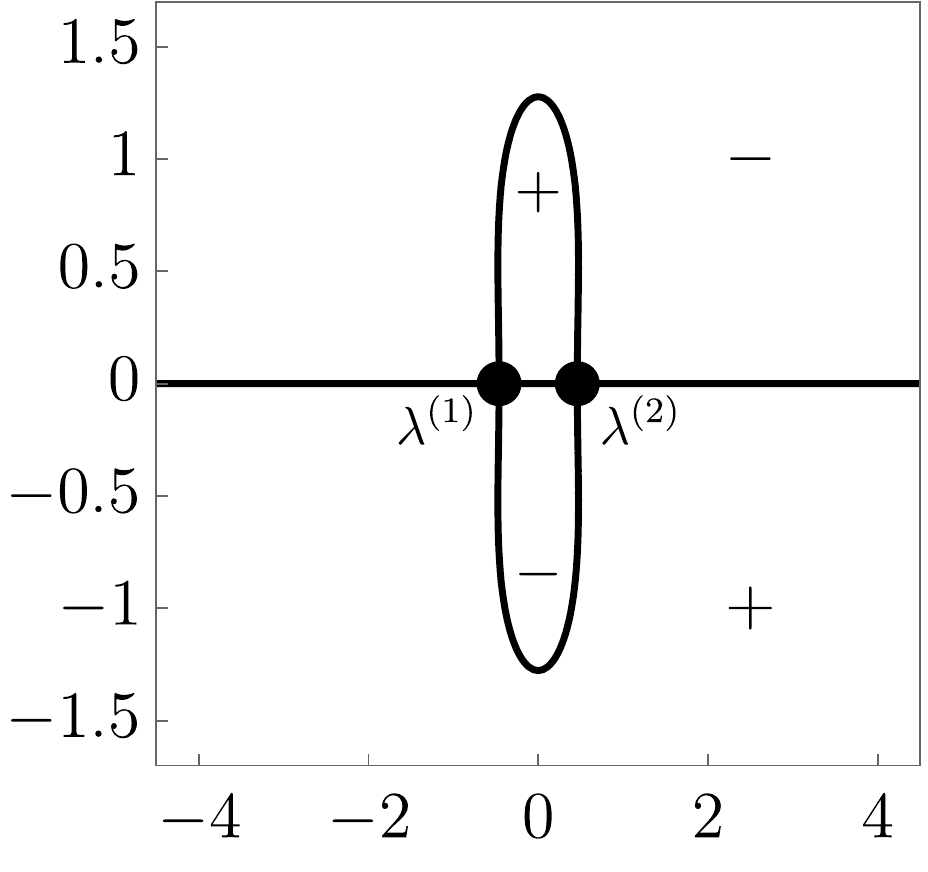} 
\includegraphics[height=2in]{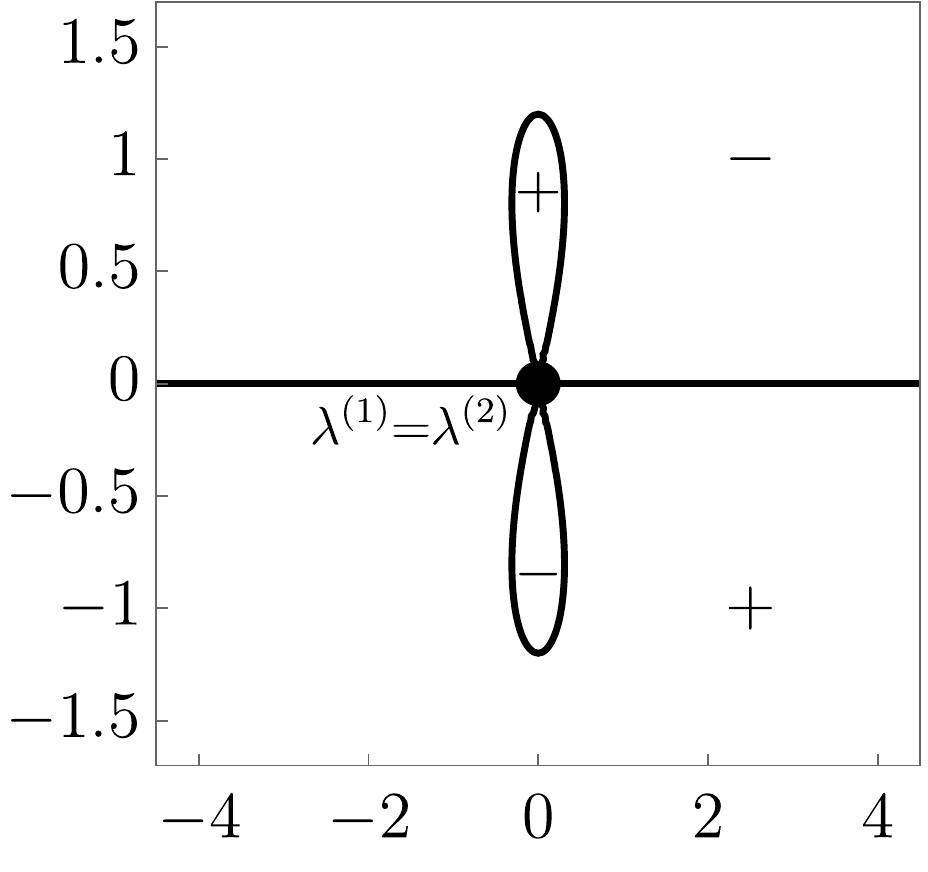} 
\caption{Signature charts of $\Re(\varphi(\lambda;\chi,\tau))$ for $\xi=i$ in the 
algebraic-decay region, along with the critical points $\lambda^{(1)}$ and 
$\lambda^{(2)}$ (and, when it exists, $\lambda^{(0)}$).  \emph{Top left}:  
$\chi=1.2$, $\tau\approx 0.2023$.  \emph{Top middle}: $\chi=1.65$, $\tau=0.25$.  
\emph{Top right}: $\chi\approx 2.03$, $\tau=0.25$.  \emph{Bottom left}:  
Positions in the ($\chi$,$\tau$)-plane relative to the boundary curves.  
\emph{Bottom middle}:  $\chi=1.65$, $\tau=0$.  \emph{Bottom right}: $\chi=2$, 
$\tau=0$.}
\label{algebraic-phase-plots}
\end{center}
\end{figure}
Passing from the exponential-decay region to the algebraic-decay region, the 
boundary curve $\mathcal{L}_\text{AE}$ is marked by the condition 
$\lambda^+=\lambda^-$.  When these two critical points coincide they are real, 
and thus lie on a zero-level curve of $\Re(\varphi(\lambda))$.  This means that 
the two closed curves surrounding $\xi$ and $\xi^*$ along which 
$\Re(\varphi(\lambda))=0$ must intersect at $\lambda^+=\lambda^-$ for 
$(\chi,\tau)$ on $\mathcal{L}_\text{AE}$.  In the notation used in the 
algebraic-decay region the double critical point is 
$\lambda^{(1)}=\lambda^{(2)}$.  See the top right and bottom right panels in 
Figure \ref{algebraic-phase-plots}.  

Now, as $(\chi,\tau)$ moves into the algebraic-decay region from 
$\mathcal{L}_\text{AE}$, the double critical point splits into the two real 
critical points $\lambda^{(1)}$ and $\lambda^{(2)}$.  By definition, no 
critical points coincide inside the algebraic-decay region.  In particular, 
this means that in the algebraic-decay region there is a domain $D_{\rm up}$ in the 
upper half-plane that contains $\xi$, abuts the real axis along the interval 
$(\lambda^{(1)},\lambda^{(2)})$, and is bounded by curves along which 
$\Re(\varphi(\lambda))=0$.  Furthermore, $\Re(\varphi(\lambda))>0$ for all 
$\lambda\in D_{\rm up}$, and $\Re(\varphi(\lambda))<0$ for all $\lambda$ in the 
upper half-plane sufficiently close to $D_{\rm up}$.  There is an analogous domain 
$D_{\rm down}$ in the lower half-plane containing $\xi^*$ such that 
$\Re(\varphi(\lambda))<0$ for all $\lambda\in D_{\rm down}$, and 
$\Re(\varphi(\lambda))>0$ for all $\lambda$ in the lower half-plane sufficiently 
close to $D_{\rm down}$.  See the top middle and bottom middle panels in Figure 
\ref{algebraic-phase-plots}.  
\end{proof}

Define the domain $D$ to be the union of $D_{\rm up}$, $D_{\rm down}$, and the interval 
$(\lambda^{(1)},\lambda^{(2)})$, so that $\partial D$ is a simple Jordan curve 
passing through $\lambda^{(1)}$ and $\lambda^{(2)}$ along which 
$\Re(\varphi(\lambda))=0$.  We write $\Gamma_\text{up}$ for the portion of 
$\partial D$ in the upper half-plane and $\Gamma_\text{down}$ for the portion of 
$\partial D$ in the lower half-plane.  See Figure \ref{algebraic-lenses}.  
We are now ready to carry out our first 
Riemann-Hilbert transformation, which will deform the jump contour from 
$\partial D_0$ to $\Gamma_\text{up}\cup\Gamma_\text{down}$.  Set 
\eq
{\bf O}^{[n]}(\lambda;\chi,\tau):= \begin{cases} {\bf N}^{[n]}(\lambda;\chi,\tau){\bf V}_{\bf N}^{[n]}(\lambda;\chi,\tau), & \lambda\in D_0\cap D^\mathsf{c}, \\ {\bf N}^{[n]}(\lambda;\chi,\tau){\bf V}_{\bf N}^{[n]}(\lambda;\chi,\tau)^{-1}, & \lambda\in D_0^\mathsf{c}\cap D, \\ {\bf N}^{[n]}(\lambda;\chi,\tau), & \text{otherwise}.
\end{cases}
\endeq
Then, orienting $\Gamma_\text{up}\cup\Gamma_\text{down}$ clockwise, the function 
${\bf O}^{[n]}(\lambda)$ satisfies exactly the same Riemann-Hilbert 
problem as ${\bf N}^{[n]}(\lambda)$ with $\partial D_0$ replaced by 
$\Gamma_\text{up}\cup\Gamma_\text{down}$.  Note that the matrix $\mathcal{S}^{-1}$ 
has the following two factorizations:
\eq
\begin{split}
\label{Sinv-on-Gamma}
\mathcal{S}^{-1} & = \bbm 1 & \frac{c_2^*}{c_1} \\ 0 & 1 \ebm \bbm \frac{|{\bf c}|}{c_1} & 0 \\ 0 & \frac{c_1}{|{\bf c}|} \ebm \bbm 1 & 0 \\ -\frac{c_2}{c_1} & 1 \ebm \quad\quad (\text{use for }\lambda\in\Gamma_\text{up}), \\
\mathcal{S}^{-1} & = \bbm 1 & 0 \\ -\frac{c_2}{c_1^*} & 1 \ebm \bbm \frac{c_1^*}{|{\bf c}|} & 0 \\ 0 & \frac{|{\bf c}|}{c_1^*} \ebm \bbm 1 & \frac{c_2^*}{c_1^*} \\ 0 & 1 \ebm \quad\quad (\text{use for }\lambda\in\Gamma_\text{down}).
\end{split}
\endeq
Following the exponential-decay region analysis in \cite{BilmanB:2019}, we 
define the following four contours:
\begin{itemize}
\item $\Gamma_\text{up}^\text{out}$ runs from $\lambda^{(1)}$ to $\lambda^{(2)}$ in 
the upper half-plane entirely in the region where $\Re(\varphi(\lambda))<0$.
\item $\Gamma_\text{up}^\text{in}$ runs from $\lambda^{(1)}$ to $\lambda^{(2)}$ 
entirely in $D_{\rm up}$ (so $\Re(\varphi(\lambda))>0$), and can be deformed to 
$\Gamma_\text{up}$ without passing through $\xi$.
\item $\Gamma_\text{down}^\text{out}$ runs from $\lambda^{(2)}$ to $\lambda^{(1)}$ 
in the lower half-plane entirely in the region where $\Re(\varphi(\lambda))>0$.
\item $\Gamma_\text{down}^\text{in}$ runs from $\lambda^{(1)}$ to $\lambda^{(2)}$ 
entirely in $D_{\rm down}$ (so $\Re(\varphi(\lambda))<0$), and can be deformed to 
$\Gamma_\text{down}$ without passing through $\xi^*$.
\end{itemize}
We also write 
\eq
\Gamma_\text{lens} := \Gamma_\text{up}^\text{out} \cup \Gamma_\text{up}^\text{in} \cup \Gamma_\text{down}^\text{out} \cup \Gamma_\text{down}^\text{in} \quad \text{and} \quad \Gamma:= \Gamma_\text{up} \cup \Gamma_\text{down} \cup \Gamma_\text{lens}.
\endeq
We next define the following four domains:
\begin{itemize}
\item $L_\text{up}^\text{out}$ is the domain in the upper half-plane bounded by 
$\Gamma_\text{up}^\text{out}$ and $\partial D$.  
\item $L_\text{up}^\text{in}$ is the domain in the upper half-plane bounded by 
$\Gamma_\text{up}^\text{in}$ and $\partial D$.  
\item $L_\text{down}^\text{out}$ is the domain in the lower half-plane bounded by 
$\Gamma_\text{down}^\text{out}$ and $\partial D$.  
\item $L_\text{down}^\text{in}$ is the domain in the lower half-plane bounded by 
$\Gamma_\text{down}^\text{in}$ and $\partial D$.  
\end{itemize}
See Figure \ref{algebraic-lenses}.  
\begin{figure}[h]
\begin{center}
\includegraphics[height=2in]{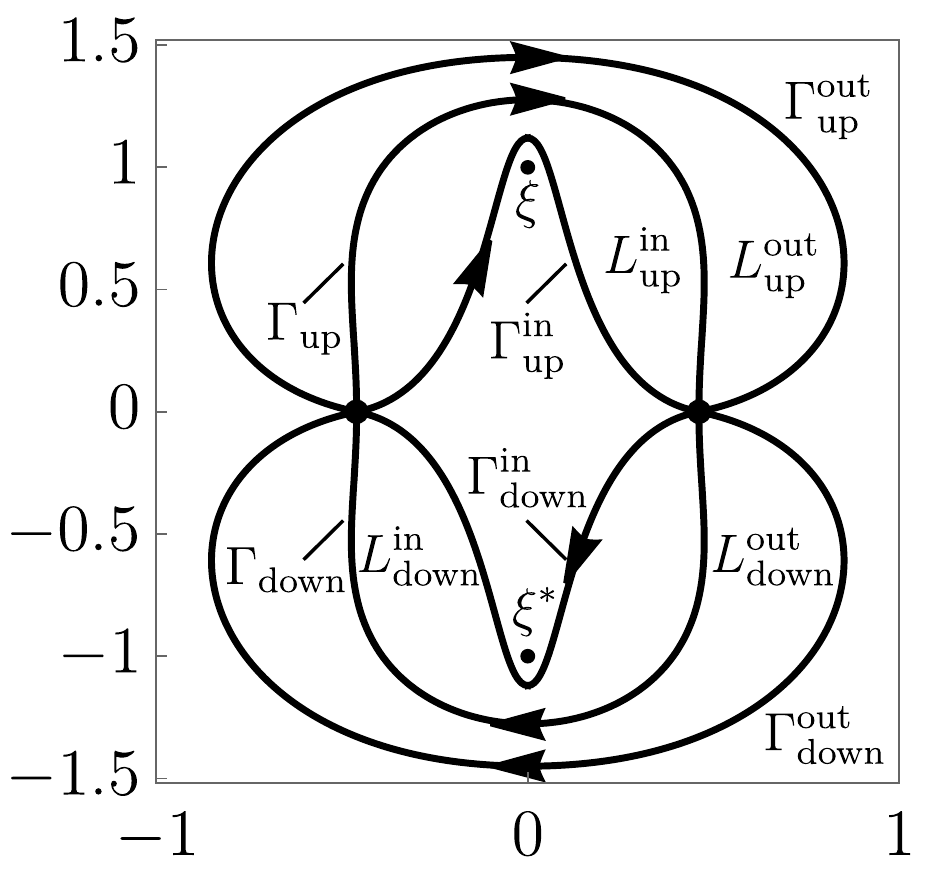} 
\caption{The lenses and lens boundaries in the algebraic-decay region.}
\label{algebraic-lenses}
\end{center}
\end{figure}

Using these lenses, we make the change of variables 
\eq
{\bf Q}^{[n]}(\lambda;\chi,\tau) := \begin{cases}
{\bf O}^{[n]}(\lambda;\chi,\tau)\bbm 1 & \frac{c_2^*}{c_1}e^{-2n\varphi(\lambda;\chi,\tau)} \\ 0 & 1 \ebm, & \lambda\in L_\text{up}^\text{in}, \\
{\bf O}^{[n]}(\lambda;\chi,\tau)\bbm 1 & 0 \\ -\frac{c_2}{c_1}e^{2n\varphi(\lambda;\chi,\tau)} & 1 \ebm^{-1}, & \lambda\in L_\text{up}^\text{out}, \vspace{.025in} \\
{\bf O}^{[n]}(\lambda;\chi,\tau)\bbm 1 & 0 \\ -\frac{c_2}{c_1^*}e^{2n\varphi(\lambda;\chi,\tau)} & 1 \ebm, & \lambda\in L_\text{down}^\text{in}, \\
{\bf O}^{[n]}(\lambda;\chi,\tau)\bbm 1 & \frac{c_2^*}{c_1^*}e^{-2n\varphi(\lambda;\chi,\tau)} \\ 0 & 1 \ebm^{-1}, & \lambda\in L_\text{down}^\text{out}, \\
{\bf O}^{[n]}(\lambda;\chi,\tau), & \text{otherwise}. \end{cases}
\endeq
Then \(\mathbf{Q}^{[n]}(\lambda ; \chi, \tau)\) is analytic for \(\lambda \notin \Gamma \), has the normalization \(\mathbf{Q}^{[n]}(\lambda ; \chi, \tau)=\mathbb{I}+\mathcal{O}\left(\lambda^{-1}\right)\) as \(\lambda \rightarrow \infty\), and satisfies the jump condition \(\mathbf{Q}_{+}^{[ n]}(\lambda ; \chi, \tau)=\mathbf{Q}_{-}^{[ n]}(\lambda ; \chi, \tau) \mathbf{V}_{\mathbf{Q}}^{[n]}(\lambda ; \chi, \tau)\) for $\lambda\in \Gamma$, where 
\begin{equation}
\mathbf{V}_{\mathbf{Q}}^{[n]}(\lambda ; \chi, \tau) :=
\begin{cases}
\bbm 1 & \frac{c_{2}^{*}}{c_{1}}   e^{-2 n \varphi(\lambda ; \chi, \tau)} \\ 0 &  1 \ebm, & \lambda \in \Gamma_\text{up}^\text{in}\,,\vspace{.025in}\\
\bbm \frac{|\mathbf{c}|}{c_1} & 0 \\ 0 &  \frac{c_1}{|\mathbf{c}|} \ebm, & \lambda \in \Gamma_\text{up}\,,\vspace{.025in}\\
\bbm 1 & 0 \\  -\frac{c_{2}}{c_{1}}   e^{2 n \varphi(\lambda ; \chi, \tau)} &  1 \ebm, & \lambda \in \Gamma_\text{up}^\text{out}\,,\vspace{.025in}\\
\bbm 1 & 0 \\  -\frac{c_{2}}{c_{1}^*}   e^{2 n \varphi(\lambda ; \chi, \tau)} &  1 \ebm, & \lambda \in \Gamma_\text{down}^\text{in}\,,\\
\bbm \frac{c_1^*}{|\mathbf{c}|} & 0 \\ 0 &  \frac{|\mathbf{c}|}{c_1^*} \ebm, & \lambda \in \Gamma_\text{down}\,,\\
\bbm 1 & \frac{c_{2}^*}{c_{1}^*}   e^{-2 n \varphi(\lambda ; \chi, \tau)} \\ 0 &  1 \ebm, & \lambda \in \Gamma_\text{down}^\text{out}\,.\\
\end{cases}
\end{equation}
We perform the following sectionally analytic substitutions to eliminate the jump matrices supported on $\Gamma_\text{up}$ and $\Gamma_\text{down}$ at the expense of introducing a jump discontinuity across the interval 
\eq
I:=[\lambda^{(1)}, \lambda^{(2)}]\subset \mathbb{R}
\endeq
separating the regions $D_\xi$ and $D_{\xi^*}$:
\begin{equation}
\mathbf{R}^{[n]}(\lambda; \chi, \tau):=
\begin{cases}
\mathbf{Q}^{[n]}(\lambda; \chi, \tau) \begin{bmatrix} \frac{|\mathbf{c}|}{c_1} & 0 \\ 0 &  \frac{c_1}{|\mathbf{c}|} \end{bmatrix}, & \lambda\in D_\text{up} \setminus \Gamma_\text{up}^\text{in}\,,\\
\mathbf{Q}^{[n]}(\lambda; \chi, \tau) \begin{bmatrix} \frac{c_1^*}{|\mathbf{c}|} & 0 \\ 0 &  \frac{|\mathbf{c}|}{c_1^*} \end{bmatrix}, & \lambda\in D_\text{down} \setminus \Gamma_\text{down}^\text{in}\,,\\
\mathbf{Q}^{[n]}(\lambda; \chi, \tau), & \text{otherwise}.
\end{cases}
\end{equation}
This substitution preserves the normalization 
\(\mathbf{R}^{[n]}(\lambda)=\mathbb{I}+\mathcal{O}\left(\lambda^{-1}\right)\) 
as \(\lambda \rightarrow \infty\) and $\mathbf{R}^{[n]}(\lambda)$ is analytic 
for $\lambda\notin\Gamma\cup I$. We orient $I$ from $\lambda^{(1)}$ to 
$\lambda^{(2)}$. Then 
$\mathbf{R}^{[n]}(\lambda)$ satisfies the jump condition \(\mathbf{R}_{+}^{[ n]}(\lambda ; \chi, \tau)=\mathbf{R}_{-}^{[ n]}(\lambda ; \chi, \tau) \mathbf{V}_{\mathbf{R}}^{[n]}(\lambda ; \chi, \tau)\) for $\lambda\in \Gamma\cup I$, where 
\begin{equation}
\mathbf{V}_{\mathbf{R}}^{[n]}(\lambda ; \chi, \tau) :=
\begin{cases}
\bbm 1 & \frac{c_1 c_2^*}{|\mathbf{c}|^2}   e^{-2 n \varphi(\lambda ; \chi, \tau)} \\ 0 &  1 \ebm, & \lambda \in \Gamma_\text{up}^\text{in}\,,\vspace{.025in}\\
\bbm 1 & 0 \\ - \frac{c_{2}}{c_{1}}   e^{2 n \varphi(\lambda ; \chi, \tau)} &  1 \ebm, & \lambda \in \Gamma_\text{up}^\text{out}\,,\vspace{.025in}\\
\bbm 1 & 0 \\  -\frac{c_1^* c_{2}}{|\mathbf{c}|^2}   e^{2 n \varphi(\lambda ; \chi, \tau)} &  1 \ebm, & \lambda \in \Gamma_\text{down}^\text{in}\,,\vspace{.025in}\\
\bbm 1 & \frac{c_2^*}{c_{1}^*}   e^{-2 n \varphi(\lambda ; \chi, \tau)} \\ 0 &  1 \ebm, & \lambda \in \Gamma_\text{down}^\text{out}\,,\\
\bbm \frac{|\mathbf{c}|^2}{|c_1|^2} & 0 \\ 0 &  \frac{|c_1|^2}{|\mathbf{c}|^2} \ebm, & \lambda \in I.
\end{cases}
\label{eq:R-jump1}
\end{equation}
This piecewise analytic transformation also preserves the recovery formula
\begin{equation}
\psi^{[2n]}(n\chi,n\tau) = 2i \lim_{\lambda\to\infty} \lambda [{\bf R}^{[n]}(\lambda;\chi, \tau)]_{12}.
\label{eq:R-recovery}
\end{equation}
Some algebraic manipulations of the jump matrix are now in order. First, we recall 
$\theta(\lambda;\chi,\tau) := -i \varphi(\lambda;\chi,\tau)$ from \eqref{theta-def}
and then note that the elements of the diagonal jump matrix supported on $I$ satisfy
\begin{equation}
\frac{|\mathbf{c}|^2}{|c_1|^2} = 1 + \left\lvert\frac{c_2}{c_1}\right\rvert^2 =  e^{2\pi p},\quad p:=\frac{1}{2\pi} \log\left(1 + \left\lvert\frac{c_2}{c_1}\right\rvert^2 \right)>0.
\end{equation}
Now, set
\begin{equation}
\kappa:=\left\lvert \frac{c_2}{c_1} \right\rvert >0,\quad \nu:=\arg\left(\frac{c_2}{c_1}\right),
\end{equation}
where $\arg(\cdot)$ denotes the principal branch, and observe that
\begin{equation}
\frac{c_1 c_2^*}{|\mathbf{c}|^2} = \frac{c_2^*}{c_1^*} \frac{|c_1|^2}{|\mathbf{c}|^2}=\kappa  e^{-i \nu}  e^{-2\pi p}.
\end{equation}
Thus, we can rewrite the jump matrix \eqref{eq:R-jump1} as
\begin{equation}
\mathbf{V}_{\mathbf{R}}^{[n]}(\lambda ; \chi, \tau) =
\begin{cases}
\bbm 1 & \kappa  e^{-i\nu} e^{-2\pi p}   e^{-2 i n \theta(\lambda ; \chi, \tau)} \\ 0 &  1 \ebm, & \lambda \in \Gamma_\text{up}^\text{in}\,,\vspace{.025in}\\
\bbm 1 & 0 \\  -\kappa  e^{i\nu}   e^{2 i n \theta(\lambda ; \chi, \tau)} &  1 \ebm, & \lambda \in \Gamma_\text{up}^\text{out}\,,\vspace{.025in}\\
\bbm 1 & 0 \\  -\kappa  e^{i\nu}  e^{-2\pi p}  e^{2 i n \theta(\lambda ; \chi, \tau)} &  1 \ebm, & \lambda \in \Gamma_\text{down}^\text{in}\,,\vspace{.025in}\\
\bbm 1 & \kappa e^{-i \nu}  e^{-2 i n \theta(\lambda ; \chi, \tau)} \\ 0 &  1 \ebm, & \lambda \in \Gamma_\text{down}^\text{out}\,,\\
 e^{2\pi p \sigma_3}, & \lambda \in I.
\end{cases}
\label{eq:R-jump}
\end{equation}
By Lemma \ref{algebraic-lemma}, all of the jump matrices except for the 
diagonal jump matrix $ e^{2\pi p \sigma_3}$ supported on $I$ decay 
exponentially fast to the identity matrix as $n\to\infty$ away from the 
critical points $\lambda^{(1)}$ and $\lambda^{(2)}$.  The asymptotic analysis 
now closely follows \cite[\S 4.1]{BilmanLM:2018}.

\subsubsection*{Parametrix Construction} We eliminate the constant jump 
condition on $I$ and deal with the non-uniform decay near the points 
$\lambda^{(1)}$ and $\lambda^{(2)}$ with the aid of a 
\emph{global parametrix} $\mathbf{T}^{[n]}(\lambda)$. First, define an 
\emph{outer parametrix} by
\begin{equation}
\mathbf{T}^{(\infty)}(\lambda;\chi,\tau):= \left(\frac{\lambda-\lambda^{(1)}(\chi,\tau)}{\lambda-\lambda^{(2)}(\chi,\tau)}\right)^{i p\sigma_3},
\label{eq:T-out}
\end{equation}
where the powers $\pm i p$ are taken as the principal branch so that the 
locus where $(\lambda-\lambda^{(1)})(\lambda-\lambda^{(2)})^{-1}$ is negative 
coincides with the interval $I$. It is clear that 
$\mathbf{T}^{(\infty)}(\lambda;\chi,\tau)=\mathbb{I}+\mathcal{O}(\lambda^{-1})$ 
as $\lambda\to \infty$ and it can be easily verified that 
$\mathbf{T}^{(\infty)}(\lambda;\chi,\tau)$ is analytic for $\lambda$ in 
$\mathbb{C}\setminus I$, satisfying the jump condition
\begin{equation}
\mathbf{T}^{(\infty)}_+(\lambda;\chi,\tau)=\mathbf{T}^{(\infty)}_-(\lambda;\chi,\tau)  e^{2\pi p \sigma_3},\quad \lambda\in I.
\end{equation}

We now move onto constructing \emph{inner parametrices} that will satisfy the 
jump conditions exactly in small, $n$-independent disks $\mathbb{D}^{(1)}$ 
and $\mathbb{D}^{(2)}$ centered at $\lambda^{(1)}$ and $\lambda^{(2)}$, 
respectively.  Before proceeding, we note that 
\eq
\label{thetaprime-conds}
\theta''(\lambda^{(1)};\chi,\tau)<0 \quad \text{and} \quad \theta''(\lambda^{(2)};\chi,\tau)>0
\endeq
for $(\chi,\tau)$ in the algebraic-decay region.  To see this, recall from 
\S\ref{subsec-regions} that the interval $0<\chi<\frac{2}{\beta}$ with 
$\tau=0$ is always contained in the algebraic-decay region.  Direct 
calculation shows that 
\eq
\theta'(\lambda;\chi,0) = \frac{\chi(\lambda-\alpha)^2+\beta^2\chi-2\beta}{(\lambda-\alpha)^2+\beta^2}, \quad \theta''(\lambda;\chi,0) = \frac{4\beta(\lambda-\alpha)}{(\alpha^2+\beta^2-2\alpha\lambda+\lambda^2)^2}
\endeq
(recall $\xi=\alpha+i\beta$).  From the first equation it is immediate that 
$\lambda^{(1)}<0<\lambda^{(2)}$ for $\tau=0$ since $0<\chi<\frac{2}{\beta}$.  
Then the second equation shows that $\theta''(\lambda)<0$ whenever 
$\lambda<\alpha$ (and so, in particular, $\theta''(\lambda^{(1)})<0$) and 
that $\theta''(\lambda)>0$ whenever $\lambda>\alpha$ (and so, in particular, 
$\theta''(\lambda^{(2)})>0$).  Now $\theta(\lambda;\chi,\tau)$ is continuous 
for real $\lambda$, $\chi$, and $\tau$ (with the exception of an additive jump 
of $2\pi i$ across the logarithmic branch cut), and thus the only way the 
concavity at the critical points can change is if two critical points 
coincide.  However, this condition is exactly the boundary of the 
algebraic-decay region, and thus \eqref{thetaprime-conds} holds true 
everywhere in the algebraic-decay region.  

Now, recalling that $\theta'(\lambda^{(1)};\chi,\tau)=0$ and 
$\theta'(\lambda^{(2)};\chi,\tau)=0$, we define the conformal mappings 
$f_1(\lambda;\chi,\tau)$ and $f_2(\lambda;\chi,\tau)$ locally near 
$\lambda=\lambda^{(1)}$ and $\lambda=\lambda^{(2)}$, respectively, by
\begin{equation}
f_1(\lambda;\chi,\tau)^2 := 2(\theta(\lambda^{(1)};\chi,\tau)-\theta(\lambda;\chi,\tau)) \quad \text{and}\quad f_2(\lambda;\chi,\tau)^2 := 2(\theta(\lambda;\chi,\tau)-\theta(\lambda^{(2)};\chi,\tau)),
\label{eq:f-1-2}
\end{equation}
where we choose the solutions satisfying $f_1'(\lambda^{(1)};\chi,\tau)<0$ 
and $f_2'(\lambda^{(2)};\chi,\tau)>0$. Now introducing the rescaled conformal 
coordinates 
\eq
\zeta_1 := n^{1/2} f_1(\lambda;\chi,\tau), \quad \zeta_2 := n^{1/2} f_2(\lambda;\chi,\tau)
\endeq
and taking the rotation by $\pi$ performed by $f_1$ into account, the jump 
conditions satisfied by
\begin{equation}
\mathbf{U}^{(1)}(\lambda;\chi,\tau):= \mathbf{R}^{[n]}(\lambda;\chi,\tau) e^{-i n \theta(\lambda^{(1)};\chi,\tau)\sigma_3} e^{-i\nu\sigma_3/2}\bbm 0 & 1 \\ -1 & 0 \ebm, \quad \lambda\in\mathbb{D}^{(1)}
\label{eq:R-local-1}
\end{equation}
and by
\begin{equation}
\mathbf{U}^{(2)}(\lambda;\chi,\tau):= \mathbf{R}^{[n]}(\lambda;\chi,\tau) e^{-i n \theta(\lambda^{(2)};\chi,\tau)\sigma_3} e^{-i\nu\sigma_3/2}, \quad \lambda\in\mathbb{D}^{(2)}
\label{eq:R-local-2}
\end{equation}
have the same form when expressed in terms of the respective conformal 
coordinates $\zeta = \zeta_1$ and $\zeta=\zeta_2$ and when the jump contours 
are locally taken to be the rays $\arg(\zeta)=\pm \pi/4$, 
$\arg(\zeta)=\pm 3\pi/4$, and $\arg(-\zeta)=0$. Moreover, the resulting jump 
conditions coincide precisely with those in Riemann-Hilbert Problem A.1 
for a parabolic cylinder parametrix in \cite[Appendix A]{Miller:2018}. See 
Figure 9 in \cite{Miller:2018} for the relevant jump contours and matrices. 
Note that the condition $\kappa^2 =  e^{2\pi p} -1$ for consistency of jump 
conditions at $\zeta=0$ holds. We now let $\mathbf{U}(\zeta)$ denote the 
unique solution of the Riemann-Hilbert Problem A.1 in 
\cite[Appendix A]{Miller:2018}.  Here \(\mathbf{U}(\zeta)\) is analytic for 
\(\zeta\) in the five sectors 
$|\arg(\zeta)|<\frac{1}{4}\pi$, $\frac{1}{4}\pi<\arg(\zeta)<\frac{3}{4}\pi$,
$-\frac{3}{4} \pi<\arg (\zeta)<-\frac{1}{4} \pi$, 
$\frac{3}{4} \pi<\arg (\zeta)<\pi$, and $-\pi<\arg(\zeta)<-\frac{3}{4}\pi$.  
It takes continuous boundary values on the excluded rays and at the origin 
from each sector.  Furthermore, 
\(\mathbf{U}(\zeta) \zeta^{i p \sigma_{3}} = \mathbb{I}+\mathcal{O}(\zeta^{-1})\) as \(\zeta \rightarrow \infty\) uniformly in all directions and 
from each sector.  We also have that \(\mathbf{U}(\zeta) \zeta^{i p \sigma_{3}} \) has a complete asymptotic series expansion in descending integer powers of $\zeta$ as $\zeta\to\infty$, with all coefficients being independent of the sector in which $\zeta\to\infty$ \cite[Appendix A.1]{Miller:2018}. In more detail, as given in (A.9) in \cite{Miller:2018}, we have
\begin{equation}
\mathbf{U}(\zeta)\zeta^{i p \sigma_3} = \mathbb{I } + \frac{1}{2 i \zeta}\begin{bmatrix} 0 & r(p, \kappa) \\ -q(p, \kappa) & 0 \end{bmatrix} + \begin{bmatrix}\mathcal{O}(\zeta^{-2}) & \mathcal{O}(\zeta^{-3}) \\ \mathcal{O}(\zeta^{-3}) & \mathcal{O}(\zeta^{-2}) \end{bmatrix}, \quad \zeta\to \infty,
\label{eq:PC-normalization}
\end{equation}
where
\begin{equation}
r(p,\kappa):=2  e^{i \pi/4}\sqrt{\pi}  \frac{ e^{\pi p / 2}  e^{i p \ln (2)}}{\kappa \Gamma(i p)},\quad q(p,\kappa):= -\frac{2 p}{r(p,\kappa)}.
\label{eq:r-q}
\end{equation}

We introduce the inner parametrices $\mathbf{T}^{(1)}(\lambda)$ and 
$\mathbf{T}^{(2)}(\lambda)$ by
\begin{equation}
\mathbf{T}^{(1)}(\lambda;\chi,\tau):=\mathbf{Y}^{(1)}(\lambda;\chi,\tau)\mathbf{U}(n^{1/2}f_1(\lambda;\chi,\tau))\bbm 0 & -1 \\ 1 & 0 \ebm e^{i\nu\sigma_3/2} e^{i n \theta(\lambda^{(1)};\chi,\tau)\sigma_3},\quad \lambda\in \mathbb{D}^{(1)}
\label{eq:T-1}
\end{equation}
and
\begin{equation}
\mathbf{T}^{(2)}(\lambda;\chi,\tau):=\mathbf{Y}^{(2)}(\lambda;\chi,\tau)\mathbf{U}(n^{1/2}f_2(\lambda;\chi,\tau)) e^{i\nu\sigma_3/2} e^{i n \theta(\lambda^{(2)};\chi,\tau)\sigma_3},\quad \lambda\in \mathbb{D}^{(2)},
\label{eq:T-2}
\end{equation}
where the holomorphic prefactor matrices $\mathbf{Y}^{(1)}(\lambda)$ and 
$\mathbf{Y}^{(2)}(\lambda)$ will now be chosen to match well with the 
outer parametrix $\mathbf{T}^{(\infty)}$ on the disk boundaries 
$\partial \mathbb{D}^{(j)}$, $j=1,2$.  Define 
\eq
\begin{split}
\mathbf{H}^{(1)}(\lambda;\chi,\tau) & := (\lambda^{(2)} - \lambda)^{-i p \sigma_3} \left( \frac{\lambda^{(1)}-\lambda}{f_1(\lambda;\chi,\tau)}\right)^{i p \sigma_3} \bbm 0 & 1 \\ -1 & 0 \ebm, \quad \lambda\in\mathbb{D}^{(1)}, \\
\mathbf{H}^{(2)}(\lambda;\chi,\tau) & := (\lambda - \lambda^{(1)})^{i p \sigma_3} \left( \frac{f_2(\lambda;\chi,\tau)}{\lambda-\lambda^{(2)}}\right)^{i p \sigma_3}, \quad \lambda\in\mathbb{D}^{(2)}.
\end{split}
\endeq
Here all the power functions are taken as the principal branch, and hence 
$\mathbf{H}^{(1)}(\lambda)$ and $\mathbf{H}^{(2)}(\lambda)$ are holomorphic 
as matrix-valued functions of $\lambda$ in their domain of definition.  
Recalling the transformations \eqref{eq:R-local-1} and \eqref{eq:R-local-2}, 
note that the outer parametrix $\mathbf{T}^{(\infty)}(\lambda)$ can be 
expressed locally as 
\eq
\mathbf{T}^{(\infty)}(\lambda) e^{-i n \theta(\lambda^{(1)})\sigma_3} e^{-i\nu\sigma_3/2}\bbm 0 & 1 \\ -1 & 0 \ebm = n^{-i p \sigma_3 / 2} e^{-i \nu \sigma_3 / 2}  e^{-i n \theta(\lambda^{(1)})\sigma_3}\mathbf{H}^{(1)}(\lambda)  \zeta_{1}^{-i p \sigma_3}, \quad \lambda\in\mathbb{D}^{(1)}
\label{eq:T-out-local-1}
\endeq
and
\eq
\mathbf{T}^{(\infty)}(\lambda) e^{-i n \theta(\lambda^{(2)})\sigma_3} e^{-i\nu\sigma_3/2} = n^{i p \sigma_3 / 2} e^{-i \nu \sigma_3 / 2}  e^{-i n \theta(\lambda^{(2)})\sigma_3}\mathbf{H}^{(2)}(\lambda)  \zeta_{2}^{-i p \sigma_3}, \quad \lambda\in\mathbb{D}^{(2)}.
\label{eq:T-out-local-2}
\endeq
In light of these formul\ae{}, we choose
\begin{equation}
\mathbf{Y}^{(1)}(\lambda)=\mathbf{Y}^{(1)}(\lambda;\chi,\tau,n):=n^{-i p \sigma_3 / 2} e^{-i \nu \sigma_3 / 2}  e^{-i n \theta(\lambda^{(1)};\chi,\tau)\sigma_3}\mathbf{H}^{(1)}(\lambda;\chi,\tau) 
\end{equation}
and
\begin{equation}
\mathbf{Y}^{(2)}(\lambda)=\mathbf{Y}^{(2)}(\lambda;\chi,\tau,n):=n^{i p \sigma_3 / 2} e^{-i \nu \sigma_3 / 2}  e^{-i n \theta(\lambda^{(2)};\chi,\tau)\sigma_3}\mathbf{H}^{(2)}(\lambda;\chi,\tau),
\end{equation}
noting that both of these matrix-valued functions remain bounded as 
$n\to\infty$ and $\mathbf{Y}^{(j)}(\lambda;\chi,\tau)$ is a holomorphic 
function for $\lambda\in\mathbb{D}^{(j)}$, $j=1,2$. Then from 
\eqref{eq:T-1} and \eqref{eq:T-out-local-1} it follows that
\eq
\begin{split}
\mathbf{T}^{(1)}&(\lambda)\mathbf{T}^{(\infty)}(\lambda)^{-1} \\ 
  & = n^{-i p \sigma_3 / 2} e^{-i \nu \sigma_3 / 2}  e^{-i n \theta(\lambda^{(1)})\sigma_3}  \mathbf{H}^{(1)}(\lambda)  \mathbf{U}(\zeta_1) \zeta_1^{i p \sigma_3}\mathbf{H}^{(1)}(\lambda)^{-1}e^{i n \theta(\lambda^{(1)})\sigma_3} e^{i \nu \sigma_3 / 2} n^{i p \sigma_3 / 2}
\label{eq:W-jump-D1}
\end{split}
\endeq
for $\lambda\in\partial \mathbb{D}^{(1)}$, 
and from \eqref{eq:T-2} and \eqref{eq:T-out-local-2} it follows that
\eq
\begin{split}
\mathbf{T}^{(2)}&(\lambda)\mathbf{T}^{(\infty)}(\lambda)^{-1} \\ 
 & = n^{i p \sigma_3 / 2} e^{-i \nu \sigma_3 / 2}  e^{-i n \theta(\lambda^{(2)})\sigma_3}  \mathbf{H}^{(2)}(\lambda)  \mathbf{U}(\zeta_2)\zeta_2^{i p \sigma_3}\mathbf{H}^{(2)}(\lambda)^{-1}e^{i n \theta(\lambda^{(2)})\sigma_3} e^{i \nu \sigma_3 / 2} n^{-i p \sigma_3 / 2}
\label{eq:W-jump-D2}
\end{split}
\endeq
for $\lambda\in\partial \mathbb{D}^{(2)}$.

Finally, we define the \emph{global parametrix} 
$\mathbf{T}^{[n]}(\lambda;\chi,\tau)$ by
\begin{equation}
\mathbf{T}^{[n]}(\lambda;\chi,\tau):= \begin{cases}
\mathbf{T}^{(1)}(\lambda;\chi,\tau),& \lambda\in \mathbb{D}^{(1)},\\
\mathbf{T}^{(2)}(\lambda;\chi,\tau),& \lambda\in \mathbb{D}^{(2)},\\
\mathbf{T}^{(\infty)}(\lambda;\chi,\tau), & \text{otherwise}.
\end{cases}
\end{equation}
Note that $\mathbf{T}^{[n]}(\lambda;\chi,\tau)$ is a sectionally analytic 
function of $\lambda$, the determinant of 
$\mathbf{T}^{[n]}(\lambda;\chi,\tau))$ is identically 1, and 
$\mathbf{T}^{[n]}(\lambda;\chi,\tau)=\mathbb{I}+\mathcal{O}(\lambda^{-1})$ 
as $\lambda\to\infty$.

\subsubsection*{Error Analysis and Asymptotics}
We proceed by quantifying the error made in approximating 
$\mathbf{R}^{[n]}(\lambda;\chi,\tau)$ by the global parametrix 
$\mathbf{T}^{[n]}(\lambda;\chi,\tau)$.  Consider the ratio
\begin{equation}
\mathbf{W}^{[n]}(\lambda;\chi,\tau) := \mathbf{R}^{[n]}(\lambda;\chi,\tau)\mathbf{T}^{[n]}(\lambda;\chi,\tau)^{-1}.
\label{eq:W-def}
\end{equation}
Now $\mathbf{W}^{[n]}$ extends as a sectionally analytic function of $\lambda$ to $\mathbb{C}\setminus (\partial \mathbb{D}^{(1)} \cup \partial \mathbb{D}^{(2)} \cup \Gamma_\mathbf{W})$, where
\begin{equation}
 \Gamma_\mathbf{W}:= \Gamma\setminus \left(\overline{\mathbb{D}^{(1)}} \cup \overline{\mathbb{D}^{(2)}} \right)=(\Gamma_\text{up}^\text{in} \cup \Gamma_\text{up}^\text{out}\cup \Gamma_\text{down}^\text{in}\cup \Gamma_\text{down}^\text{out})  \setminus\left(\overline{\mathbb{D}^{(1)}} \cup \overline{\mathbb{D}^{(2)}} \right)
\end{equation}
denotes the portion of $\Gamma$ across which $\mathbf{W}^{[n]}$ has a jump 
discontinuity.  Take $\partial \mathbb{D}^{(1)} $ and 
$\partial \mathbb{D}^{(2)} $ to have clockwise orientations. Thus, 
$\mathbf{W}^{[n]}$ satisfies a jump condition of the form
\begin{equation}
\mathbf{W}_{+}^{[n]}(\lambda;\chi,\tau)=\mathbf{W}_{-}^{[n]}(\lambda;\chi,\tau)\mathbf{V}_\mathbf{W}^{[n]}(\lambda;\chi,\tau),\quad \lambda\in \partial \mathbb{D}^{(1)} \cup \partial \mathbb{D}^{(2)} \cup \Gamma_\mathbf{W}.
\label{eq:W-jump}
\end{equation}
Since $\mathbf{T}^{(\infty)}(\lambda)$ defined in \eqref{eq:T-out} is 
analytic across any arc of $\Gamma_\mathbf{W}$, we have
\begin{equation}
\begin{aligned}
\mathbf{V}_\mathbf{W}^{[n]}(\lambda;\chi,\tau)& = \mathbf{W}_{-}(\lambda;\chi,\tau)^{-1} \mathbf{W}_{+}(\lambda;\chi,\tau)\\
&=\mathbf{T}^{(\infty)}(\lambda;\chi,\tau)\mathbf{R}^{[n]}_{-}(\lambda;\chi,\tau)^{-1}\mathbf{R}^{[n]}_{+}(\lambda;\chi,\tau)\mathbf{T}^{(\infty)}(\lambda;\chi,\tau)^{-1},\quad \lambda\in\Gamma_\mathbf{W},
\end{aligned}
\label{eq:W-jump-arcs}
\end{equation}
where the product 
$\mathbf{R}^{[n]}_{-}(\lambda;\chi,\tau)^{-1}\mathbf{R}^{[n]}_{+}(\lambda;\chi,\tau)$ 
coincides with $\mathbf{V}_\mathbf{R}^{[n]}(\lambda;\chi,\tau)$ given in 
\eqref{eq:R-jump}. Since the exponential factors 
$ e^{\pm 2i n \theta(\lambda;\chi,\tau)}$ in \eqref{eq:R-jump} are restricted 
to the exterior of the disks $\mathbb{D}^{(1)}$ and $\mathbb{D}^{(2)}$ in 
\eqref{eq:W-jump-arcs}, and $\mathbf{T}^{(\infty)}(\lambda;\chi,\tau)$ is 
independent of $n$, there exists a constant $d\equiv d(\chi,\tau)>0$ such that
\begin{equation}
\sup_{\lambda\in\Gamma_\mathbf{W}} \| \mathbf{V}_\mathbf{W}^{[n]}(\lambda;\chi,\tau) - \mathbb{I} \| = \mathcal{O}( e^{-n d(\chi,\tau)}), \quad n\to \infty,
\label{eq:W-jump-arcs-estimate}
\end{equation}
where $\| \cdot \|$ denotes the matrix norm induced from an arbitrary vector 
norm on $\mathbb{C}^2$. On the remaining jump contours 
$\partial \mathbb{D}^{(1)}\cup \partial \mathbb{D}^{(2)}$ for 
$\mathbf{W}^{[n]}(\lambda)$ (see \eqref{eq:W-jump}), we have
\begin{equation}
\mathbf{V}^{[n]}_{\mathbf{W}}(\lambda;\chi,\tau) = \mathbf{T}^{(j)}(\lambda;\chi,\tau) \mathbf{T}^{(\infty)}(\lambda;\chi,\tau)^{-1},\quad \lambda\in \partial \mathbb{D}^{(j)},~j=1,2.
\label{eq:V-for-W}
\end{equation}
Now, observe that the factors conjugating 
$\mathbf{U}(\zeta_j) \zeta_j^{i p \sigma_3}$, $j=1,2$ in \eqref{eq:W-jump-D1} 
and \eqref{eq:W-jump-D2} all remain bounded as $n\to\infty$. Recalling that 
$\zeta_{j}$ is proportional to $n^{-1/2}$ for $z\in\mathbb{D}^{(j)}$, from 
\eqref{eq:PC-normalization} we obtain
\begin{equation}
\sup_{\lambda\in \partial \mathbb{D}^{(1)}\cup \partial \mathbb{D}^{(2)} } \| \mathbf{V}_\mathbf{W}^{[n]}(\lambda;\chi,\tau) - \mathbb{I} \| = \mathcal{O}(n^{-1/2}), \quad n\to \infty.
\label{eq:W-jump-circles-estimate}
\end{equation}
The jump condition \eqref{eq:W-jump} implies that 
\eq
\mathbf{W}^{[n]}_+(\lambda) - \mathbf{W}^{[n]}_-(\lambda) = \mathbf{W}^{[n]}_-(\lambda) (\mathbf{V}^{[n]}_\mathbf{W}(\lambda)-\mathbb{I}),
\endeq
and 
$\mathbf{W}^{[n]}(\lambda;\chi,\tau)=\mathbb{I}+\mathcal{O}(\lambda^{-1})$ 
as $\lambda\to\infty$ since both $\mathbf{R}^{[n]}(\lambda;\chi, \tau)$ and 
$\mathbf{T}^{[n]}(\lambda;\chi,\tau)^{-1}$ are normalized to the identity as 
$\lambda\to\infty$. Therefore, it follows from the Plemelj formula that
\eq
\begin{split}
\mathbf{W}^{[n]}(\lambda ; \chi, \tau) = \mathbb{I}+\frac{1}{2 \pi i}\int_{\partial \mathbb{D}^{(1)} \cup \partial \mathbb{D}^{(2)} \cup \Gamma_\mathbf{W}} \frac{\mathbf{W}^{[n]}_{-}(s; \chi, \tau) (\mathbf{V}^{[n]}_\mathbf{W}(s;\chi,\tau)-\mathbb{I})}{s-\lambda}\,d s,&\\
 \lambda\in \mathbb{C}\setminus\big( \partial \mathbb{D}^{(1)} \cup \partial \mathbb{D}^{(2)} \cup & \Gamma_\mathbf{W} \big).
\label{eq:W-integral-eq}
\end{split}
\endeq

Precisely as in \cite[\S 4.1]{BilmanLM:2018}, one can let $\lambda$ tend 
to a point on the contour 
$\partial \mathbb{D}^{(1)} \cup \partial \mathbb{D}^{(2)} \cup \Gamma_\mathbf{W}$ 
from the right side with respect to the orientation to obtain a closed 
integral equation for $\mathbf{W}_-(\lambda;\chi,\tau)$ defined on 
$\partial \mathbb{D}^{(1)} \cup \partial \mathbb{D}^{(2)} \cup \Gamma_\mathbf{W}$ 
away from the self-intersection points. The resulting integral equation is 
uniquely solvable by a Neumann series on 
$L^2(\partial \mathbb{D}^{(1)} \cup \partial \mathbb{D}^{(2)} \cup \Gamma_\mathbf{W})$ 
for sufficiently large $n$, and its solutions satisfy the estimate
\begin{equation}
\mathbf{W}^{[n]}_-(\lambda;\chi,\tau)-\mathbb{I} = \mathcal{O}(n^{-1/2}),\quad n \to \infty
\label{eq:W-L2-estimate}
\end{equation}
in the $L^2(\partial \mathbb{D}^{(1)} \cup \partial \mathbb{D}^{(2)} \cup \Gamma_\mathbf{W})$ sense. We refer the reader to \cite[\S 4.1]{BilmanLM:2018} for the details regarding this argument. From the integral equation \eqref{eq:W-integral-eq} we now extract the Laurent series expansion of $\mathbf{W}^{[n]}(\lambda;\chi,\tau)$ convergent for sufficiently large $\lambda$:
\begin{equation}
\mathbf{W}^{[n]} (\lambda;\chi,\tau)=\mathbb{I} - \frac{1}{2\pi i}\sum_{k=1}^\infty \lambda^{-k} \int_{ \partial \mathbb{D}^{(1)} \cup \partial \mathbb{D}^{(2)} \cup \Gamma_\mathbf{W} } \mathbf{W}^{[n]}_-(s;\chi,\tau)( \mathbf{V}^{[n]}_{\mathbf{W}} (s;\chi,\tau) - \mathbb{I})s^{k-1}\,d s,
\label{eq:W-Laurent}
\end{equation}
for  $|\lambda|>\sup\{|s| \colon s\in \partial \mathbb{D}^{(1)} \cup \partial \mathbb{D}^{(2)} \cup \Gamma_\mathbf{W} \} $.

On the other hand, $\mathbf{T}^{(\infty)}(\lambda;\chi,\tau)$ is a diagonal 
matrix tending to the identity as $\lambda\to\infty$.  From 
\eqref{eq:R-recovery} and \eqref{eq:W-def} it follows that
\begin{equation}
\psi^{[2n]}(n\chi,n\tau) = 2i \lim_{n\to\infty} \lambda [{\bf W}^{[n]}(\lambda;\chi,\tau)]_{12}.
\end{equation}
This, together with the Laurent series expansion \eqref{eq:W-Laurent}, yields 
the expression
\begin{equation}
\begin{aligned}
\psi^{[2n]}(n\chi,n\tau) = -\frac{1}{\pi}&\left( \int_{\partial \mathbb{D}^{(1)} \cup \partial \mathbb{D}^{(2)} \cup \Gamma_\mathbf{W} } [\mathbf{W}^{[n]}_-(s;\chi,\tau)]_{11}[\mathbf{V}^{[n]}_\mathbf{W}(s;\chi,\tau)]_{12}\,d s \right. \\ &\quad +\left.  \int_{\partial \mathbb{D}^{(1)} \cup \partial \mathbb{D}^{(2)} \cup \Gamma_\mathbf{W} } [\mathbf{W}^{[n]}_-(s;\chi,\tau)]_{12}([\mathbf{V}^{[n]}_\mathbf{W}(s;\chi,\tau)]_{22}-1)\,d s  \right).
\end{aligned}
\end{equation}
Now, because the domain of integration in the integrals above is a compact contour, the $L^{1}$-norm on $\partial \mathbb{D}^{(1)} \cup \partial \mathbb{D}^{(2)} \cup \Gamma_\mathbf{W}$ is subordinate to the $L^{2}$-norm.
Therefore, combining the $L^\infty$-type estimates \eqref{eq:W-jump-arcs-estimate} and \eqref{eq:W-jump-circles-estimate} with the $L^2$-type estimate \eqref{eq:W-L2-estimate}, we arrive at
\begin{equation}
\psi^{[2n]}(n\chi,n\tau) = -\frac{1}{\pi} \int_{\partial \mathbb{D}^{(1)} \cup \partial \mathbb{D}^{(2)} \cup \Gamma_\mathbf{W} }[\mathbf{V}^{[n]}_\mathbf{W}(s;\chi,\tau)]_{12}\,d s + \mathcal{O}(n^{-1}),\quad n\to\infty.
\end{equation}
Here the error term is uniform for $(\chi,\tau)$ chosen from any compacta 
inside the interior of the algebraic-decay region.  Moreover, the same 
formula holds with a different error term, of the same order, if we replace 
the integration contour 
$\partial \mathbb{D}^{(1)} \cup \partial \mathbb{D}^{(2)} \cup \Gamma_\mathbf{W}$ 
with $\partial \mathbb{D}^{(1)} \cup \partial \mathbb{D}^{(2)}$ due to the 
exponential decay in the estimate \eqref{eq:W-jump-arcs-estimate}:
\begin{equation}
\psi^{[2n]}(n\chi,n\tau) = -\frac{1}{\pi} \int_{\partial \mathbb{D}^{(1)} \cup \partial \mathbb{D}^{(2)}} [\mathbf{V}^{[n]}_\mathbf{W}(s;\chi,\tau)]_{12}\,d s + \mathcal{O}(n^{-1}),\quad n\to\infty.
\end{equation}
Using \eqref{eq:W-jump-D1} and \eqref{eq:W-jump-D2} together with the 
normalization \eqref{eq:PC-normalization} in \eqref{eq:V-for-W} lets us 
write, as $n\to\infty$, 
\begin{equation}
[\mathbf{V}_\mathbf{W}^{[n]}(\lambda)]_{12} = \frac{n^{-i p}  e^{-i\nu}  e^{-2i n \theta(\lambda^{(1)})} }{2i n^{1/2} f_1(\lambda)}
 q( [\mathbf{H}^{(1)}(\lambda)]_{12})^2+ \mathcal{O}(n^{-1}),\quad \lambda\in\partial \mathbb{D}^{(1)}
\label{eq:V-12-D1}
\end{equation}
and
\begin{equation}
[\mathbf{V}_\mathbf{W}^{[n]}(\lambda)]_{12} = \frac{n^{i p}  e^{-i\nu}  e^{-2i n \theta(\lambda^{(2)})} }{2i n^{1/2} f_2(\lambda)}
 r ([\mathbf{H}^{(2)}(\lambda)]_{11})^2 + \mathcal{O}(n^{-1}),\quad \lambda\in\partial \mathbb{D}^{(2)}\,,
\label{eq:eq:V-12-D2}
\end{equation}
where $r\equiv r(p,k)$ and $q\equiv q(p,k)$ are given in \eqref{eq:r-q}, and 
both of the error estimates are uniform on the relevant circles. As 
$f_{j}(\lambda)$ has a simple zero at $\lambda^{(j)}$, 
and the matrix elements of $\mathbf{H}^{(j)}(\lambda)$ are analytic 
in $\mathbb{D}^{(j)}$, $j=1,2$, the integrals of the explicit leading terms 
in \eqref{eq:W-jump-D1} and \eqref{eq:W-jump-D2} can be evaluated by a  
residue calculation at $\lambda=\lambda^{(1)}$ and at 
$\lambda=\lambda^{(2)}$, respectively. Doing so gives
\eq
\begin{split}
\psi^{[2n]}(n\chi,n\tau) 
  = \frac{ e^{-i \nu}}{n^{1/2}}&\left[
 \frac{n^{-i p}  e^{-2i n \theta(\lambda^{(1)};\chi,\tau)} }{f_1'(\lambda^{(1)};\chi,\tau)} q ([\mathbf{H}^{(1)}(\lambda^{(1)};\chi,\tau)]_{12})^2\right.\\
 & \left.\quad+
 \frac{n^{i p}  e^{-2i n \theta(\lambda^{(2)};\chi,\tau)} }{ f_2'(\lambda^{(2)};\chi,\tau)} r ([\mathbf{H}^{(2)}(\lambda^{(2)};\chi,\tau)]_{11})^2 \right] + \mathcal{O}(n^{-1})
\label{eq:psi-2n-formula-1},\quad n \to +\infty.
\end{split}
\endeq
To get a more explicit formula, note first that by the definitions \eqref{eq:f-1-2} we have
\begin{equation}
f'_1(\lambda^{(1)};\chi,\tau) = - \sqrt{- \theta''(\lambda^{(1)};\chi,\tau)} \quad\text{and}\quad f'_2(\lambda^{(2)};\chi,\tau) = \sqrt{ \theta''(\lambda^{(2)};\chi,\tau)}.
\end{equation}
Next, we calculate the terms {
$[\mathbf{H}^{(1)}(\lambda^{(1)})]_{12}$ and 
$[\mathbf{H}^{(2)}(\lambda^{(2)})]_{11}$} in 
\eqref{eq:psi-2n-formula-1} explicitly. 
Applying l'H\^opital's rule in the 
definitions \eqref{eq:T-out-local-1} and \eqref{eq:T-out-local-2} gives
\begin{equation}
\mathbf{H}^{(1)}(\lambda^{(1)})=(\lambda^{(2)}-\lambda^{(1)})^{-i p \sigma_{3}}\left(\frac{-1}{f_1^{\prime}(\lambda^{(1)})}\right)^{i p\sigma_{3}}\bbm 0 & 1 \\ -1 & 0 \ebm
\end{equation}
and
\begin{equation}
\mathbf{H}^{(2)}(\lambda^{(2)})=(\lambda^{(2)}-\lambda^{(1)})^{i p \sigma_{3}}\left(f_2^{\prime}(\lambda^{(2)})\right)^{i p\sigma_{3}}.
\end{equation}
Thus, we have obtained
\begin{equation}
\begin{aligned}
\frac{ q ([\mathbf{H}^{(1)}(\lambda^{(1)})]_{12})^2}{f_1'(\lambda^{(1)})} &= - (\lambda^{(2)}-\lambda^{(1)})^{-2i p}(-\theta''(\lambda^{(1)}))^{-i p}\frac{q}{\sqrt{-\theta''(\lambda^{(1)})}}\,,\\
\frac{r ([\mathbf{H}^{(2)}(\lambda^{(2)})]_{11})^2 }{f_2'(\lambda^{(2)})} &=(\lambda^{(2)} -\lambda^{(1)} )^{2i p} \theta''(\lambda^{(2)})^{i p} \frac{r}{\sqrt{\theta''(\lambda^{(2)})}}.
\end{aligned}
\end{equation}
Finally, since $p>0$ and $\kappa>0$, it can be deduced that 
$q(p,\kappa) = -r(p,\kappa)^*$ using the identity given in 
\cite[Equation (5.4.3)]{dlmf} for the modulus of the gamma function on the 
imaginary axis. With these at hand, one can check that 
$|r|=|r(p,\kappa)| = \sqrt{2 p}$, and consequently Equation  
\eqref{eq:psi-2n-formula-1} can be rewritten as Equation 
\eqref{psi2n-alg-result}.  This completes the proof of Theorem 
\ref{alg-decay-thm}.

{Since the completion of the first draft of this work, one of the authors and Miller showed \cite{BilmanM:2021} that Theorem~\ref{alg-decay-thm} holds for a more general, continuum family of solutions $\{q(x,t;\mathbf{G}, M)\}_{M>0}$ (in the notation of \cite{BilmanM:2021}) of the focusing NLS equation \eqref{nls}, which includes fundamental rogue wave solutions studied in \cite{BilmanLM:2018, BilmanM:2021} as well as a special case of multiple-pole solitons considered in this work with the choices
\begin{equation}
\mathbf{G}:=\mathcal{S}^{-1}=\frac{1}{\sqrt{2}}\begin{bmatrix} 1  &  1 \\ -1 & 1 \end{bmatrix}\quad\text{and}\quad
\mathbf{G}:=\mathcal{S}^{-1}=\frac{1}{\sqrt{2}}\begin{bmatrix} 1  &  -1 \\ 1 & 1 \end{bmatrix},
\end{equation}
which corresponds to setting $c_1=c_2=1$ and $c_1=-c_2=1$, respectively, along with $\xi=i$.
}

\section{The non-oscillatory region}
\label{sec-nonosc}

We now study the non-oscillatory region.  
In this region the leading-order solution  arises from a single band in the model Riemann-Hilbert problem.  
To see this 
it is necessary to introduce a so-called $g$-function, a standard technique 
in the asymptotic analysis of Riemann-Hilbert problems (see, for instance, 
\cite{DeiftVZ:1997,LyngM:2007}).  Define 
$g(\lambda;\chi,\tau)$ 
as the unique solution of the following Riemann-Hilbert problem. {
Recalling the definitions of the real numbers $\lambda^{(1)}<\lambda^{(2)}$ from Theorem~\ref{nonosc-thm},
we take the branch cut of the function
\begin{equation}
\lambda \mapsto \log\left( \frac{\lambda-\xi^*}{\lambda-\xi}\right)
\end{equation}
appearing in the phase $\varphi(\lambda;\chi,\tau)$ (cf. \eqref{phi-def}) to be a Schwarz-symmetric arc $\Sigma_\mathrm{c}$ which connects $\lambda=\xi$ and $\lambda=\xi^*$ while passing through the midpoint of $\lambda^{(1)}$ and $\lambda^{(2)}$, which will be derived in more detail later on.
}
\begin{rhp}[The $g$-function in the non-oscillatory region]
Fix a pole location $\xi\in\mathbb{C}^+$, a pair of nonzero complex numbers 
$(c_1,c_2)$, and a 
pair of real numbers $(\chi,\tau)$ in the non-oscillatory region.  Determine the unique contour $\Sigma(\chi,\tau)$ and the unique function 
$g(\lambda;\chi,\tau)$ satisfying the following conditions.
\begin{itemize}
\item[]\textbf{Analyticity:}  $g(\lambda)$ is analytic for 
$\lambda\in\mathbb{C}$ except on $\Sigma$, where it achieves continuous boundary 
values.  The contour $\Sigma$ is simple, bounded, and symmetric across the real 
axis.  
\item[]\textbf{Jump condition:}  The boundary values taken by $g(\lambda)$ are
related by the jump condition
\eq
\label{g-condition-1}
g_+(\lambda) + g_-(\lambda) - 2\varphi(\lambda) = { -2 i K}, \quad \lambda \in \Sigma,
\endeq
{where $K=K(\chi,\tau)$ is a real-valued constant to be determined.}
Furthermore, 
\eq
\Re(\varphi(\lambda)-g_+(\lambda)) =  \Re(\varphi(\lambda)-g_-(\lambda)) = 0, \quad \lambda\in\Sigma. 
\endeq
\item[]\textbf{Normalization:}  As $\lambda\to\infty$, $g(\lambda)$
satisfies the condition
\begin{equation}
\label{g-condition-3}
g(\lambda) = \mathcal{O}\left(\lambda^{-1}\right)
\end{equation}
with the limit being uniform with respect to direction.
\item[]\textbf{Symmetry:}  $g(\lambda)$ satisfies the symmetry condition
\eq
\label{g-condition-2}
g(\lambda) = -g(\lambda^*)^*.\\
\endeq
\end{itemize}
\label{rhp:g-nonosc}
\end{rhp}
We now solve Riemann-Hilbert Problem \ref{rhp:g-nonosc} by first solving for 
$g'(\lambda)$.  Note that the function $g'(\lambda)$ satisfies the jump condition 
\eq
\label{g'-condition-1}
g'_+(\lambda) + g'_-(\lambda) = 2i \chi+4i\lambda\tau+\frac{2}{\lambda-\xi^*}-\frac{2}{\lambda-\xi}, \quad \lambda \in \Sigma
\endeq
and the normalization 
\eq
\label{g'-condition-3}
g'(\lambda) = \mathcal{O}\left(\lambda^{-2}\right), \quad \lambda \to \infty.
\endeq
Momentarily suppose that the contour $\Sigma$ is known and has endpoints 
$a\equiv a(\chi,\tau)$ and $a^*\equiv a(\chi,\tau)^*$.  We orient $\Sigma$ from 
$a^*$ to $a$.  Define 
\eq
R(\lambda) := ((\lambda-a)(\lambda-a^{*}))^{1/2}
\endeq
chosen with branch cut $\Sigma$ and asymptotic behavior 
$R(\lambda)=\lambda+\mathcal{O}(1)$ as $\lambda \to \infty$.  
Then, by the Plemelj formula we have
\eq 
\label{g'function}
g'(\lambda)=\frac{R(\lambda)}{2\pi i}\int_{\Sigma}\frac{2i \chi+4is\tau+\frac{2}{s-\xi^*}-\frac{2}{s-\xi}}{R_{+}(s)(s-\lambda)}ds.
\endeq
These integrals can be calculated explicitly via residues by turning the path 
integral along $\Sigma$ into an integral along a large closed loop, yielding 
\eq \label{g'-function-new}
g'(\lambda)=\frac{R(\lambda)}{R(\xi^{*})(\xi^{*}-\lambda)}-\frac{R(\lambda)}{R(\xi)(\xi-\lambda)}-2i\tau R(\lambda)+i\chi+2i\tau\lambda+\frac{1}{\lambda-\xi^*}-\frac{1}{\lambda-\xi}.
\endeq
Imposing the normalization condition \eqref{g'-condition-3}, we require the terms 
propotional to $\lambda^0$ and $\lambda^{-1}$ in the large-$\lambda$ expansion of 
\eqref{g'-function-new} to be zero:
\eq 
\label{a-condition-1}
\mathcal{O}(1):\,\chi+\tau(a+a^{*})+\frac{i}{R(\xi^{*})}-\frac{i}{R(\xi)}=0,
\endeq
\eq 
\label{a-condition-2}
\mathcal{O}(\lambda^{-1}):\,\frac{\chi}{2}(a+a^{*})+\tau\left(\frac{3}{4}(a+a^*)^2-aa^*\right)+\frac{i\xi^{*}}{R(\xi^{*})}-\frac{i\xi}{R(\xi)}=0.
\endeq
Multiplying \eqref{a-condition-1} by $\xi^{*}$ and using it to eliminate 
$\frac{i\xi^{*}}{R(\xi^{*})}$ in \eqref{a-condition-2}, we have
\eq 
\label{a-condition-4}
\chi\left(\frac{S}{2}-\alpha+i\beta\right)+\tau\left(\frac{3}{4}S^2-P-(\alpha-i\beta)S\right) = \frac{-2\beta}{(P-(\alpha+i\beta)S+(\alpha+i\beta)^2)^{1/2}},
\endeq
where we have written $\xi=\alpha+i\beta$ and defined
\eq
S:=a+a^*, \quad P:=aa^*.
\endeq
Square both sides of equation \eqref{a-condition-4} and clear the denominator. 
Noting that the quantities $\chi$, $\tau$, $S$, $P$, $\alpha$, and $\beta$ are 
all real, we see that the imaginary part is zero if 
\eq
\label{P-from-S}
P = \frac{8(\alpha^2+\beta^2)\tau(S\tau+\chi)+(S-2\alpha)(3St+2\chi)^2}{4\tau(3S\tau+2\chi-2\alpha\tau)}.
\endeq
Plugging this value for $P$ into the real part gives a septic equation for $S$, 
which we do not record here.  This septic equation has three complex-conjugate 
pairs of roots and one real root, which is $S$.   We can then compute $P$ from 
\eqref{P-from-S}, and finally compute $a$ from the known values of $P$ and $S$.  
{Since $g'(\lambda)$ is integrable at $\lambda=\infty$}, the function $g(\lambda)$ is now defined by 
\eq
g(\lambda) := \int_\infty^\lambda g'(s)ds,
\endeq
where the path of integration does not pass through $\Sigma$. 
{Although this determines $g$ as the unique antiderivative that satisfies $g(\lambda)=\mathcal{O}(\lambda^{-1})$, it is more convenient to determine the value of the (integration) constant $K$ that appears in the jump condition \eqref{g-condition-1} by a different calculation. The very same $g$-function and its different variations recently played a central role in the asymptotic analysis of high-order rogue waves in a work \cite{BilmanM:2021} by one of the authors with Miller, and we will use the approach taken there.
Before doing this, we proceed with finalizing the choice of $\Sigma$.
}

From 
\eqref{g'function} we see that redefining $\Sigma$ changes the branch 
cut of $R(\lambda)$ but only changes $g'(\lambda)$ (and thus $g(\lambda)$) by 
an overall sign.  Therefore, the choice of $\Sigma$ does not change the contours on 
which $\Re(\varphi(\lambda)-g(\lambda))=0$.  We thus redefine $\Sigma$ to be 
the unique simple contour from $a^*$ to $a$ on which 
$\Re(\varphi(\lambda)-g(\lambda))=0$ and for which 
$\Re(\varphi(\lambda)-g(\lambda))$ is positive to either side in the upper 
half-plane and negative to both sides in the lower half-plane.  The following 
lemma shows that such a choice is possible and furthermore gives the 
necessary facts about $\varphi(\lambda)-g(\lambda)$ we will need to carry out 
the steepest-descent analysis.
\begin{figure}[t]
\begin{center}
\includegraphics[height=2in]{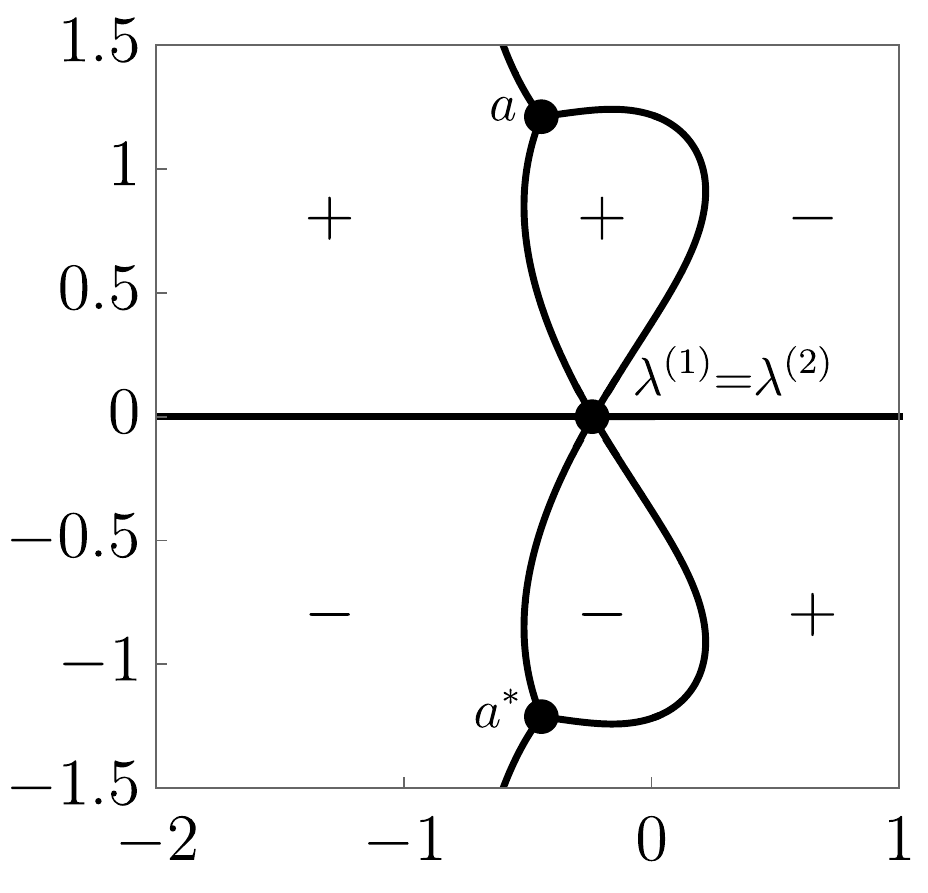} \\
\includegraphics[height=2in]{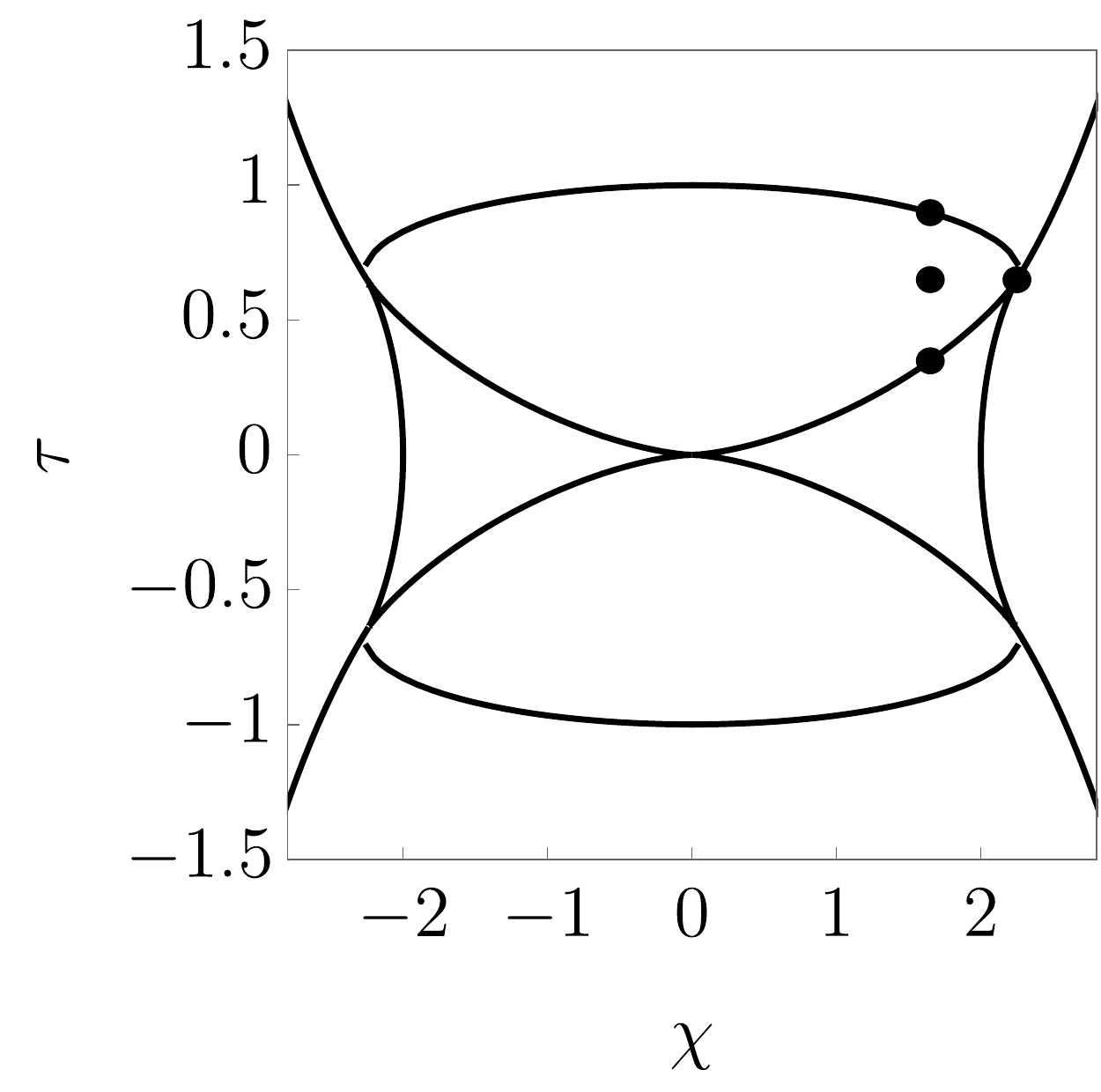} 
\includegraphics[height=2in]{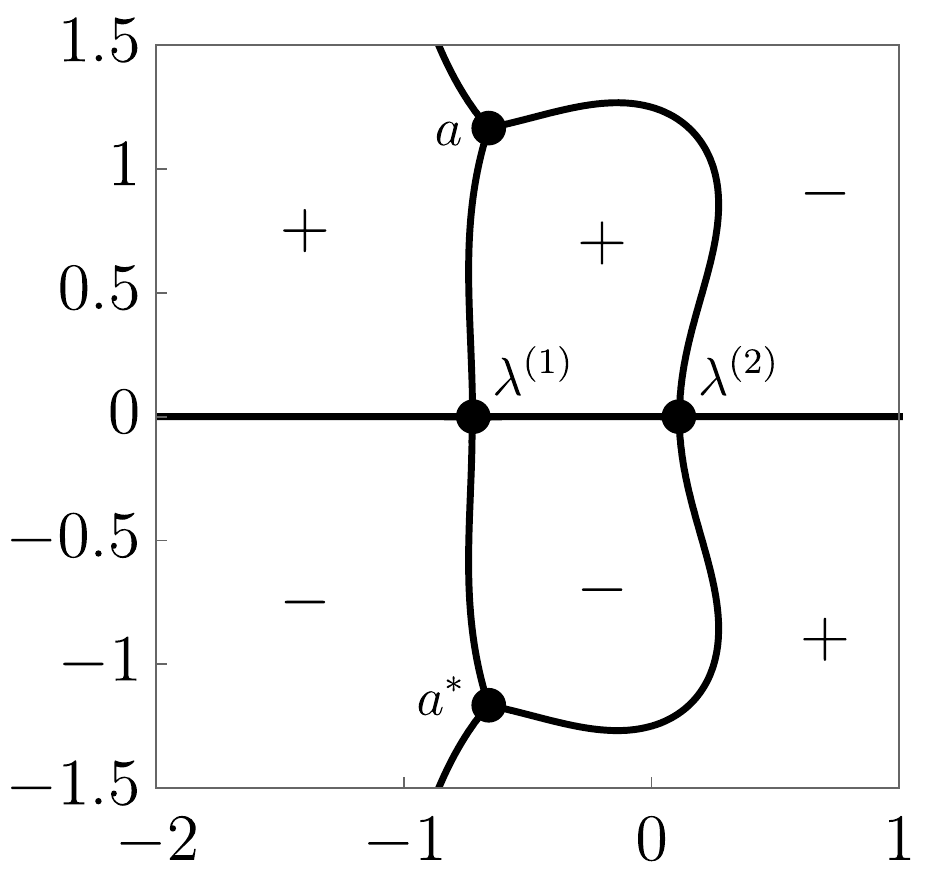} 
\includegraphics[height=2in]{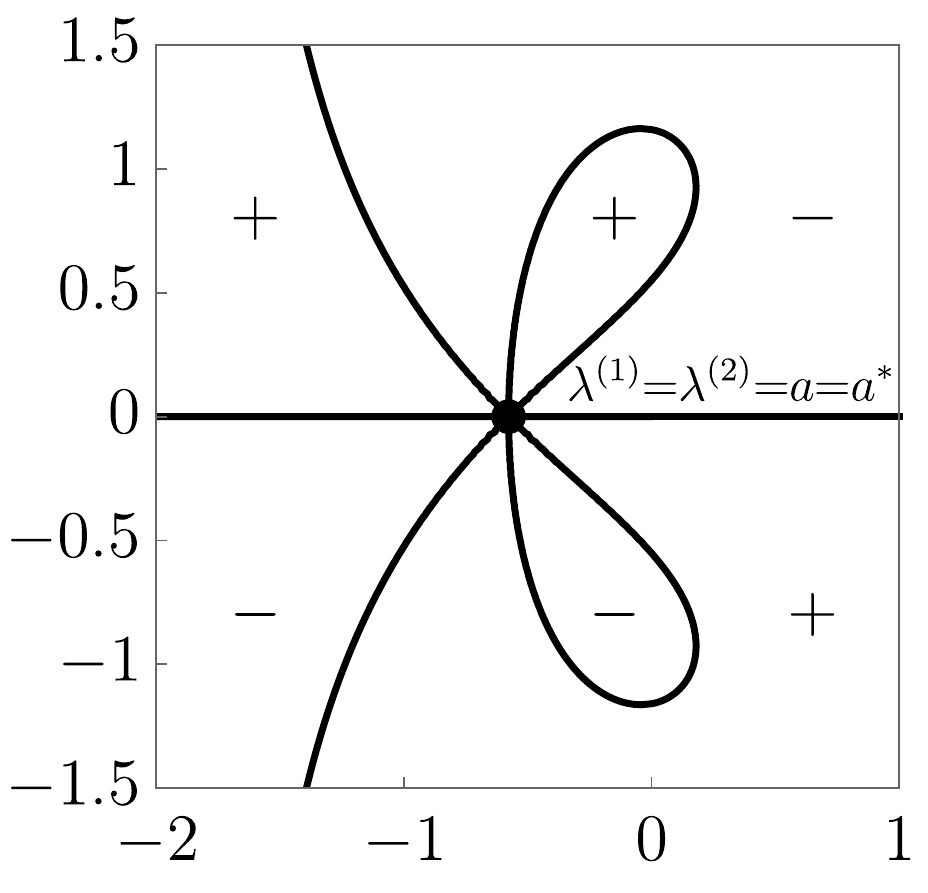} \\
\includegraphics[height=2in]{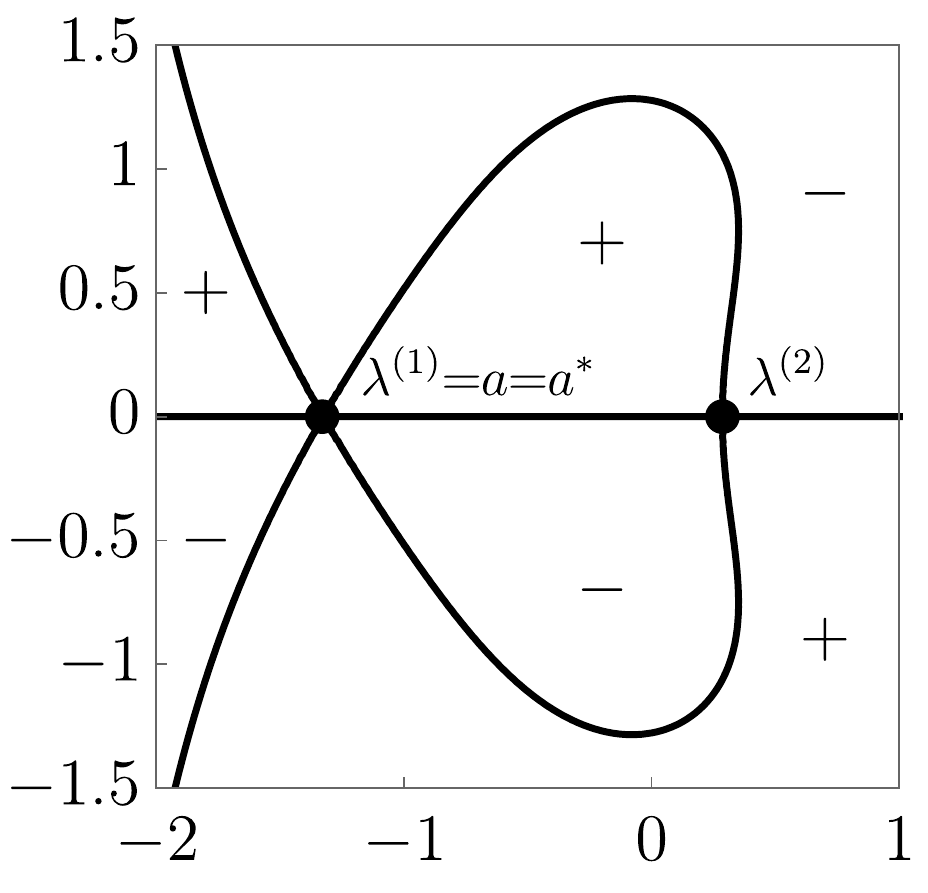}
\caption{Signature charts of 
$\Re(\varphi(\lambda;\chi,\tau)-g(\lambda;\chi,\tau))$ for $\xi=i$ in the 
non-oscillatory region, along with the critical points $\lambda^{(1)}$ and 
$\lambda^{(2)}$ and the band endpoints $a$ and $a^*$.  \emph{Top}: $\chi=1.65$, 
$\tau\approx 0.8983$.  \emph{Center left}:  Positions in the ($\chi$,$\tau$)-plane 
relative to the boundary curves.  \emph{Center}: $\chi=1.65$, $\tau=0.65$.  
\emph{Center right}: $\chi=\frac{9}{4}$, $\tau=\frac{3\sqrt{3}}{8}$.  
\emph{Bottom}:  $\chi=1.65$, $\tau\approx 0.3488$.}
\label{nonoscillatory-phase-plots}
\end{center}
\end{figure}
\begin{lemma}
\label{nonosc-lemma}
In the non-oscillatory region, there is a domain $D_{\rm up}$ in the upper half-plane 
with the following properties:
\begin{itemize}
\item $D_{\rm up}$ contains $\xi$, is bounded by curves along which 
$\Re(\varphi(\lambda)-g(\lambda))=0$, and abuts the real axis along a single 
interval denoted by $(\lambda^{(1)},\lambda^{(2)})$.  
\item $\Re(\varphi(\lambda)-g(\lambda))>0$ for all $\lambda\in D_{\rm up}$.
\item One arc of the boundary of $D_{\rm up}$ is the contour 
$\Sigma_{\rm up}:=\Sigma\cap\mathbb{C}^+$ from $\lambda^{(1)}$ to $a$, along which 
$\Re(\varphi(\lambda)-g(\lambda))>0$ for any $\lambda$ sufficiently close to 
either side of $\Sigma_{\rm up}$.  
\item The remaining boundary of $D_{\rm up}$ in the upper half-plane is a contour 
from $a$ to $\lambda^{(2)}$ (denoted $\Gamma_{\rm up}$) along which 
$\Re(\varphi(\lambda)-g(\lambda))<0$ for any $\lambda$ in the exterior of 
$\overline{D_{\rm up}}$ but sufficiently close to $D_{\rm up}$.  
\end{itemize}
The domain $D_{\rm down}$ in the lower half-plane, defined as the reflection through 
the real axis of $D_{\rm up}$, has the following properties:
\begin{itemize}
\item $D_{\rm down}$ contains $\xi^*$, is bounded by curves along which 
$\Re(\varphi(\lambda)-g(\lambda))=0$, and abuts the real axis along the same 
interval as $D_{\rm up}$.  
\item $\Re(\varphi(\lambda)-g(\lambda))<0$ for all $\lambda\in D_{\rm down}$.
\item One arc of the boundary of $D_{\rm down}$ is the contour  
$\Sigma_{\rm down}:=\Sigma\cap\mathbb{C}^-$ from $a^*$ to $\lambda^{(1)}$, along 
which $\Re(\varphi(\lambda)-g(\lambda))<0$ for any $\lambda$ sufficiently close to 
either side of $\Sigma_{\rm down}$.  
\item The remaining boundary of $D_{\rm down}$ in the lower half-plane is a contour 
from $\lambda^{(2)}$ to $a^*$ (denoted $\Gamma_{\rm down}$) along which 
$\Re(\varphi(\lambda)-g(\lambda))>0$ for any $\lambda$ in the exterior of 
$\overline{D_{\rm down}}$ but sufficiently close to $D_{\rm down}$.  
\end{itemize}
\end{lemma}
\begin{proof}
From \eqref{phi-def} and \eqref{g'-function-new} we see that 
\eq
\varphi'(\lambda)-g'(\lambda) = R(\lambda)\left(2i\tau - \frac{1}{R(\xi^{*})(\xi^{*}-\lambda)} + \frac{1}{R(\xi)(\xi-\lambda)} \right).
\endeq
From here we see that $\phi'(\lambda)-g'(\lambda)$ has two square-root branch 
points at $a$ and $a^*$.  Setting the term in parentheses equal to zero and 
rewriting as a quadratic expression in $\lambda$, we see 
$\phi'(\lambda)-g'(\lambda)$ also has two other zeros that we label as 
$\lambda^{(1)}$ and $\lambda^{(2)}$.  The fact that $\lambda^{(1)}$ and 
$\lambda^{(2)}$ must be real, as well as the topological structure of the 
signature chart of $\Re(\varphi(\lambda)-g(\lambda))$, follows from analytic 
continuation from the boundary curve $\mathcal{L}_\text{AN}$ (at which 
$g(\lambda)\equiv 0$).  See Figure \ref{nonoscillatory-phase-plots}.  
\end{proof}

{We now revisit the jump condition \eqref{g-condition-1} and proceed with the determination of the constant $K$. Note that the endpoints $\lambda=a$ and $\lambda=a^*$ of $\Sigma$ have already been determined in the earlier construction.
Recall that $g(\lambda)$ is analytic for $\lambda\in \mathbb{C}\setminus \Sigma$ with $g(\lambda)=\mathcal{O}(\lambda^{-1})$ as $\lambda\to \infty$. The fact that $\xi$ is contained in the region $D_\mathrm{up}$ and that $\Sigma_\mathrm{up}$ is a subset of the boundary of $D_\mathrm{up}$ ensures that $\Sigma \cap \Sigma_{\mathrm{c}} = \varnothing$. Thus, we may proceed as in \cite{BilmanM:2021} and express $g(\lambda)$ in the form $g(\lambda)=R(\lambda) k(\lambda)$, where $k(\lambda)$ is necessarily analytic for $\lambda\in\mathbb{C}\setminus \Sigma$ with continuous boundary values except at the endpoints $\lambda=a,a^*$ where $g(\lambda)$ is required to be bounded. Then, requiring $k(\lambda)=\mathcal{O}(\lambda^{-2})$ as  $\lambda\to\infty$, \eqref{g-condition-1} implies that
\begin{equation}
k_+(\lambda) - k_-(\lambda) = \frac{2\varphi(\lambda) - 2i K}{R_+(\lambda)},\quad \lambda\in  \Sigma,
\end{equation}
hence the Plemelj formula gives
\begin{equation}
\label{k-fun-def}
k(\lambda) = \frac{1}{i \pi }\int_{\Sigma}  \frac{\varphi(s) - i K}{R_+(s)(s-\lambda)}\,ds.
\end{equation}
Enforcing the condition $k(\lambda)=\mathcal{O}(\lambda^{-2})$ as  $\lambda\to\infty$ in the representation \eqref{k-fun-def} results in the condition
\begin{equation}
\label{K-condition}
\int_{\Sigma}\frac{\varphi(\lambda) - i K }{R_+(\lambda)}\,d\lambda = 0.
\end{equation}
First, recall that $R(\lambda)=\lambda+\mathcal{O}(1)$ as $\lambda\to\infty$. Thus, for an arbitrary clockwise-oriented loop $C$ surrounding the branch cut $\Sigma$ of $R(\lambda)$ we can obtain by a residue calculation at $\lambda=\infty$:
\begin{equation}
\label{K-integral-1}
\int_{\Sigma}\frac{d\lambda}{R_+(\lambda)} = \frac{1}{2}\oint_C \frac{d\lambda}{R(\lambda)} = -i \pi.
\end{equation}
As the integral above is nonzero, the condition \eqref{K-condition} successfully determines the constant $K$. The remaining integral
\begin{equation}
\label{K-integral-2}
\int_{\Sigma} \frac{\varphi(\lambda) }{R_+(\lambda)}\,d\lambda  = 
\int_{\Sigma} \frac{i(\chi \lambda + \tau \lambda^2) }{R_+(\lambda)}\,d\lambda +
\int_{\Sigma} \frac{\log\left( \frac{\lambda-\xi^*}{\lambda- \xi}\right) }{R_+(\lambda)}\,d\lambda
\end{equation}
in \eqref{K-condition} can also be computed similarly. Using the expansion
\begin{equation}
R(\lambda)^{-1} = \lambda^{-1} + \frac{1}{2}(a+a^*) \lambda^{-2} + \frac{1}{4}\left( (a+a^*)^2 + \frac{1}{2}(a-a^*)^2\right) \lambda^{-3} + \mathcal{O}(\lambda^{-4}), \quad \lambda\to \infty,
\end{equation}
we find that
\begin{equation}
\label{K-integral-2-val-1}
\int_{\Sigma} \frac{i(\chi \lambda + \tau \lambda^2) }{R_+(\lambda)}\,d\lambda =
\pi \left[\frac{1}{2}\chi(a+a^*) + \frac{1}{4}\tau \left( (a+a^*)^2 + \frac{1}{2}(a-a^*)^2\right)  \right].
\end{equation}
Next, to evaluate the second integral on the right-hand side of \eqref{K-integral-2} we again let $C$ be a clockwise-oriented loop surrounding the branch cut $\Sigma$ of $R(\lambda)$ but excluding the branch cut $\Sigma_\mathrm{c}$ of the logarithm in the integrand. Then, since the integrand is integrable at $\lambda=\infty$, letting $C'$ be a counter-clockwise oriented contour that surrounds $\Sigma_\mathrm{c}$ but that excludes $\Sigma$ yields
\begin{equation}
\int_{\Sigma} \frac{\log\left( \frac{\lambda-\xi^*}{\lambda- \xi}\right) }{R_+(\lambda)}\,d\lambda =
\frac{1}{2} \oint_{C} \frac{\log\left( \frac{\lambda-\xi^*}{\lambda- \xi}\right) }{R(\lambda)}\,d\lambda =
\frac{1}{2} \oint_{C'} \frac{\log\left( \frac{\lambda-\xi^*}{\lambda- \xi}\right) }{R(\lambda)}\,d\lambda.
\end{equation}
Now, recalling that $R(\lambda)$ is analytic on $\Sigma_\mathrm{c}$, we may collapse the contour $C'$ to both sides of $\Sigma_\mathrm{c}$ and use the fact that the boundary values of the logarithm differ by $2\pi i$ on $\Sigma_\mathrm{c}$ to obtain
\begin{equation}
\label{K-integral-2-val-2}
\int_{\Sigma} \frac{\log\left( \frac{\lambda-\xi^*}{\lambda- \xi}\right) }{R_+(\lambda)}\,d\lambda =
i \pi \int_{\Sigma_{\mathrm{c}}} \frac{1}{R(\lambda)}\,d\lambda.
\end{equation}
Combining \eqref{K-integral-2-val-1} and \eqref{K-integral-2-val-2} gives
\begin{equation}
\int_{\Sigma} \frac{\varphi(\lambda) }{R_+(\lambda)}\,d\lambda  = 
\pi \left[\frac{1}{2}\chi(a+a^*) + \frac{1}{4}\tau \left( (a+a^*)^2 + \frac{1}{2}(a-a^*)^2\right)  \right] + 
i \pi \int_{\Sigma_{\mathrm{c}}} \frac{1}{R(\lambda)}\,d\lambda,
\end{equation}
which, together with \eqref{K-integral-2} results in
\begin{equation}
K(\chi,\tau) =  \left[\frac{1}{2}\chi(a+a^*) + \frac{1}{4}\tau \left( (a+a^*)^2 + \frac{1}{2}(a-a^*)^2\right)  \right] + 
i \int_{\Sigma_{\mathrm{c}}} \frac{1}{R(\lambda)}\,d\lambda,
\label{K-def}
\end{equation}
which is real-valued.
}

\begin{figure}[h]
\begin{center}
\includegraphics[height=2in]{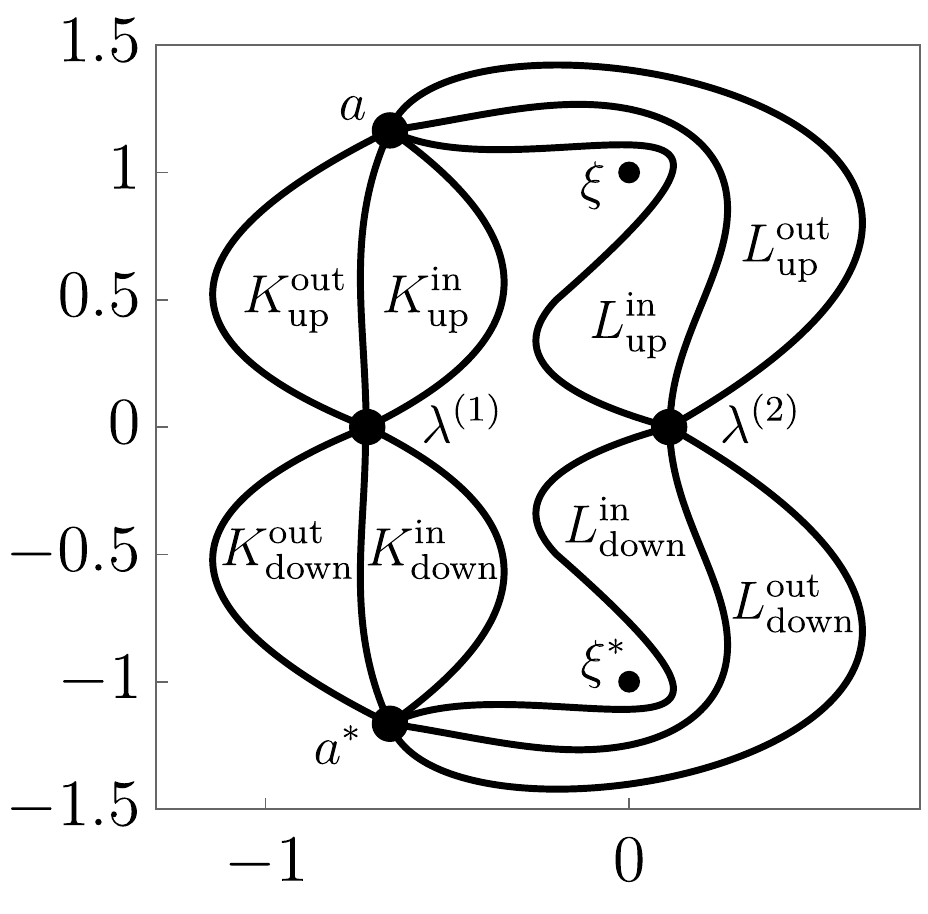} 
\includegraphics[height=2in]{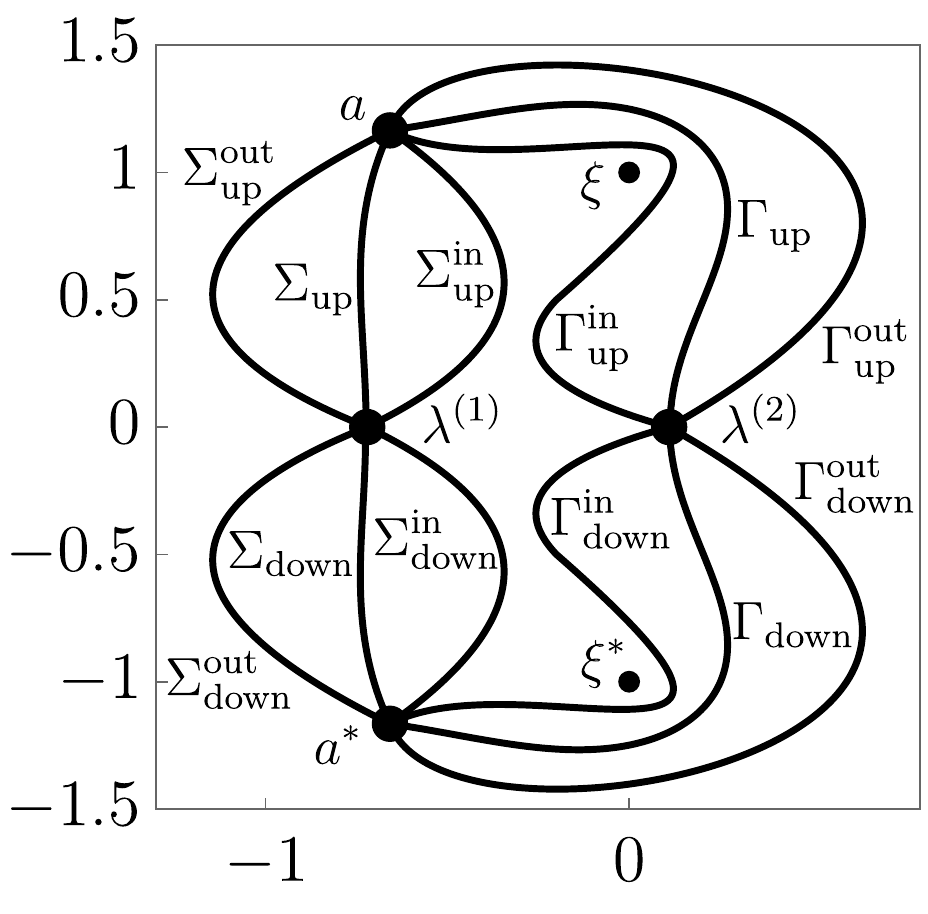} 
\caption{The domains (left) and contours (right) used in the definition of 
${\bf Q}^{[n]}(\lambda)$ in the non-oscillatory region.}
\label{nonosc-lenses}
\end{center}
\end{figure}

We are now ready to carry out the first Riemann-Hilbert transformation.  Let 
the domain $D$ be the union of $D_{\rm up}$, $D_{\rm down}$, and 
the interval $(\lambda^{(1)},\lambda^{(2)})$.  Note $D$ is bounded by 
$\Sigma_\text{up}\cup\Gamma_\text{up}\cup\Gamma_\text{down}\cup\Sigma_\text{down}$.
Recall the function ${\bf N}^{[n]}(\lambda)$ satisfying 
Riemann-Hilbert Problem \ref{rhp-N} and make the change of variables 
\eq
{\bf O}^{[n]}(\lambda;\chi,\tau):= \begin{cases} {\bf N}^{[n]}(\lambda;\chi,\tau){\bf V}_{\bf N}^{[n]}(\lambda;\chi,\tau), & \lambda\in D_0\cap D^\mathsf{c}, \\ {\bf N}^{[n]}(\lambda;\chi,\tau){\bf V}_{\bf N}^{[n]}(\lambda;\chi,\tau)^{-1}, & \lambda\in D_0^\mathsf{c}\cap D, \\ {\bf N}^{[n]}(\lambda;\chi,\tau), & \text{otherwise}.
\end{cases}
\endeq
Now ${\bf O}^{[n]}(\lambda)$ satisfies the same Riemann-Hilbert problem as 
${\bf N}^{[n]}(\lambda)$ with the jump contour $\partial D_0$ replaced by 
$\partial D$.  Next, we introduce the $g$-function via 
\eq \label{g-transform}
{\bf P}^{[n]}(\lambda;\chi,\tau):={\bf O}^{[n]}(\lambda;\chi,\tau)e^{-ng(\lambda;\chi,\tau)\sigma_3}.
\endeq
The jump condition for ${\bf P}^{[n]}(\lambda)$ is now 
\eq
\label{P-jumps-nonosc}
{\bf P}_{+}^{[n]}(\lambda) = {\bf P}_{-}^{[n]}(\lambda)e^{-n(\varphi(\lambda)-g_{-}(\lambda))\sigma_3}\mathcal{S}^{-1}e^{n(\varphi(\lambda)-g_{+}(\lambda))\sigma_3}, \quad \lambda\in\partial D.
\endeq
We define the following contours:
\begin{itemize}
\item $\Sigma_\text{up}^\text{out}$ runs from $\lambda^{(1)}$ to $a$ in 
the upper half-plane entirely in the region exterior to $D$ in which 
$\Re(\varphi(\lambda)-g(\lambda))>0$.
\item $\Sigma_\text{up}^\text{in}$ runs from $\lambda^{(1)}$ to $a$ 
entirely in $D_{\rm up}$ (so $\Re(\varphi(\lambda)-g(\lambda))>0$), and can be 
deformed to $\Sigma_\text{up}$ without passing through $\xi$.
\item $\Gamma_\text{up}^\text{out}$ runs from $a$ to $\lambda^{(2)}$ 
in the upper half-plane entirely in the region where 
$\Re(\varphi(\lambda)-g(\lambda))<0$.
\item $\Gamma_\text{up}^\text{in}$ runs from $a$ to $\lambda^{(2)}$ 
entirely in $D_{\rm up}$ (so $\Re(\varphi(\lambda)-g(\lambda))>0$), and can be 
deformed to $\Gamma_\text{up}$ without passing through $\xi$.
\item $\Sigma_\text{down}^\text{out}$ (oriented from $a^*$ to $\lambda^{(1)}$), 
$\Sigma_\text{down}^\text{in}$ (oriented from $a^*$ to $\lambda^{(1)}$), 
$\Gamma_\text{down}^\text{out}$ (oriented from $\lambda^{(2)}$ to $a^*$), 
and $\Gamma_\text{down}^\text{in}$ (oriented from $\lambda^{(2)}$ to $a^*$) 
are the reflections through the real axis of $\Sigma_\text{up}^\text{out}$, 
$\Sigma_\text{up}^\text{in}$, $\Gamma_\text{up}^\text{out}$, and 
$\Gamma_\text{up}^\text{in}$, respectively.
\end{itemize}
Define the following eight domains:
\begin{itemize}
\item $K_\text{up}^\text{out}$ (respectively, $K_\text{up}^\text{in}$) is the 
domain in the upper half-plane bounded by $\Sigma_\text{up}^\text{out}$ 
(respectively, $\Sigma_\text{up}^\text{in}$) and $\Sigma_\text{up}$.  
\item $L_\text{up}^\text{out}$ (respectively, $L_\text{up}^\text{in}$) is the 
domain in the upper half-plane bounded by $\Gamma_\text{up}^\text{out}$ 
(respectively, $\Gamma_\text{up}^\text{in}$) and $\Gamma_\text{up}$.  
\item $K_\text{down}^\text{out}$, $K_\text{down}^\text{in}$, 
$L_\text{down}^\text{out}$, and $L_\text{down}^\text{in}$ are the reflections 
through the real axis of $K_\text{up}^\text{out}$, $K_\text{up}^\text{in}$, 
$L_\text{up}^\text{out}$, and $L_\text{up}^\text{in}$, respectively.
\end{itemize}
See Figure \ref{nonosc-lenses}.
On $\Sigma$ we will use the following alternative factorizations of 
$\mathcal{S}^{-1}$:
\eq
\begin{split}
\label{Sinv-on-Sigma}
\mathcal{S}^{-1} & = \bbm 1 & -\frac{c_1^*}{c_2} \\ 0 & 1 \ebm \bbm 0 & \frac{|{\bf c}|}{c_2} \\ -\frac{c_2}{|{\bf c}|}  & 0 \ebm \bbm 1 & -\frac{c_1}{c_2} \\ 0 & 1 \ebm \quad\quad (\text{use for }\lambda\in\Sigma_\text{up}), \\
\mathcal{S}^{-1} & = \bbm 1 & 0 \\ \frac{c_1}{c_2^*} & 1 \ebm \bbm 0 & \frac{c_2^*}{|{\bf c}|} \\ -\frac{|{\bf c}|}{c_2^*} & 0\ebm \bbm 1 & 0 \\ \frac{c_1^*}{c_2^*} & 1 \ebm \quad\quad (\text{use for }\lambda\in\Sigma_\text{down}).
\end{split}
\endeq
We open lenses by defining 
\eq
\label{Q-def-nonosc}
{\bf Q}^{[n]}(\lambda;\chi,\tau) := \begin{cases}
{\bf P}^{[n]}(\lambda;\chi,\tau)\bbm 1 & -\frac{c_1^*}{c_2}e^{-2n(\varphi(\lambda;\chi,\tau)-g(\lambda;\chi,\tau))} \\ 0 & 1 \ebm, & \lambda\in K_\text{up}^\text{in}, \\
{\bf P}^{[n]}(\lambda;\chi,\tau)\bbm 1 & -\frac{c_1}{c_2}e^{-2n(\varphi(\lambda;\chi,\tau)-g(\lambda;\chi,\tau))} \\ 0 & 1 \ebm^{-1}, & \lambda\in K_\text{up}^\text{out},\vspace{.025in} \\
{\bf P}^{[n]}(\lambda;\chi,\tau)\bbm 1 & 0 \\ \frac{c_1}{c_2^*}e^{2n(\varphi(\lambda;\chi,\tau)-g(\lambda;\chi,\tau))} & 1 \ebm, & \lambda\in K_\text{down}^\text{in}, \\
{\bf P}^{[n]}(\lambda;\chi,\tau)\bbm 1 & 0 \\ \frac{c_1^*}{c_2^*}e^{2n(\varphi(\lambda;\chi,\tau)-g(\lambda;\chi,\tau))} & 1 \ebm^{-1}, & \lambda\in K_\text{down}^\text{out},\vspace{.025in} \\
{\bf P}^{[n]}(\lambda;\chi,\tau)\bbm 1 & \frac{c_2^*}{c_1}e^{-2n(\varphi(\lambda;\chi,\tau)-g(\lambda;\chi,\tau))} \\ 0 & 1 \ebm, & \lambda\in L_\text{up}^\text{in}, \\
{\bf P}^{[n]}(\lambda;\chi,\tau)\bbm 1 & 0 \\ -\frac{c_2}{c_1}e^{2n(\varphi(\lambda;\chi,\tau)-g(\lambda;\chi,\tau))} & 1 \ebm^{-1}, & \lambda\in L_\text{up}^\text{out},\vspace{.025in} \\
{\bf P}^{[n]}(\lambda;\chi,\tau)\bbm 1 & 0 \\ -\frac{c_2}{c_1^*}e^{2n(\varphi(\lambda;\chi,\tau)-g(\lambda;\chi,\tau))} & 1 \ebm, & \lambda\in L_\text{down}^\text{in}, \\
{\bf P}^{[n]}(\lambda;\chi,\tau)\bbm 1 & \frac{c_2^*}{c_1^*}e^{-2n(\varphi(\lambda;\chi,\tau)-g(\lambda;\chi,\tau))} \\ 0 & 1 \ebm^{-1}, & \lambda\in L_\text{down}^\text{out}, \\
{\bf P}^{[n]}(\lambda;\chi,\tau), & \text{otherwise}. \end{cases}
\endeq
Using \eqref{P-jumps-nonosc}, \eqref{Sinv-on-Sigma}, and \eqref{Sinv-on-Gamma}, 
we see that ${\bf Q}^{[n]}(\lambda)$ satisfies the jumps  
${\bf Q}_+^{[n]}(\lambda)={\bf Q}_-^{[n]}(\lambda){\bf V}_{\bf Q}^{[n]}(\lambda)$,
where the jumps on the various contours are given by
\eq
\begin{alignedat}{3}
&\Sigma_\text{up} :\,\,\bbm 0 & \frac{|{\bf c}|}{c_2} {e^{-2i n K} } \\ -\frac{c_2}{|{\bf c}|}{e^{2i n K}}  & 0 \ebm,\quad
&&\Sigma_\text{down}:\,\,\bbm 0 & \frac{c_2^*}{|{\bf c}|} {e^{-2i n K}} \\ -\frac{|{\bf c}|}{c_2^*} {e^{2i n K}} & 0\ebm,\quad
&&\Gamma_\text{up}:\,\,\bbm \frac{|{\bf c}|}{c_1} & 0\\ 0 & \frac{c_1}{|{\bf c}|} \ebm,  
\\
&\Gamma_\text{down}:\,\,\bbm \frac{c_1^{*}}{|{\bf c}|}& 0\\ 0 & \frac{|{\bf c}|}{c_1^{*}}  \ebm,  \quad
&&\Sigma_\text{up}^\text{in}:\,\,\bbm 1& -\frac{c_1^*}{c_2}e^{-2n(\varphi-g)}\\ 0 & 1  \ebm,\quad
&&\Sigma_\text{up}^\text{out}:\,\,\bbm 1& -\frac{c_1}{c_2}e^{-2n(\varphi-g)}\\ 0 & 1  \ebm,
\\
&\Sigma_\text{down}^\text{in}:\bbm 1& 0\\ \frac{c_1}{c_2^*}e^{2n(\varphi-g)} & 0  \ebm, \quad
&&\Sigma_\text{down}^\text{out}:\,\,\bbm 1& 0\\ \frac{c_1^{*}}{c_2^*}e^{2n(\varphi-g)} & 0  \ebm,\quad
&&\Gamma_\text{up}^\text{in}:\,\,\bbm 1& \frac{c_2^{*}}{c_1}e^{-2n(\varphi-g)}\\ 0 & 1  \ebm,\\
\hspace{.2in}
&\Gamma_\text{up}^\text{out}:\,\,\bbm 1& 0\\ -\frac{c_2}{c_1}e^{2n(\varphi-g)} & 0  \ebm,\quad
&&\Gamma_\text{down}^\text{in}:\,\,\bbm 1& 0\\ -\frac{c_2}{c_1^*}e^{2n(\varphi-g)} & 0  \ebm,\quad
&&\Gamma_\text{down}^\text{out}:\,\,\bbm 1& \frac{c_2^{*}}{c_1^{*}}e^{-2n(\varphi-g)}\\ 0 & 1 \ebm.
\hspace{1in}
\end{alignedat}{3}
\endeq
Lemma \ref{nonosc-lemma} shows that, except for the four constant jumps, all of 
the jumps decay exponentially to the identity for $\lambda$ bounded away from 
$a$, $a^*$, $\lambda^{(1)}$, and $\lambda^{(2)}$.  We are thus ready to define 
the outer model Riemann-Hilbert problem.
\begin{rhp}[The outer model problem in the non-oscillatory region]
Fix a pole location $\xi\in\mathbb{C}^+$, a pair of nonzero complex numbers 
$(c_1,c_2)$, and 
a pair of real numbers $(\chi,\tau)$ in the non-oscillatory region.  Determine 
the $2\times 2$ matrix ${\bf R}^{(\infty)}(\lambda;\chi,\tau)$ with the following 
properties:
\begin{itemize}
\item[]\textbf{Analyticity:}  ${\bf R}^{(\infty)}(\lambda;\chi,\tau)$
is analytic for $\lambda\in \mathbb{C}$ except on $\Sigma_{\rm up}\cup\Sigma_{\rm down}\cup\Gamma_{\rm up}\cup\Gamma_{\rm down}$,
where it achieves continuous boundary values on the interior of each arc.
\item[]\textbf{Jump condition:}  The boundary values taken by ${\bf R}^{(\infty)}(\lambda;\chi,\tau)$ are
related by the jump conditions
${\bf R}^{(\infty)}_+(\lambda;\chi,\tau)={\bf R}^{(\infty)}_-(\lambda;\chi,\tau){\bf V}_{\bf R}^{(\infty)}(\lambda;\chi,\tau)$,
where
\eq
{\bf V}_{\bf R}^{(\infty)}(\lambda;\chi,\tau) := \begin{cases} \bbm 0 & \frac{|{\bf c}|}{c_2} {e^{-2i n K}} \\ -\frac{c_2}{|{\bf c}|} {e^{2i n K}} & 0\ebm, & \lambda\in \Sigma_{\rm up}, \vspace{.025in} \\ \bbm 0 & \frac{c_2^*}{|{\bf c}|} {e^{-2i n K}} \\ -\frac{|{\bf c}|}{c_2^*} {e^{ 2i n K}} & 0\ebm, & \lambda\in \Sigma_{\rm down},\vspace{.025in} \\
\bbm \frac{|{\bf c}|}{c_1} & 0 \\ 0 & \frac{c_1}{|{\bf c}|} \ebm, & \lambda\in \Gamma_{\rm up}, \vspace{.025in} \\ \bbm \frac{c_1^*}{|{\bf c}|} & 0 \\ 0 & \frac{|{\bf c}|}{c_1^*} \ebm, & \lambda\in \Gamma_{\rm down}.
\end{cases}
\endeq
\item[]\textbf{Normalization:}  As $\lambda\to\infty$, the matrix ${\bf R}^{(\infty)}(\lambda;\chi,\tau)$
satisfies the condition
\begin{equation}
{\bf R}^{(\infty)}(\lambda;\chi,\tau) = \mathbb{I}+\mathcal{O}(\lambda^{-1})
\end{equation}
with the limit being uniform with respect to direction.
\end{itemize}
\label{rhp:R-model}
\end{rhp}
The first step in solving for ${\bf R}^{(\infty)}(\lambda)$ is to remove the 
dependence on $c_1$ and $c_2$.  Define the function 
\eq
\begin{split}
f(\lambda):=\frac{R(\lambda)}{2\pi i} &\left[\int_{\Sigma_{\rm up}} \frac{\log\left(\frac{c_2}{|{\bf c}|}\right)}{R_{+}(s)(s-\lambda)} ds+ \int_{\Sigma_{\rm down}} \frac{\log\left(\frac{|{\bf c}|}{c_2^{*}}\right)}{R_{+}(s)(s-\lambda)} ds \right.\\ 
 & \hspace*{.2in}+ \left.\int_{\Gamma_{\rm up}} \frac{\log\left(\frac{|{\bf c}|}{c_1}\right)}{R(s)(s-\lambda)} ds +\int_{\Gamma_{\rm down}} \frac{\log\left(\frac{c_1^{*}}{|{\bf c}|}\right)}{R(s)(s-\lambda)} ds \right].
\end{split}
\endeq
Then $f(\lambda)$ satisfies the jump conditions 
\eq
\label{f-jump-nonosc}
\begin{split}
f_+(\lambda) + f_-(\lambda) & = -\log\left(\frac{|{\bf c}|}{c_2}\right),\quad\quad \lambda \in \Sigma_{\rm up}, \\
f_+(\lambda) + f_-(\lambda) & = -\log\left(\frac{c_2^{*}}{|{\bf c}|}\right),\quad\quad \lambda \in \Sigma_{\rm down}, \\
f_+(\lambda) - f_-(\lambda) & = -\log\left(\frac{c_1}{|{\bf c}|}\right),\quad\quad \lambda \in \Gamma_{\rm up}, \\
f_+(\lambda) - f_-(\lambda) & = -\log\left(\frac{|{\bf c}|}{c_1^{*}}\right),\quad\quad \lambda \in \Gamma_{\rm down}, \\
\end{split}
\endeq
and the symmetry
\eq
f(\lambda) = -(f(\lambda^{*}))^*.
\endeq
We also have that $f(\lambda)$ is bounded as $\lambda\to\infty$, and 
\eq
\begin{split}
f(\infty):=\lim_{\lambda\to\infty}f(\lambda) = -\frac{1}{2\pi i}&\left[\int_{\Sigma_{\rm up}} \frac{\log\left(\frac{c_2}{|{\bf c}|}\right)}{R_{+}(s)} ds+ \int_{\Sigma_{\rm down}} \frac{\log\left(\frac{|{\bf c}|}{c_2^{*}}\right)}{R_{+}(s)} ds\right. \\
&\hspace*{.2in}+ \left.\int_{\Gamma_{\rm up}} \frac{\log\left(\frac{|{\bf c}|}{c_1}\right)}{R(s)} ds +\int_{\Gamma_{\rm down}} \frac{\log\left(\frac{c_1^{*}}{|{\bf c}|}\right)}{R(s)} ds  \right].
\end{split}
\label{finf-def}
\endeq
We note $f(\infty)$ is a purely imaginary number.  Introduce
\eq 
\label{f-function-transform}
{\bf S}(\lambda):=e^{f(\infty)\sigma_3}{\bf R}^{(\infty)}(\lambda)e^{-f(\lambda)\sigma_3}.
\endeq
Thus, we have ${\bf S}_+(\lambda)={\bf S}_-(\lambda){\bf V}_{\bf S}(\lambda)$,
where
\eq
{\bf V}_{\bf S}(\lambda) := \begin{cases} \bbm 0 & \frac{|{\bf c}|}{c_2}e^{f_+(\lambda)+f_-(\lambda)} {e^{-2inK}} \\ -\frac{c_2}{|{\bf c}|}e^{-(f_+(\lambda)+f_-(\lambda))}{e^{2inK}} & 0\ebm, & \lambda\in \Sigma_{\rm up}, \vspace{.025in} \\ \bbm 0 & \frac{c_2^*}{|{\bf c}|}e^{f_+(\lambda)+f_-(\lambda)} {e^{-2inK}} \\ -\frac{|{\bf c}|}{c_2^*} e^{-(f_+(\lambda)+f_-(\lambda))} {e^{2inK}} & 0\ebm, & \lambda\in \Sigma_{\rm down},\vspace{.025in} \\
\bbm \frac{|{\bf c}|}{c_1}e^{-(f_+(\lambda)-f_-(\lambda))} & 0 \\ 0 & \frac{c_1}{|{\bf c}|}e^{f_+(\lambda)-f_-(\lambda)} \ebm, & \lambda\in \Gamma_{\rm up}, \vspace{.025in} \\ \bbm \frac{c_1^*}{|{\bf c}|}e^{-(f_+(\lambda)-f_-(\lambda))} & 0 \\ 0 & \frac{|{\bf c}|}{c_1^*}e^{f_+(\lambda)-f_-(\lambda)}  \ebm, & \lambda\in \Gamma_{\rm down}.
\end{cases}
\endeq
From the conditions \eqref{f-jump-nonosc} for $f(\lambda)$ we see the jump 
simplifies to 
\eq
{\bf S}_+(\lambda)={\bf S}_-(\lambda){e^{-inK\sigma_3}}\bbm 0 & 1 \\ -1 & 0\ebm {e^{inK\sigma_3}}, \quad \lambda\in\Sigma.
\endeq
Along with the normalization condition 
${\bf S}(\lambda)=\mathbb{I}+\mathcal{O}(\lambda^{-1})$, this specifies that 
${\bf S}(\lambda)$ must be 
\eq
{\bf S}(\lambda)={e^{- i n K\sigma_3}}
\bbm \displaystyle \frac{\gamma{{(\lambda)}}+\gamma{{(\lambda)}}^{-1}}{2} & \displaystyle \frac{-i\gamma{{(\lambda)}}+i\gamma{{(\lambda)}}^{-1}}{2} \\ \displaystyle \frac{i\gamma{{(\lambda)}}-i\gamma{{(\lambda)}}^{-1}}{2} & \displaystyle \frac{\gamma{{(\lambda)}}+\gamma{{(\lambda)}}^{-1}}{2} \ebm {e^{i n K\sigma_3}},
\endeq
where 
\eq
\label{beta-nonosc}
\gamma(\lambda):=\left(\frac{\lambda-a}{\lambda-a^{*}}\right)^{1/4}
\endeq
is cut on $\Sigma$ and has asymptotic behavior 
$\gamma(\lambda)=1+\mathcal{O}(\lambda^{-1})$ as $\lambda\to\infty$.  Thus, we have
\eq
{\bf R}^{(\infty)}(\lambda)={e^{-i n K\sigma_3}}
\bbm \displaystyle \frac{\gamma{{(\lambda)}}+\gamma{{(\lambda)}}^{-1}}{2}e^{f(\lambda)-f(\infty)} & \displaystyle \frac{\gamma{{(\lambda)}}-\gamma{{(\lambda)}}^{-1}}{2i}e^{-f(\lambda)-f(\infty)}  \\ \displaystyle -\frac{\gamma{{(\lambda)}}-\gamma{{(\lambda)}}^{-1}}{2i}e^{f(\lambda)+f(\infty)} & \displaystyle \frac{\gamma{{(\lambda)}}+\gamma{{(\lambda)}}^{-1}}{2}e^{-f(\lambda)+f(\infty)} \ebm {e^{inK\sigma_3}}.
\endeq
To complete the definition of the global model solution ${\bf R}(\lambda)$, we 
need to define local parametrices ${\bf R}^{(1)}(\lambda)$, 
${\bf R}^{(2)}(\lambda)$, ${\bf R}^{(a)}(\lambda)$, and ${\bf R}^{(a^*)}(\lambda)$ 
in small, fixed disks $\mathbb{D}^{(1)}$, $\mathbb{D}^{(2)}$, $\mathbb{D}^{(a)}$, 
and $\mathbb{D}^{(a^*)}$ centered at $\lambda^{(1)}$, $\lambda^{(2)}$, $a$, and 
$a^*$, respectively.  These local parametrices satisfy two conditions:
\begin{itemize}
\item ${\bf R}^{(\bullet)}(\lambda)$ satisfies the same jump conditions as ${\bf Q}^{[n]}(\lambda)$ for $\lambda\in\mathbb{D}^{(\bullet)}$, where $\bullet\in\{1,2,a,a^*\}$.
\item ${\bf R}^{(\bullet)}(\lambda) = \begin{cases} {\bf R}^{(\infty)}(\lambda)(\mathbb{I}+\mathcal{O}(n^{-1/2})), &  \lambda\in\partial\mathbb{D}^{(\bullet)}, \text{ where }\bullet\in\{1,2\}, \\ {\bf R}^{(\infty)}(\lambda)(\mathbb{I}+\mathcal{O}(n^{-1})), & \lambda\in\partial\mathbb{D}^{(\bullet)}\text{ where }\bullet\in\{a,a^*\}.\end{cases}$
\end{itemize}
While we will not need their explicit form, the parametrices 
${\bf R}^{(1)}(\lambda)$ and ${\bf R}^{(2)}(\lambda)$ can be constructed 
explicitly using parabolic cylinder functions (see, for example, 
\S\ref{sec-alg-decay}), while the parametrices ${\bf R}^{(1)}(\lambda)$ and 
${\bf R}^{(2)}(\lambda)$ can be constructed explicitly using Airy functions (see, 
for example, \cite{DeiftKMVZ:1999}).  Then the function 
\eq
{\bf R}(\lambda) := \begin{cases} {\bf R}^{(1)}(\lambda), & \lambda\in\mathbb{D}^{(1)}, \\ {\bf R}^{(2)}(\lambda), & \lambda\in\mathbb{D}^{(2)}, \\ {\bf R}^{(a)}(\lambda), & \lambda\in\mathbb{D}^{(a)}, \\ {\bf R}^{(a^*)}(\lambda), & \lambda\in\mathbb{D}^{(a^*)}, \\ {\bf R}^{(\infty)}(\lambda), & \text{otherwise} \end{cases}
\endeq
is a valid approximation to ${\bf Q}^{[n]}(\lambda)$ everywhere in the complex 
$\lambda$-plane as $n\to\infty$.  In particular, we have 
\eq
{\bf Q}^{[n]}(\lambda) = \left(\mathbb{I}+\mathcal{O}(n^{-1/2})\right){\bf R}(\lambda).
\endeq
Working our way through the various transformations, we see that, for 
$|\lambda|$ sufficiently large,
\eq
\begin{split}
[{\bf M}^{[n]}&(\lambda;n\chi,n\tau)]_{12}  = \left(\frac{\lambda-\xi^*}{\lambda-\xi}\right)^n[{\bf N}^{[n]}(\lambda;\chi,\tau)]_{12} = \left(\frac{\lambda-\xi^*}{\lambda-\xi}\right)^n[{\bf O}^{[n]}(\lambda;\chi,\tau)]_{12} \\ 
 & = \left(\frac{\lambda-\xi^*}{\lambda-\xi}\right)^ne^{-ng(\lambda;\chi,\tau)}[{\bf P}^{[n]}(\lambda;\chi,\tau)]_{12} = \left(\frac{\lambda-\xi^*}{\lambda-\xi}\right)^ne^{-ng(\lambda;\chi,\tau)}[{\bf Q}^{[n]}(\lambda;\chi,\tau)]_{12} \\
 & = \left(\frac{\lambda-\xi^*}{\lambda-\xi}\right)^ne^{-ng(\lambda;\chi,\tau)}\left([{\bf R}^{(\infty)}(\lambda;\chi,\tau)]_{12}+\mathcal{O}(n^{-1/2})\right) \\ 
 & = \left( \frac{\lambda-\xi^*}{\lambda-\xi}\right)^{n}e^{-ng(\lambda;\chi,\tau)-f(\lambda;\chi,\tau)-f(\infty;\chi,\tau)} \left(\frac{\gamma(\lambda;\chi,\tau)-\gamma(\lambda;\chi,\tau)^{-1}}{2i}{e^{-2 i n K(\chi,\tau)}} + \mathcal{O}(n^{-1/2})\right).
\end{split}
\endeq
From 
\eq
\gamma(\lambda)-\gamma(\lambda)^{-1} = \frac{a^*-a}{2\lambda} + \mathcal{O}(\lambda^{-2}),
\endeq
\eq
\left(\frac{\lambda-\xi^*}{\lambda-\xi}\right)^n = 1 + \mathcal{O}(\lambda^{-1}),
\endeq
and
\eq
e^{-ng(\lambda)-f(\lambda)-f({\infty})} = e^{-2f(\infty)} + \mathcal{O}(\lambda^{-1}),
\endeq
we see 
\eq
\lim_{\lambda\to\infty} \lambda [{\bf M}^{[n]}(\lambda;n\chi,n\tau)]_{12} = \frac{a^*(\chi,\tau)-a(\chi,\tau)}{4i}e^{-2f(\infty;\chi,\tau)} {e^{-2 i n K(\chi,\tau)}} + \mathcal{O}(n^{-1/2}).
\endeq
Along with \eqref{psi-from-M}, this establishes Theorem \ref{nonosc-thm}.

\section{The oscillatory region}
\label{sec-osc}

Finally, we consider the oscillatory region.  From the Riemann-Hilbert point of 
view, this region is distinguished by a two-band model problem.  We begin by 
solving the following Riemann-Hilbert problem for $G(\lambda;\chi,\tau)$.  
\begin{rhp}[The $G$-function in the oscillatory region]
Fix a pole location $\xi\in\mathbb{C}^+$, a pair of nonzero complex numbers 
$(c_1,c_2)$, and a 
pair of real numbers $(\chi,\tau)$ in the oscillatory region.  Determine 
the unique contours $\Sigma_{\rm up}(\chi,\tau)$, 
$\Sigma_{\rm down}(\chi,\tau)$, and $\Gamma_{\rm mid}(\chi,\tau)$, the 
unique constants $\Omega(\chi,\tau)$ and $d(\chi,\tau)$, and the unique function 
$G(\lambda;\chi,\tau)$ satisfying the following conditions.
\begin{itemize}
\item[]\textbf{Analyticity:}  $G(\lambda)$ is analytic for 
$\lambda\in\mathbb{C}$ except on 
$\Sigma_{\rm up}\cup\Sigma_{\rm down}\cup\Gamma_{\rm mid}$, where it achieves 
continuous boundary values.  All three contours are simple and bounded.  
$\Sigma_{\rm down}$ is the reflection of $\Sigma_{\rm up}$ through the real 
axis.  $\Gamma_{\rm mid}$ is symmetric across the real axis and connects 
$\Sigma_{\rm down}$ to $\Sigma_{\rm up}$.  
\item[]\textbf{Jump condition:}  The boundary values taken by $G(\lambda)$ are
related by the jump conditions
\eq
\begin{split}
G_+(\lambda) + G_-(\lambda) & = 2\varphi(\lambda) + \Omega, \quad \lambda \in \Sigma_{\rm up},\\
G_+(\lambda) + G_-(\lambda) & = 2\varphi(\lambda) - \Omega^* = 2\varphi(\lambda) + \Omega, \quad \lambda \in \Sigma_{\rm down},\\
G_+(\lambda) - G_-(\lambda) & = d, \quad \lambda \in \Gamma_{\rm mid}.
\end{split}
\endeq
Here $\Omega$ and $d$ are purely imaginary constants.  
Furthermore, 
\eq
\Re(\varphi(\lambda)-G_+(\lambda)) = \Re(\varphi(\lambda)-G_-(\lambda)) = 0, \quad \lambda\in\Sigma_{\rm up}\cup\Sigma_{\rm down}\cup\Gamma_{\rm mid}. 
\endeq
\item[]\textbf{Normalization:}  As $\lambda\to\infty$, $G(\lambda)$ satisfies 
\begin{equation}
G(\lambda) = \mathcal{O}\left(\lambda^{-1}\right)
\end{equation}
with the limit being uniform with respect to direction.
\item[]\textbf{Symmetry:}  $G(\lambda)$ satisfies the symmetry condition
\eq
\label{G-symmetry}
G(\lambda) = -G(\lambda^*)^*.\\
\endeq
\end{itemize}
\label{rhp:G-osc}
\end{rhp}
The symmetry condition immediately implies that $d$ is purely imaginary.  
However, the fact that $\Omega$ is purely imaginary is a condition on 
$\Sigma_\text{up}$ and $\Sigma_\text{down}$.

Assume that $\Sigma_\text{up}$ and $\Sigma_\text{down}$ are known.  Suppose 
$\Sigma_\text{up}$ is oriented from $b\equiv b(\chi,\tau)$ to 
$a\equiv a(\chi,\tau)$ with $\Im(a)>\Im(b)$ and $\Sigma_\text{down}$ is oriented from 
$a^*$ to $b^*$.  
The band endpoints $a$ and $b$ are uniquely determined by the conditions 
\eq
\label{a-b-def}
G(\lambda)=\mathcal{O}(\lambda^{-1}), \quad \Re(\Omega)=0.
\endeq
We now differentiate and solve for $G'(\lambda)$.  Observe that $G'(\lambda)$ has 
jumps 
\eq
\label{G'-condition-1}
G'_+(\lambda) + G'_-(\lambda) = 2i \chi+4i\lambda\tau+\frac{2}{\lambda-\xi^*}-\frac{2}{\lambda-\xi},\quad\quad \lambda \in \Sigma_\text{up} \cup \Sigma_\text{down}
\endeq
and normalization 
\eq
\label{G'-condition-3}
G'(\lambda) =  \mathcal{O}(\lambda^{-2}), \quad \lambda \to \infty.
\endeq
Define 
\eq
\mathfrak{R}(\lambda):=((\lambda-a)(\lambda-a^{*})(\lambda-b)(\lambda-b^{*}))^{1/2}
\endeq
to be the function cut on $\Sigma_\text{up}\cup\Sigma_\text{down}$ with 
asymptotic behavior $\mathfrak{R}(\lambda) = \lambda^2 +\mathcal{O}(\lambda)$ as 
$\lambda \to \infty.$
Note that if we define the symmetric functions 
\eq
\label{symmetric-fns}
\begin{split}
\mathfrak{s}_1 := a+a^*+b+b^*, \quad \mathfrak{s}_2 := aa^*+ab+ab^*+a^*b+a^*b^*+bb^*, \\
\mathfrak{s}_3 := aa^*b+aa^*b^*+abb^*+a^*bb^*, \quad \mathfrak{s}_4 := aa^*bb^*,\hspace{.55in}
\end{split}
\endeq
then we can write
\eq
\mathfrak{R}(\lambda)=(\lambda^4-\mathfrak{s}_1\lambda^3+\mathfrak{s}_2\lambda^2-\mathfrak{s}_3\lambda+\mathfrak{s}_4)^{1/2}.
\endeq
By the Plemelj formula, we have
\eq 
\label{G'function}
G'(\lambda)=\frac{\mathfrak{R}(\lambda)}{2\pi i}\int_{\Sigma_\text{up}\cup\Sigma_\text{down}}\frac{2i \chi+4is\tau+\frac{2}{s-\xi^*}-\frac{2}{s-\xi}}{\mathfrak{R}_{+}(s)(s-\lambda)}ds.
\endeq
Similar to the calculation for $g'(\lambda)$ in \S\ref{sec-nonosc}, an explicit 
residue computation gives 
\eq 
\label{G'-function-new}
G'(\lambda)=i\chi+2i\tau\lambda+\frac{1}{\lambda-\xi^*}-\frac{1}{\lambda-\xi}
+\frac{\mathfrak{R}(\lambda)}{\mathfrak{R}(\xi^{*})(\xi^{*}-\lambda)}-\frac{\mathfrak{R}(\lambda)}{\mathfrak{R}(\xi)(\xi-\lambda)}.
\endeq
We now present a computationally effective method of determining $a$ and $b$.  
Imposing the growth condition $G'(\lambda)=\mathcal{O}(\lambda^{-2})$ 
leads to the following three conditions arising from requiring the terms 
proportional to $\lambda^1$, $\lambda^0$, and $\lambda^{-1}$ in the 
large-$\lambda$ expansion of \eqref{G'-function-new} to be zero:
\eq \label{ab-condition-1}
\mathcal{O}(\lambda):\,2\tau+\frac{i}{\mathfrak{R}(\xi^{*})}-\frac{i}{\mathfrak{R}(\xi)}=0,
\endeq
\eq \label{ab-condition-2}
\mathcal{O}(1):\,\chi+\tau\mathfrak{s}_1+\frac{i\xi^{*}}{\mathfrak{R}(\xi^{*})}-\frac{i\xi}{\mathfrak{R}(\xi)}=0,
\endeq
\eq 
\label{ab-condition-3}
\mathcal{O}(\lambda^{-1}):\,\frac{\chi}{2}\mathfrak{s}_1+\tau\left(\frac{3}{4}\mathfrak{s}_1^2-\mathfrak{s}_2\right)
+\frac{i(\xi^*)^2}{\mathfrak{R}(\xi^{*})}-\frac{i\xi^2}{\mathfrak{R}(\xi)}=0.
\endeq
These are three real conditions on the two complex unknowns $a$ and $b$ (the 
fourth condition will be $\Re(\Omega)=0$).  Multiplying equation 
\eqref{ab-condition-1} by $\xi^{*}$ and plugging it into \eqref{ab-condition-2}, 
we have
\eq 
\label{ab-condition-4}
\chi+\tau\mathfrak{s}_1-2\tau \xi^*=-i\frac{\xi^*-\xi}{\mathfrak{R}(\xi)}.
\endeq
Next, multiplying equation \eqref{ab-condition-1} by $(\xi^*)^2$ and plugging it 
into \eqref{ab-condition-3}, we have
\eq \label{ab-condition-5}
\frac{\chi}{2}\mathfrak{s}_1+\tau\left(\frac{3}{4}\mathfrak{s}_1^2-\mathfrak{s}_2\right)-2\tau(\xi^*)^2
=-i\frac{(\xi^*-\xi)(\xi^*+\xi)}{\mathfrak{R}(\xi)}.
\endeq
Then, multiplying equation \eqref{ab-condition-4} by $(\xi^*+\xi)$ and equating 
it with \eqref{ab-condition-5}, we have
\eq \label{ab-condition-6}
\mathfrak{s}_2=\frac{3}{4}\mathfrak{s}_1^2+\left(\frac{1}{2}\frac{\chi}{\tau}-\xi^{*}-\xi\right)\mathfrak{s}_1+2\xi\xi^*
-(\xi^*+\xi)\frac{\chi}{\tau},
\endeq
which indicates that if $\mathfrak{s}_1$ is real then $\mathfrak{s}_2$ is real.  Now use 
\eqref{ab-condition-6} to eliminate $\mathfrak{s}_2$ in \eqref{ab-condition-4} (here 
$\mathfrak{s}_2$ appears in $\mathfrak{R}(\xi)$).  Take 
the real and imaginary parts to get two real equations on the three real 
variables $\mathfrak{s}_1, \mathfrak{s}_3$, and $\mathfrak{s}_4$.  These equations are both linear 
in $\mathfrak{s}_3$ and $\mathfrak{s}_4$, so $\mathfrak{s}_3$ and $\mathfrak{s}_4$ can be solved exactly 
in terms of $\mathfrak{s}_1$.  Thus, given $\mathfrak{s}_1$, we can determine 
$\mathfrak{s}_2$, $\mathfrak{s}_3$, 
and $\mathfrak{s}_4$, from which the system \eqref{symmetric-fns} can be inverted to 
obtain $a$ and $b$.  At this point we can define $G(\lambda)$ by 
\eq
\label{G-from-Gprime}
G(\lambda):=\int_\infty^\lambda G'(s)ds,
\endeq
where the path of integration is chosen to avoid 
$\Sigma_\text{up}\cup\Sigma_\text{down}\cup\Gamma_\text{mid}$.  Finally, we 
choose $\mathfrak{s}_1$ so that, once $a$ and $b$ and thus $G(\lambda)$ have been 
computed, $d:=G_+(\lambda)-G_-(\lambda)$ is purely imaginary (here $d$ is 
independent of $\lambda$ as long as $\lambda\in\Gamma_\text{mid}$). 

The final step in the definition of $G(\lambda)$ is the choice of cuts.  
Similar to the non-oscillatory case, we note from \eqref{G'function} that 
shifting $\Sigma_\text{up}$ or $\Sigma_\text{down}$ only changes $G(\lambda)$ by 
at most a sign, and so has no effect on the placement of the contours along which 
$\Re(\varphi(\lambda)-G(\lambda))=0$.  Therefore, we redefine $\Sigma_\text{up}$ 
to be the simple contour from $b$ to $a$ along which 
$\Re(\varphi(\lambda)-G(\lambda))=0$ and $\Re(\varphi(\lambda)-G(\lambda))$ is 
positive to either side.  The symmetry condition \eqref{G-symmetry} then forces 
$\Sigma_\text{down}$ to be the reflection of $\Sigma_\text{up}$ through the 
real axis.  We also choose $\Gamma_\text{mid}$ (whose main role is to restrict 
the integration path in \eqref{G-from-Gprime}) to be the contour from $b^*$ to 
$b$ along which $\Re(\varphi(\lambda)-G(\lambda))=0$.  The fact that such 
contours exist along which $\Re(\varphi(\lambda)-G(\lambda))=0$ is proven next
in Lemma \ref{osc-lemma}.  

\begin{lemma}
\label{osc-lemma}
In the oscillatory region, there is a domain $D_{\rm up}$ in the upper half-plane 
with the following properties:
\begin{itemize}
\item $D_{\rm up}$ contains $\xi$ and is bounded by a simple Jordan curve along which 
$\Re(\varphi(\lambda)-G(\lambda))=0$.  This curve contains the points $a$ and 
$b$.
\item $\Re(\varphi(\lambda)-G(\lambda))>0$ for all $\lambda\in D_{\rm up}$.
\item One arc of the boundary of $D_{\rm up}$ is the contour 
$\Sigma_{\rm up}$ from $b$ to $a$, along which  $\Re(\varphi(\lambda)-G(\lambda))>0$ 
for any $\lambda$ sufficiently close to either side of $\Sigma_{\rm up}$.  
\item The remaining boundary of $D_{\rm up}$ is a contour from $a$ to $b$ (denoted 
$\Gamma_{\rm up}$) along which $\Re(\varphi(\lambda)-G(\lambda))<0$ for any 
$\lambda$ in the exterior of $\overline{D_{\rm up}}$ but sufficiently close to 
$D_{\rm up}$.  
\end{itemize}
The domain $D_{\rm down}$ in the lower half-plane, defined as the reflection of 
$D_{\rm up}$ through the real axis, has the following properties:
\begin{itemize}
\item $D_{\rm down}$ contains $\xi^*$ and is bounded by a simple Jordan curve along 
which $\Re(\varphi(\lambda)-G(\lambda))=0$.
\item $\Re(\varphi(\lambda)-G(\lambda))<0$ for all $\lambda\in D_{\rm down}$.
\item One arc of the boundary of $D_{\rm down}$ is a contour (denoted 
$\Sigma_{\rm down}$) from $a^*$ to $b^*$, along which 
$\Re(\varphi(\lambda)-G(\lambda))<0$ for any $\lambda$ sufficiently close to 
either side of $\Sigma_{\rm down}$.  
\item The remaining boundary of $D_{\rm down}$ is a contour from $b^*$ to $a^*$ 
(denoted $\Gamma_{\rm down}$) along which $\Re(\varphi(\lambda)-G(\lambda))>0$ 
for any $\lambda$ in the exterior of $\overline{D_{\rm down}}$ but sufficiently 
close to $D_{\rm down}$.  
\end{itemize}
\end{lemma}
\begin{proof}
The proof is similar to that of Lemma \ref{nonosc-lemma}.  From \eqref{phi-def} 
and \eqref{G'-function-new}, we see 
\eq
\varphi'(\lambda) - G'(\lambda) = \mathfrak{R}(\lambda)\left(\frac{1}{\mathfrak{R}(\xi)(\xi-\lambda)}-\frac{1}{\mathfrak{R}(\xi^*)(\xi^*-\lambda)}\right).
\endeq
From the first factor $\mathfrak{R}(\lambda)$, we see 
$\varphi'(\lambda)-G'(\lambda)$ has four square-root branch points and the same 
branch cut as $\mathfrak{R}(\lambda)$.  From the second factor we can clear 
denominators and see that $\varphi(\lambda)-G(\lambda)$ has exactly one critical 
point.  By symmetry this critical point must lie on the real axis, and thus on a 
curve on which $\varphi(\lambda)-G(\lambda)=0$.  The topology of the level curves 
and the structure of the signature chart of $\Re(\varphi(\lambda)-G(\lambda))$ 
is deduced from analytic continuation from either $\mathcal{L}_\text{NO}$ (the 
shared boundary with the non-oscillatory region) or from $\mathcal{L}_\text{EO}$ 
(the shared boundary with the exponential-decay region).
\end{proof}
The signature chart of $\Re(\varphi(\lambda)-G(\lambda))$ is illustrated in 
Figure \ref{oscillatory-phase-plots}.
We now begin our transformations of Riemann-Hilbert Problem \ref{rhp-N}.  Define 
\eq
{\bf O}^{[n]}(\lambda;\chi,\tau):= \begin{cases} {\bf N}^{[n]}(\lambda;\chi,\tau){\bf V}_{\bf N}^{[n]}(\lambda;\chi,\tau), & \lambda\in D_0\cap(D_{\rm up}\cup D_{\rm down})^\mathsf{c}, \\ {\bf N}^{[n]}(\lambda;\chi,\tau){\bf V}_{\bf N}^{[n]}(\lambda;\chi,\tau)^{-1}, & \lambda\in D_0^\mathsf{c}\cap(D_{\rm up}\cup D_{\rm down}), \\ {\bf N}^{[n]}(\lambda;\chi,\tau), & \text{otherwise}.
\end{cases}
\endeq
The jump for ${\bf O}^{[n]}(\lambda)$ lies on 
$\Sigma_\text{up}\cup\Sigma_\text{down}\cup\Gamma_\text{up}\cup\Gamma_\text{down}$.
Next, define 
\eq \label{G-transform-osc}
{\bf P}^{[n]}(\lambda;\chi,\tau):={\bf O}^{[n]}(\lambda;\chi,\tau)e^{-nG(\lambda)\sigma_3}.
\endeq
The matrix ${\bf P}^{[n]}(\lambda)$ has an additional jump on 
$\Gamma_\text{mid}$, namely
\eq
{\bf P}_+^{[n]}(\lambda) = {\bf P}_-^{[n]}(\lambda)\bbm e^{-n(G_+(\lambda)-G_-(\lambda))} & 0 \\ 0 & e^{n(G_+(\lambda)-G_-(\lambda))}\ebm = {\bf P}_-^{[n]}(\lambda)\bbm e^{-nd} & 0 \\ 0 & e^{nd}\ebm, \quad \lambda\in\Gamma_\text{mid}.
\endeq
\begin{figure}
\begin{center}
\includegraphics[height=2.1in]{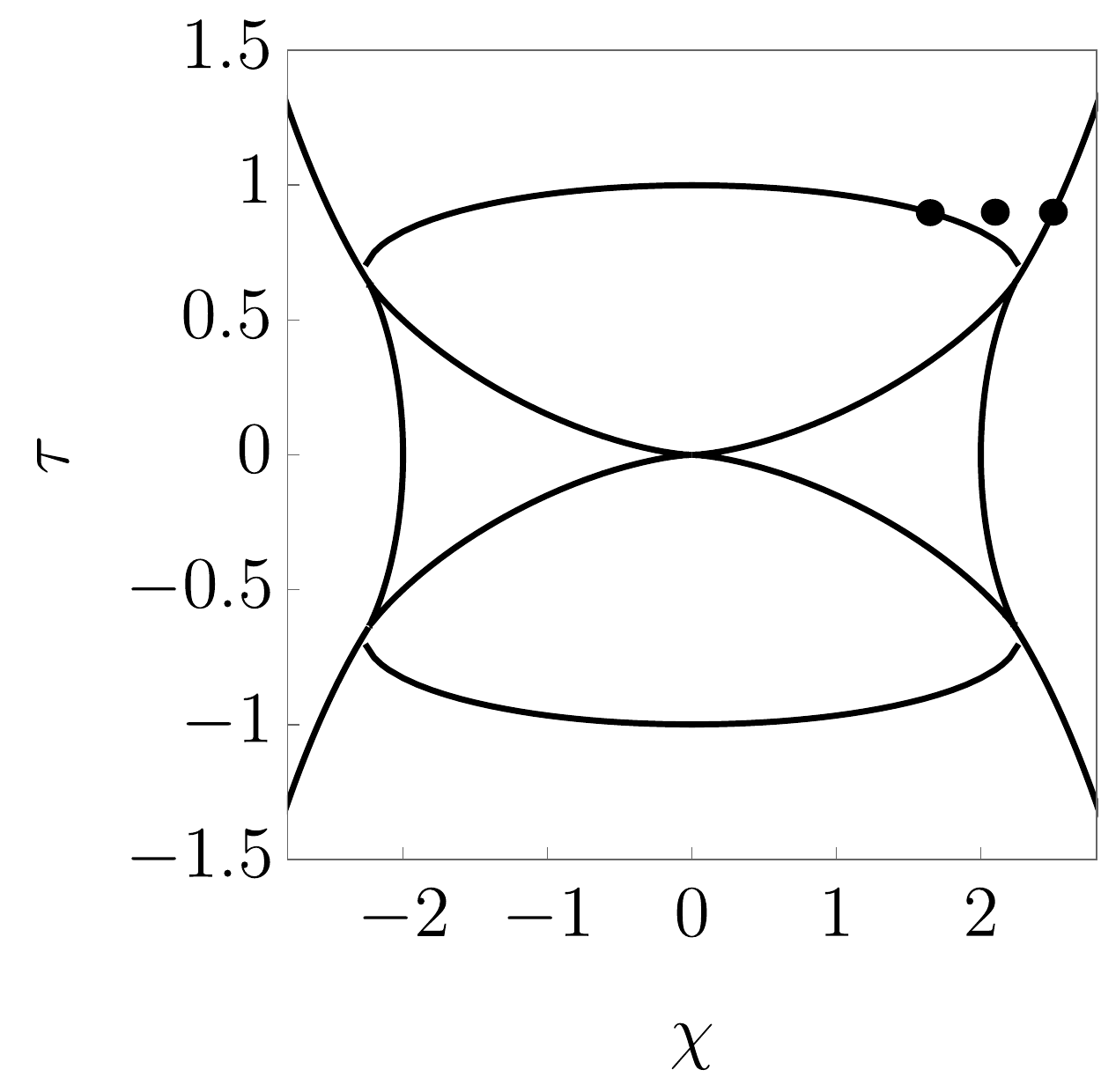} \\
\includegraphics[height=2in]{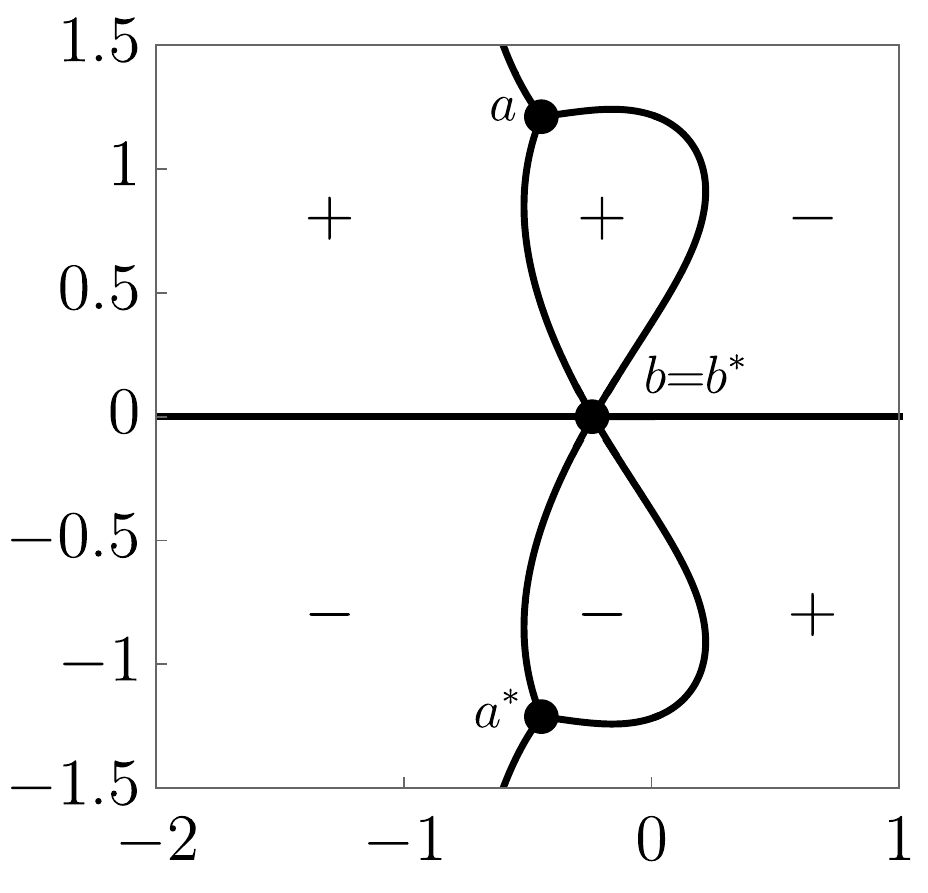} 
\includegraphics[height=2in]{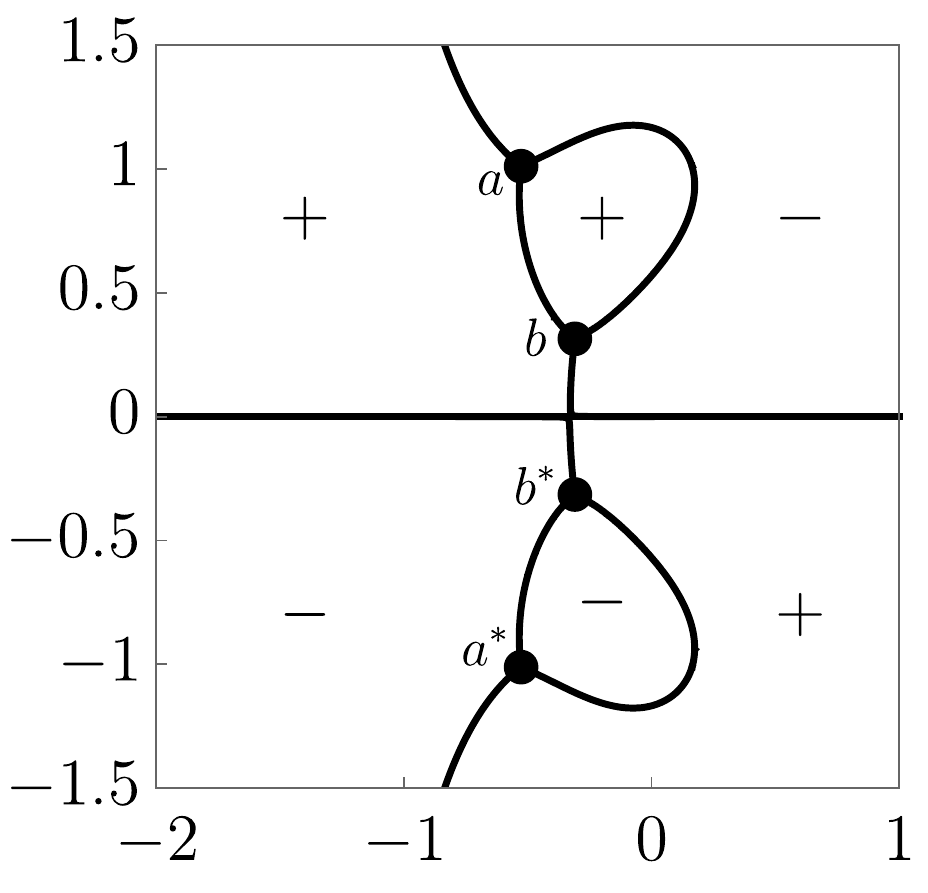} 
\includegraphics[height=2in]{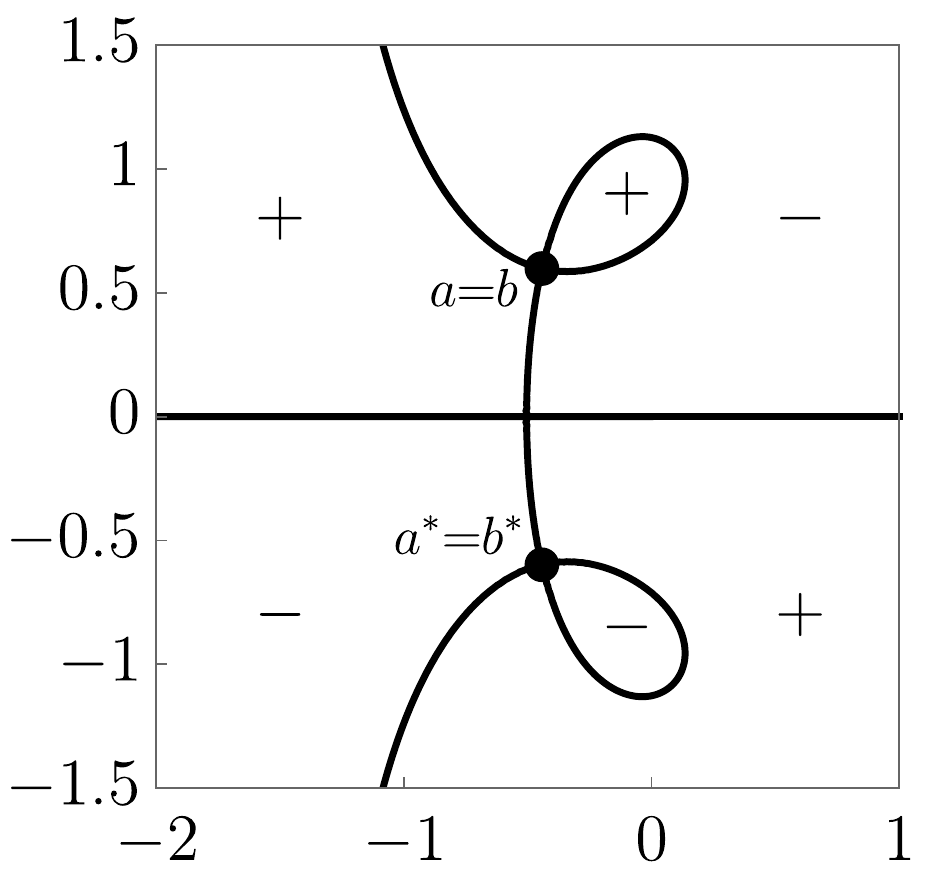} 
\caption{Signature charts of 
$\Re(\varphi(\lambda;\chi,\tau)-G(\lambda;\chi,\tau))$ for $\xi=i$ in the 
oscillatory region, along with the band endpoints $a$, $a^*$, $b$, and $b^*$.  
\emph{Top}: Positions in the ($\chi$,$\tau$)-plane relative to 
the boundary curves.  \emph{Bottom right}:  $\chi=1.65$, $\tau\approx 0.8983$.  
\emph{Bottom middle}:  $\chi=2.1$, $\tau=0.9$.  \emph{Bottom right}: 
$\chi\approx 2.502$, $\tau=0.9$.} 
\label{oscillatory-phase-plots}
\end{center}
\end{figure}
Analogously to the non-oscillatory region, we define the contours 
\begin{itemize}
\item $\Sigma_\text{up}^\text{out}$ runs from $b$ to $a$ in 
the upper half-plane entirely in the region exterior to $D_{\rm up}$ in which 
$\Re(\varphi(\lambda)-G(\lambda))>0$.
\item $\Sigma_\text{up}^\text{in}$ runs from $b$ to $a$ 
entirely in $D_{\rm up}$ (so $\Re(\varphi(\lambda)-G(\lambda))>0$), and can be 
deformed to $\Sigma_\text{up}$ without passing through $\xi$.
\item $\Gamma_\text{up}^\text{out}$ runs from $a$ to $b$ 
in the upper half-plane entirely in the region where 
$\Re(\varphi(\lambda)-G(\lambda))<0$.
\item $\Gamma_\text{up}^\text{in}$ runs from $a$ to $b$ 
entirely in $D_{\rm up}$ (so $\Re(\varphi(\lambda)-G(\lambda))>0$), and can be 
deformed to $\Gamma_\text{up}$ without passing through $\xi$.
\item $\Sigma_\text{down}^\text{out}$ (oriented from $a^*$ to $b^*$), 
$\Sigma_\text{down}^\text{in}$ (oriented from $a^*$ to $b^*$), 
$\Gamma_\text{down}^\text{out}$ (oriented from $b^*$ to $a^*$), 
and $\Gamma_\text{down}^\text{in}$ (oriented from $b^*$ to $a^*$) 
are the reflections through the real axis of $\Sigma_\text{up}^\text{out}$, 
$\Sigma_\text{up}^\text{in}$, $\Gamma_\text{up}^\text{out}$, and 
$\Gamma_\text{up}^\text{in}$, respectively.
\end{itemize}
Also define the domains 
\begin{itemize}
\item $K_\text{up}^\text{out}$ (respectively, $K_\text{up}^\text{in}$) is the 
domain in the upper half-plane bounded by $\Sigma_\text{up}^\text{out}$ 
(respectively, $\Sigma_\text{up}^\text{in}$) and $\Sigma_\text{up}$.  
\item $L_\text{up}^\text{out}$ (respectively, $L_\text{up}^\text{in}$) is the 
domain in the upper half-plane bounded by $\Gamma_\text{up}^\text{out}$ 
(respectively, $\Gamma_\text{up}^\text{in}$) and $\Gamma_\text{up}$.  
\item $K_\text{down}^\text{out}$, $K_\text{down}^\text{in}$, 
$L_\text{down}^\text{out}$, and $L_\text{down}^\text{in}$ are the reflections 
through the real axis of $K_\text{up}^\text{out}$, $K_\text{up}^\text{in}$, 
$L_\text{up}^\text{out}$, and $L_\text{up}^\text{in}$, respectively.
\end{itemize}
\begin{figure}[b]
\begin{center}
\includegraphics[height=2in]{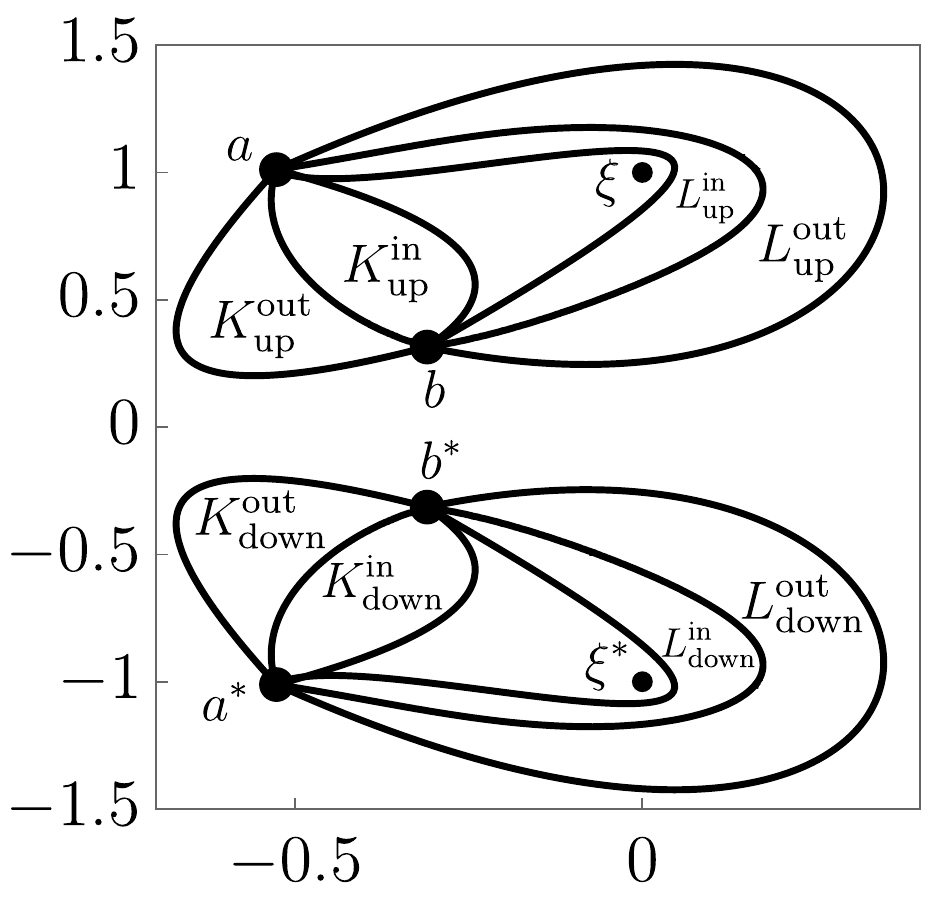} 
\includegraphics[height=2in]{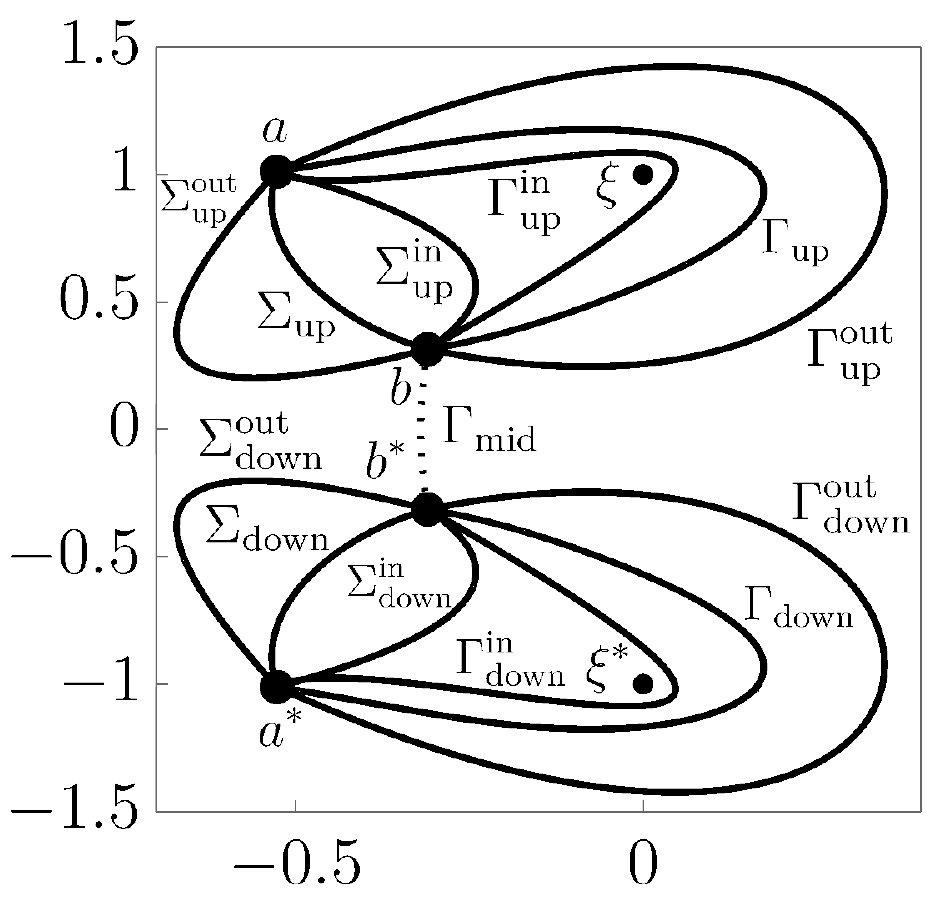} 
\caption{The domains (left) and contours (right) used in the definition of 
${\bf Q}^{[n]}(\lambda)$ in the oscillatory region.  The contour 
$\Gamma_\text{mid}$ is denoted by a dotted line.}
\label{osc-lenses}
\end{center}
\end{figure}
See Figure \ref{osc-lenses}.  Then we define ${\bf Q}^{[n]}(\lambda)$ by 
opening lenses as in \eqref{Q-def-nonosc} (except with $g(\lambda)$ replaced by 
$G(\lambda)$).  The jump matrices for ${\bf Q}^{[n]}(\lambda)$ are as follows:
\eq
\begin{split}
\Sigma_\text{up}:\,\,\bbm 0 & \frac{|{\bf c}|}{c_2}e^{n\Omega} \\ -\frac{c_2}{|{\bf c}|}e^{-n\Omega}  & 0 \ebm,\quad
\Sigma_\text{down}:\,\,\bbm 0 & \frac{c_2^*}{|{\bf c}|}e^{n\Omega} \\ -\frac{|{\bf c}|}{c_2^*}e^{-n\Omega} & 0\ebm, \hspace{0.93in}
\\
\Gamma_\text{up}:\,\,\bbm \frac{|{\bf c}|}{c_1} & 0\\ 0 & \frac{c_1}{|{\bf c}|} \ebm,  \quad
\Gamma_\text{down}:\,\,\bbm \frac{c_1^{*}}{|{\bf c}|}& 0\\ 0 & \frac{|{\bf c}|}{c_1^{*}}  \ebm, \quad  
\Gamma_\text{mid}:\,\,\bbm e^{-nd} & 0 \\ 0 & e^{nd} \ebm, \hspace{0.85in}
\\
\Sigma_\text{up}^\text{in}:\,\,\bbm 1& -\frac{c_1^*}{c_2}e^{-2n(\varphi-G)}\\ 0 & 1  \ebm,\quad
\Sigma_\text{up}^\text{out}:\,\,\bbm 1& -\frac{c_1}{c_2}e^{-2n(\varphi-G)}\\ 0 & 1  \ebm,\quad
\Sigma_\text{down}^\text{in}:\,\,\bbm 1& 0\\ \frac{c_1}{c_2^*}e^{2n(\varphi-G)} & 0  \ebm, 
\\
\Sigma_\text{down}^\text{out}:\,\,\bbm 1& 0\\ \frac{c_1^{*}}{c_2^*}e^{2n(\varphi-G)} & 0  \ebm,\quad
\Gamma_\text{up}^\text{in}:\,\,\bbm 1& \frac{c_2^{*}}{c_1}e^{-2n(\varphi-G)}\\ 0 & 1  \ebm,\quad
\Gamma_\text{up}^\text{out}:\,\,\bbm 1& 0\\ -\frac{c_2}{c_1}e^{2n(\varphi-G)} & 0  \ebm,
\hspace{.2in}
\\
\Gamma_\text{down}^\text{in}:\,\,\bbm 1& 0\\ -\frac{c_2}{c_1^*}e^{2n(\varphi-G)} & 0  \ebm,\quad
\Gamma_\text{down}^\text{out}:\,\,\bbm 1& \frac{c_2^{*}}{c_1^{*}}e^{-2n(\varphi-G)}\\ 0 & 1 \ebm.
\hspace{1in}
\end{split}
\endeq
Lemma \ref{osc-lemma} shows that all of the non-constant jump matrices decay 
exponentially fast to the identity matrix outside of small fixed neighborhoods 
$\mathbb{D}^{(a)}$, $\mathbb{D}^{(b)}$, $\mathbb{D}^{(a^*)}$, and 
$\mathbb{D}^{(b^*)}$ of $a$, $b$, $a^*$, and $b^*$, respectively.  We therefore 
arrive at the outer model problem.  

\begin{rhp}[The outer model problem in the oscillatory region]
Fix a pole location $\xi\in\mathbb{C}^+$, a pair of nonzero complex numbers 
$(c_1,c_2)$, and a 
pair of real numbers $(\chi,\tau)$ in the oscillatory region.  Determine the 
$2\times 2$ matrix ${\bf R}^{(\infty)}(\lambda;\chi,\tau)$ with the 
following properties:
\begin{itemize}
\item[]\textbf{Analyticity:}  ${\bf R}^{(\infty)}(\lambda;\chi,\tau)$ is analytic 
for $\lambda\in \mathbb{C}$ except on 
$\Sigma_{\rm up}\cup\Sigma_{\rm down}\cup\Gamma_{\rm up}\cup\Gamma_{\rm down}\cup\Gamma_{\rm mid}$,
where it achieves continuous boundary values on the interior of each arc.
\item[]\textbf{Jump condition:}  The boundary values taken by ${\bf R}^{(\infty)}(\lambda;\chi,\tau)$ are
related by the jump conditions
${\bf R}^{(\infty)}_+(\lambda;\chi,\tau)={\bf R}^{(\infty)}_-(\lambda;\chi,\tau){\bf V}_{\bf R}^{(\infty)}(\lambda;\chi,\tau)$,
where
\eq
{\bf V}_{\bf R}^{(\infty)}(\lambda;\chi,\tau) := \begin{cases} \bbm 0 & \frac{|{\bf c}|}{c_2}e^{n\Omega} \\ -\frac{c_2}{|{\bf c}|} e^{-n\Omega}& 0\ebm, & \lambda\in \Sigma_{\rm up}, \vspace{.025in} \\ \bbm 0 & \frac{c_2^*}{|{\bf c}|}e^{n\Omega} \\ -\frac{|{\bf c}|}{c_2^*}e^{-n\Omega} & 0\ebm, & \lambda\in \Sigma_{\rm down},\vspace{.025in} \\
\bbm \frac{|{\bf c}|}{c_1} & 0 \\ 0 & \frac{c_1}{|{\bf c}|} \ebm, & \lambda\in \Gamma_{\rm up}, \vspace{.025in} \\ \bbm \frac{c_1^*}{|{\bf c}|} & 0 \\ 0 & \frac{|{\bf c}|}{c_1^*} \ebm, & \lambda\in \Gamma_{\rm down}, \vspace{.025in} \\
\bbm e^{-nd} & 0 \\ 0 & e^{nd} \ebm, & \lambda\in\Gamma_{\rm mid}.
\end{cases}
\endeq
\item[]\textbf{Normalization:}  As $\lambda\to\infty$, the matrix ${\bf R}^{(\infty)}(\lambda;\chi,\tau)$
satisfies the condition
\begin{equation}
{\bf R}^{(\infty)}(\lambda;\chi,\tau) = \mathbb{I}+\mathcal{O}(\lambda^{-1})
\end{equation}
with the limit being uniform with respect to direction.
\end{itemize}
\label{rhp:R-model-oscillatory}
\end{rhp}
To remove the dependence on $c_1$, $c_2$, $\Omega$, and $d$, we define 
\eq
\begin{split}
F(\lambda):=\frac{\mathfrak{R}(\lambda)}{2\pi i}\left(\int_{\Sigma_\text{up}} \frac{-n\Omega-\log\left(\frac{|{\bf c}|}{c_2}\right)}{\mathfrak{R}_{+}(s)(s-\lambda)} ds+ \int_{\Sigma_\text{down}} \frac{-n\Omega-\log\left(\frac{c_2^{*}}{|{\bf c}|}\right)}{\mathfrak{R}_{+}(s)(s-\lambda)} ds \right. \hspace{1.2in}\\ 
\left . + \int_{\Gamma_{\rm up}} \frac{\log\left(\frac{|{\bf c}|}{c_1}\right)}{\mathfrak{R}(s)(s-\lambda)} ds +\int_{\Gamma_{\rm down}} \frac{\log\left(\frac{c_1^{*}}{|{\bf c}|}\right)}{\mathfrak{R}(s)(s-\lambda)} ds +\int_{\Gamma_\text{mid}}\frac{-nd}{\mathfrak{R}(s)(s-\lambda)}ds \right).
\end{split}
\endeq
Here $F(\lambda)$ satisfies the jump conditions 
\eq
\begin{split}
F_{+}+ F_{-} & = -n\Omega -\log\left(\frac{|{\bf c}|}{c_2}\right),\quad \lambda \in \Sigma_\text{up}, \\
F_{+}+ F_{-} & =  -n\Omega-\log\left(\frac{c_2^{*}}{|{\bf c}|}\right),\quad \lambda \in \Sigma_\text{down}, \\
F_{+}- F_{-} & = \log\left(\frac{|{\bf c}|}{c_1}\right),\quad\lambda \in \Gamma_{\rm up}, \\
F_{+}- F_{-}  & = \log\left(\frac{c_1^{*}}{|{\bf c}|}\right),\quad \lambda \in \Gamma_{\rm down}, \\
F_{+}- F_{-} & = -nd,\quad \lambda\in\Gamma_\text{mid}
\end{split}
\endeq
and the symmetry
\eq
F(\lambda) = -(F(\lambda^{*}))^*.
\endeq
As $\lambda\to\infty$ we have 
\eq
F(\lambda)=F_1\lambda+F_0+\mathcal{O}(\lambda^{-1}), 
\endeq
where 
\eq
\label{F1-def}
\begin{split}
F_1:=\frac{-1}{2\pi i}\left(\int_{\Sigma_\text{up}} \frac{-n\Omega-\log\left(\frac{|{\bf c}|}{c_2}\right)}{\mathfrak{R}_{+}(s)} ds+ \int_{\Sigma_\text{down}} \frac{-n\Omega-\log\left(\frac{c_2^{*}}{|{\bf c}|}\right)}{\mathfrak{R}_{+}(s)} ds \right. \hspace{0.8in}\\ 
\left . + \int_{\Gamma_{\rm up}} \frac{\log\left(\frac{|{\bf c}|}{c_1}\right)}{\mathfrak{R}(s)} ds +\int_{\Gamma_{\rm down}} \frac{\log\left(\frac{c_1^{*}}{|{\bf c}|}\right)}{\mathfrak{R}(s)} ds +\int_{\Gamma_\text{mid}}\frac{-nd}{\mathfrak{R}(s)(s-\lambda)}ds \right)
\end{split}
\endeq
and
\eq
\label{F0-def}
\begin{split}
F_0 := -\frac{\mathfrak{s}_1}{2}F_1 - \frac{1}{2\pi i}\left(\int_{\Sigma_\text{up}} \frac{-n\Omega-\log\left(\frac{|{\bf c}|}{c_2}\right)}{\mathfrak{R}_{+}(s)} sds+ \int_{\Sigma_\text{down}} \frac{-n\Omega-\log\left(\frac{c_2^{*}}{|{\bf c}|}\right)}{\mathfrak{R}_{+}(s)} sds \right. \hspace{0.6in}\\ 
\left . + \int_{\Gamma_{\rm up}} \frac{\log\left(\frac{|{\bf c}|}{c_1}\right)}{\mathfrak{R}(s)} sds +\int_{\Gamma_{\rm down}} \frac{\log\left(\frac{c_1^{*}}{|{\bf c}|}\right)}{\mathfrak{R}(s)} sds +\int_{\Gamma_\text{mid}}\frac{-nd}{\mathfrak{R}(s)(s-\lambda)}sds \right).
\end{split}
\endeq
Define 
\eq \label{S-def-osc}
{\bf S}(\lambda):=e^{F_0\sigma_3}{\bf R}^{(\infty)}(\lambda)e^{-F(\lambda)\sigma_3}.
\endeq
Then ${\bf S}(\lambda)$ is analytic for 
$\lambda\notin\Sigma_\text{up}\cup\Sigma_\text{down}$, has jumps 
\eq
\label{S-jumps}
{\bf S}_{+}(\lambda)={\bf S}_{-}(\lambda)\bbm 0 & 1 \\ -1 & 0 \ebm, \quad \lambda\in\Sigma_\text{up}\cup\Sigma_\text{down},
\endeq
and has large-$\lambda$ behavior 
\eq
\label{S-asymptotics}
{\bf S}(\lambda)e^{F_1\lambda\sigma_3}= \mathbb{I}+\mathcal{O}(\lambda^{-1}), \quad \lambda\to\infty.
\endeq

We now build ${\bf S}(\lambda)$ explicitly out of Riemann-theta functions.  See 
\cite{BothnerM:2019,BuckinghamM:2013}, for example, for similar constructions.  
The function $\mathfrak{R}(\lambda)$ defines a genus-one Riemann surface 
constructed from two copies of the complex plane cut on $\Sigma_\text{up}$ and 
$\Sigma_\text{down}$.  We introduce a basis of homology cycles 
$\{\mathfrak{a},\mathfrak{b}\}$ as shown in Figure \ref{osc-cycles}.  Here integration on 
the second sheet is accomplished by replacing $\mathfrak{R}(\lambda)$ by 
$-\mathfrak{R}(\lambda)$.  
\begin{figure}[ht]
\begin{center}
\includegraphics[height=2.1in]{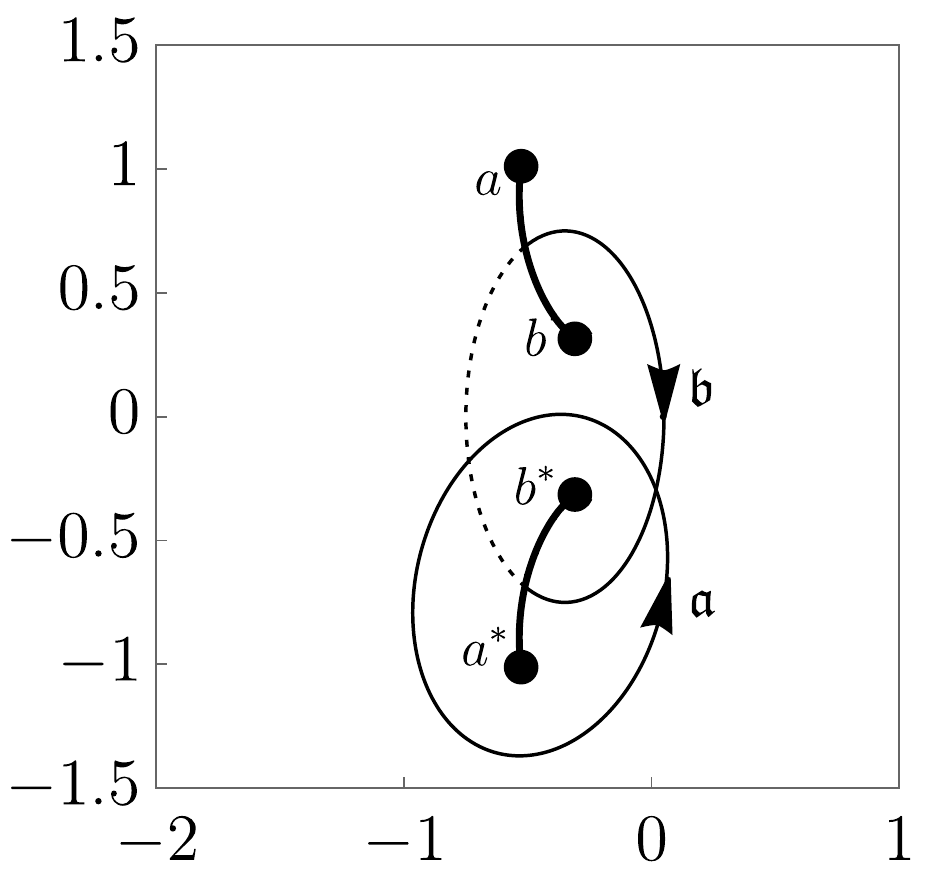}
\caption{The homology cycles $\mathfrak{a}$ and $\mathfrak{b}$ in relation to the branch cuts 
of $\mathfrak{R}(\lambda)$.  Thin solid lines lie on the first sheet while the 
dotted line lies on the second sheet.}
\label{osc-cycles}
\end{center}
\end{figure}
Define the Abel map as
\eq \label{Abel-map}
A(\lambda):=\frac{2\pi i}{\oint_{\mathfrak{a}}\frac{ds}{\mathfrak{R}(s)}}\int_{a^*}^{\lambda}\frac{ds}{\mathfrak{R}(s)}.
\endeq
We think of the integration as being on the Riemann surface (i.e. if the 
integration path passes through a branch cut then $\mathfrak{R}(\lambda)$ flips 
to $-\mathfrak{R}(\lambda)$).  The Abel map depends on the integration contour 
and changes value if an extra $\mathfrak{a}$ cycle or $\mathfrak{b}$ cycle is added.  In 
particular, adding an extra $\mathfrak{a}$ cycle to the integration contour adds 
$2\pi i$ to the Abel map, while an extra $\mathfrak{b}$ cycle adds the quantity
\eq \label{B-cycle}
B:=\frac{2\pi i}{\oint_{\mathfrak{a}}\frac{ds}{\mathfrak{R}(s)}}\oint_{\mathfrak{b}}\frac{ds}{\mathfrak{R}(s)}.
\endeq
We define the lattice 
\eq
\Lambda:=2\pi ij + Bk, \quad j,k\in\mathbb{Z}.
\endeq
Then the Abel map is well-defined modulo $\Lambda$.  We compute
\eq
\label{Abel-properties}
\begin{split}
A_{+}(\lambda)+ A_{-}(\lambda) & = -B\text{ mod }\Lambda, \quad \lambda \in \Sigma_\text{up},\\
A_{+}(\lambda)- A_{-}(\lambda)   & = -2\pi i\text{ mod }\Lambda, \quad \lambda\in\Gamma_\text{mid},\\
A_{+}(\lambda)+ A_{-}(\lambda) & = 0 \text{ mod }\Lambda, \quad \lambda \in \Sigma_\text{down}.
\end{split}
\endeq
We now define two differentials $\omega$ and $\Delta$.  Let
\eq
\omega := \frac{2\pi i}{\oint_{\mathfrak{a}}\frac{ds}{\mathfrak{R}(s)}}\frac{ds}{\mathfrak{R}(s)}
\endeq
be the holomorphic differential normalized so $\oint_\mathfrak{a}\omega=2\pi i$.  We 
also define 
\eq 
\label{one-form}
\Delta_0 := \frac{s^2-\frac{1}{2}\mathfrak{s}_1s}{\mathfrak{R}(s)}ds, \quad \Delta=\Delta_0-\left(\frac{1}{2\pi i}\oint_{\mathfrak{a}}\Delta_0\right)\omega
\endeq
so that $\oint_\mathfrak{a}\Delta=0$.  Here $\Delta_0$ is chosen to ensure that 
\eq
\label{J-def}
J := \lim_{\lambda \to \infty}\left(\lambda-\int_{a^*}^{\lambda}\Delta\right)
\endeq
exists.  We also set 
\eq
\label{U-def}
U:=\oint_{\mathfrak{b}}\Delta.
\endeq
Now $\int_{a^*}^\lambda\Delta$ satisfies the jump conditions 
\eq
\label{Delta-properties}
\begin{split}
\int_{a^*}^{\lambda_+}\Delta & = -U-\int_{a^*}^{\lambda_-}\Delta, \quad \lambda\in\Sigma_\text{up}, \\
\int_{a^*}^{\lambda_+}\Delta & = -\int_{a^*}^{\lambda_-}\Delta,\quad \lambda\in\Sigma_\text{down}
\end{split}
\endeq
(here we restrict the integration path to be on the first sheet).  The 
Riemann-theta function defined by \eqref{Riemann-theta} 
has the properties \cite{dlmf}
\eq
\Theta(-\lambda)=\Theta(\lambda),\quad \Theta(\lambda+2\pi i)=\Theta(\lambda),\quad \Theta(\lambda+B)=e^{-\frac{1}{2}B}e^{-\lambda}\Theta(\lambda).
\endeq
Also $\Theta(\lambda)=0$ if and only if $\lambda=\left(i\pi +\frac{1}{2}B\right)$ mod 
$\Lambda$.  Then for any $Q\in\mathbb{C}$,  the function
\eq \label{q-function}
q(\lambda) := \frac{\Theta(A(\lambda)-A(Q)-i\pi-\frac{B}{2}-F_1U)}{\Theta(A(\lambda)-A(Q)-i\pi-\frac{B}{2})}
e^{-F_1\int_{a^*}^{\lambda}\Delta},
\endeq
is well-defined, independent of the integration path (assuming the paths in 
$A(\lambda)$ and $\int_{a^*}^{\lambda}$ are the same).  The function $q(\lambda)$ 
has a simple zero at $\lambda=Q$ (to be determined).  Consider the matrix
\eq 
\label{T-matrix}
\begin{split}
&{\bf T}(\lambda) := \\ 
&\bbm \displaystyle \frac{\Theta(A(\lambda)+A(Q)+i\pi+\frac{B}{2}-F_1U)}{\Theta(A(\lambda)+A(Q)+i\pi+\frac{B}{2})} e^{-F_1\int_{a^*}^{\lambda}\Delta} & \displaystyle \frac{\Theta(A(\lambda)-A(Q)-i\pi-\frac{B}{2}+F_1U)}{\Theta(A(\lambda)-A(Q)-i\pi-\frac{B}{2})} e^{F_1\int_{a^*}^{\lambda}\Delta} \\ \displaystyle \frac{\Theta(A(\lambda)-A(Q)-i\pi-\frac{B}{2}-F_1U)}{\Theta(A(\lambda)-A(Q)-i\pi-\frac{B}{2})} e^{-F_1\int_{a^*}^{\lambda}\Delta} & \displaystyle \frac{\Theta(A(\lambda)+A(Q)+i\pi+\frac{B}{2}+F_1U)}{\Theta(A(\lambda)+A(Q)+i\pi+\frac{B}{2})} e^{F_1\int_{a^*}^{\lambda}\Delta}\ebm.
\end{split}
\endeq
From \eqref{Abel-properties} and \eqref{Delta-properties}, ${\bf T}(\lambda)$ has 
the jump relations 
\eq
{\bf T}_+(\lambda)={\bf T}_-(\lambda)\bbm 0 & 1 \\ 1 & 0 \ebm, \quad \lambda\in\Sigma_\text{up}\cup\Sigma_\text{down}.
\endeq

We need to slightly adjust the jump condition to that in \eqref{S-jumps} while at 
the same time removing the simple poles in the off-diagonal entries of 
${\bf T}(\lambda)$.  Analogously to \eqref{beta-nonosc}, we define 
\eq
\gamma(\lambda) := \left(\frac{(\lambda-b)(\lambda-a^*)}{(\lambda-a)(\lambda-b^*)}\right)^{1/4}
\endeq
to be the function cut on $\Sigma_\text{up}\cup\Sigma_\text{down}$ with 
asymptotic behavior $\gamma(\lambda) = 1+\mathcal{O}(\lambda^{-1})$ as 
$\lambda\to\infty.$  This function satisfies 
$\gamma_+(\lambda)=-i\gamma_-(\lambda)$ for 
$\lambda\in\Sigma_\text{up}\cup\Sigma_\text{down}$.  Define
\eq
f^\text{D}(\lambda) := \frac{\gamma(\lambda)+\gamma(\lambda)^{-1}}{2},\quad f^\text{OD}(\lambda) := \frac{\gamma(\lambda)-\gamma(\lambda)^{-1}}{2i},
\endeq
so that 
\eq
f_+^\text{D}(\lambda)=f_-^\text{OD}(\lambda),\quad f_+^\text{OD}(\lambda)=-f_-^\text{D}(\lambda), \quad \lambda\in\Sigma_\text{up}\cup\Sigma_\text{down}.
\endeq
Define $Q\equiv Q(\chi,\tau)$ to be the unique complex number such that 
\eq
\label{Q-def}
f^\text{D}(Q)f^\text{OD}(Q)=0.
\endeq
We proceed under the assumption that $Q$ is a simple zero of $f^{OD}(\lambda)$ 
and $f^{D}(\lambda)$ has no zeros.  This is the case we observe numerically for 
the parameter values in Figure \ref{osc-sol-plots-c13}.  The alternate case 
when $f^\text{D}(Q)=0$ does not change the final answer and can be handled by a 
slight modification as described in \cite{BothnerM:2019}.  
If we choose ${\bf S}(\lambda)$ in the form 
\eq \label{S-matrix}
{\bf S}(\lambda)=\bbm C_{11} & 0 \\ 0 & C_{22} \ebm \bbm f^\text{D}(\lambda)[{\bf T}(\lambda)]_{11} & -f^\text{OD}(\lambda)[{\bf T}(\lambda)]_{12} \\ f^\text{OD}(\lambda)[{\bf T}(\lambda)]_{21} & f^\text{D}(\lambda)[{\bf T}(\lambda)]_{22}\ebm,
\endeq
where $C_{11}$ and $C_{22}$ are any constants, then the jump condition 
\eqref{S-jumps} is satisfied, and ${\bf S}(\lambda)$ is analytic for 
$\lambda\notin\Sigma_\text{up}\cup\Sigma_\text{down}$.  Noting that 
$f^\text{OD}(\lambda)=\mathcal{O}(\lambda^{-1})$ and 
$f^\text{D}(\lambda)=1+\mathcal{O}(\lambda^{-1})$, we see the normalization 
\eqref{S-asymptotics} is satisfied if we choose 
\eq
\begin{split}
C_{11} & := \frac{\Theta(A(\infty)+A(Q)+i\pi+\frac{B}{2})}{\Theta(A(\infty)+A(Q)+i\pi+\frac{B}{2}-F_1U)}e^{-F_1J}, \\ 
C_{22} & := \frac{\Theta(A(\infty)+A(Q)+i\pi+\frac{B}{2})}{\Theta(A(\infty)+A(Q)+i\pi+\frac{B}{2}+F_1U)}e^{F_1J}.
\end{split}
\endeq
This completes the construction of ${\bf S}(\lambda)$, and thus of 
${\bf R}^{(\infty)}(\lambda)$ via \eqref{S-def-osc}.  

Define ${\bf R}^{(a)}(\lambda)$, ${\bf R}^{(b)}(\lambda)$, 
${\bf R}^{(a^*)}(\lambda)$, and ${\bf R}^{(b^*)}(\lambda)$ as the 
local parametrices in small, fixed disks $\mathbb{D}^{(a)}$, $\mathbb{D}^{(b)}$, 
$\mathbb{D}^{(a^*)}$, and $\mathbb{D}^{(b^*)}$ centered at $a$, $b$, $a^*$, and 
$b^*$, respectively.  Each of these parametrices can be constructed using Airy 
functions (see, for example, \cite{DeiftKMVZ:1999}).  Then the global parametrix 
\eq
{\bf R}(\lambda) := \begin{cases} {\bf R}^{(a)}(\lambda), & \lambda\in\mathbb{D}^{(a)}, \\ {\bf R}^{(b)}(\lambda), & \lambda\in\mathbb{D}^{(b)}, \\ {\bf R}^{(a^*)}(\lambda), & \lambda\in\mathbb{D}^{(a^*)}, \\ {\bf R}^{(b^*)}(\lambda), & \lambda\in\mathbb{D}^{(b^*)}, \\ {\bf R}^{(\infty)}(\lambda), & \text{otherwise} \end{cases}
\endeq
satisfies 
\eq
{\bf Q}^{[n]}(\lambda) = \left(\mathbb{I}+\mathcal{O}(n^{-1})\right){\bf R}(\lambda).
\endeq
Undoing the different Riemann-Hilbert transformations, we find that, for 
$|\lambda|$ sufficiently large,
\eq
\begin{split}
&[{\bf M}^{[n]}(\lambda;n\chi,n\tau)]_{12}  = \left(\frac{\lambda-\xi^*}{\lambda-\xi}\right)^n[{\bf N}^{[n]}(\lambda;\chi,\tau)]_{12} = \left(\frac{\lambda-\xi^*}{\lambda-\xi}\right)^n[{\bf O}^{[n]}(\lambda;\chi,\tau)]_{12} \\ 
 & = \left(\frac{\lambda-\xi^*}{\lambda-\xi}\right)^ne^{-nG(\lambda;\chi,\tau)}[{\bf P}^{[n]}(\lambda;\chi,\tau)]_{12} = \left(\frac{\lambda-\xi^*}{\lambda-\xi}\right)^n 
 e^{-nG(\lambda;\chi,\tau)}[{\bf Q}^{[n]}(\lambda;\chi,\tau)]_{12} \\
 & = \left(\frac{\lambda-\xi^*}{\lambda-\xi}\right)^n
 e^{-nG(\lambda;\chi,\tau)} \left( [{\bf R}^{(\infty)}(\lambda;\chi,\tau)]_{12} + \mathcal{O}(n^{-1})\right) \\ 
 & = \left(\frac{\lambda-\xi^*}{\lambda-\xi}\right)^n e^{-nG(\lambda;\chi,\tau)} \left( e^{-F(\lambda;\chi,\tau)-F_0(\chi,\tau)} [{\bf S}(\lambda;\chi,\tau)]_{12} +\mathcal{O}(n^{-1})\right) \\ 
 & = \left(\frac{\lambda-\xi^*}{\lambda-\xi}\right)^n e^{-nG(\lambda;\chi,\tau)} \left( -C_{11}(\chi,\tau)f^\text{OD}(\chi,\tau) e^{-F(\lambda;\chi,\tau)-F_0(\chi,\tau)}[{\bf T}(\lambda;\chi,\tau)]_{12} +\mathcal{O}(n^{-1})\right).
\end{split}
\endeq
We now apply 
\eq
f^\text{OD}(\lambda)=\frac{a-a^*-b+b^*}{4i\lambda}+\mathcal{O}(\lambda^{-2}),
\endeq
\eq
\left(\frac{\lambda-\xi^*}{\lambda-\xi}\right)^n = 1+ \mathcal{O}(\lambda^{-1}),
\endeq
and
\eq
e^{-F(\lambda)-F_0-nG(\lambda)} = e^{-F_1\lambda-2F_0}(1+\mathcal{O}(\lambda^{-1}))
\endeq
to find 

\begin{multline}
\lim_{\lambda\to\infty} \lambda [{\bf M}^{[n]}(\lambda;n\chi,n\tau)]_{12} =  \frac{\Theta(A(\infty)-A(Q)-i\pi-\frac{B}{2}+F_1U)\Theta(A(\infty)+A(Q)+i\pi+\frac{B}{2})}{\Theta(A(\infty)-A(Q)-i\pi-\frac{B}{2})\Theta(A(\infty)+A(Q)+i\pi+\frac{B}{2}-F_1U)} \\
\times\frac{a^*-a-b^*+b}{4i} e^{-2F_1J-2F_0} +\mathcal{O}(n^{-1}),
\end{multline}
where the right-hand side is a function of $\chi$ and $\tau$.  
We then recover $\psi^{[2n]}(n\chi,n\tau)$ from \eqref{psi-from-M}, 
thereby proving Theorem \ref{osc-thm}.

\appendix

\section{Construction of the {multiple-pole} solitons via Darboux transformations}
\label{app-Darboux}

We summarize the construction via Darboux transformations of the multiple-pole 
solitons that we study.  Fix $\xi=\alpha+i\beta$ with $\beta>0$ and 
${\bf c} = (c_1,c_2)\in(\mathbb{C}^*)^2$.  We start with the trivial initial 
condition $\psi^{[0]}(x,t)\equiv 0$ and repeatedly apply the same Darboux 
transformation $n$ times to obtain a solution $\psi^{[2n]}(x,t)$ with order $2n$ 
poles at $\xi$ and $\xi^*$.  See \cite{BilmanB:2019} for full details.  

We construct the associated eigenvector matrix ${\bf U}^{[n]}(\lambda;x,t)$ 
iteratively.  Define 
\eq
\label{background-U}
{\bf U}^{[0]}(\lambda;x,t) := e^{-i(\lambda x+\lambda^2 t)\sigma_3}.
\endeq
This is the background eigenvector matrix corresponding to 
$\psi^{[n]}(x,t)\equiv 0$.  Recall the circular disk $D_0$ from Riemann-Hilbert 
Problem \ref{rhp-M} that is centered at the origin and contains $\xi$.  Given 
${\bf U}^{[n]}(\lambda;x,t)$, define 
\eq
\label{s-N-w}
\begin{split}
{\bf s}^{[n]}(x,t):={\bf U}^{[n]}(\xi;x,t){\bf c}^\mathsf{T}, \quad N^{[n]}(x,t):={\bf s}^{[n]}(x,t)^\dagger{\bf s}^{[n]}(x,t), \\
w^{[n]}(x,t):={\bf c}{\bf U}^{[n]}(\xi;x,t)^\mathsf{T}\bbm 0 & -i \\ i & 0 \ebm{\bf U}^{[n]\prime}(\xi;x,t){\bf c}^\mathsf{T}.\phantom{nnnnnn}
\end{split}
\endeq
Here $\dagger$ denotes the conjugate-transpose.  From here, introduce 
\eq
\label{Yn-Zn}
\begin{split}
{\bf Y}^{[n]}(x,t):= & \frac{-4\beta^2 w^{[n]}(x,t)^*}{4\beta^2|w^{[n]}(x,t)|^2 + N^{[n]}(x,t)^2}{\bf s}^{[n]}(x,t){\bf s}^{[n]}(x,t)^\mathsf{T}\bbm 0 & -i \\ i & 0 \ebm \\ 
& + \frac{2i\beta N^{[n]}(x,t)}{4\beta^2|w^{[n]}(x,t)|^2+N^{[n]}(x,t)^2}\bbm 0 & -i \\ i & 0 \ebm{\bf s}^{[n]}(x,t)^*{\bf s}^{[n]}(x,t)^\mathsf{T}\bbm 0 & -i \\ i & 0 \ebm, \\
{\bf Z}^{[n]}(x,t) := & \bbm 0 & -i \\ i & 0 \ebm{\bf Y}^{[n]}(x,t)^*\bbm 0 & -i \\ i & 0 \ebm
\end{split}
\endeq
and define 
\eq
\label{Gn-ito-Y-and-Z}
{\bf G}^{[n]}(\lambda;x,t) := \mathbb{I} + \frac{{\bf Y}^{[n]}(x,t)}{\lambda-\xi} + \frac{{\bf Z}^{[n]}(x,t)}{\lambda-\xi^*}.
\endeq
Then we set 
\eq
\label{Unp1-ito-Gn-Un}
{\bf U}^{[n+1]}(\lambda;x,t) := \begin{cases} {\bf G}^{[n]}(\lambda;x,t){\bf U}^{[n]}(\lambda;x,t), & \lambda\notin D_0, \\ {\bf G}^{[n]}(\lambda;x,t){\bf U}^{[n]}(\lambda;x,t){\bf G}^{[n]}(\lambda;0,0)^{-1}, & \lambda\in D_0 \end{cases}
\endeq
and obtain the desired multiple-pole soliton solution of \eqref{nls} by
\eq
\label{psi-n}
\psi^{[2n+2]}(x,t) = \psi^{[2n]}(x,t) + 2i([{\bf Y}^{[n]}(x,t)]_{12}-[{\bf Y}^{[n]}(x,t)^*]_{21}).
\endeq

{
\section{Elementary symmetry properties of the multiple-pole solitons}
\label{A:symmetries}
Fix $\xi=\alpha+i \beta$, $\alpha\in\mathbb{R}$, $\beta>0$, 
and let
\begin{equation}
B(\lambda;\zeta):= \frac{\lambda-\zeta}{\lambda-\zeta^*}
\end{equation}
for convenience. First note that 
\begin{equation}
B(-\lambda; \xi) = B\big(\lambda; -\xi^*\big)^{-1}.
\label{B-symmetry-1}
\end{equation}
Next, from the definition \eqref{S-def} of $\mathcal{S}\equiv \mathcal{S}(c_1,c_2)$, it is easy to verify that
\begin{equation}
\sigma_3 \mathcal{S}(c_1,c_2) \sigma_1 = \mathcal{S}(-c_2^*, -c_1^*),\quad \sigma_1 := \begin{bmatrix} 0 & 1 \\ 1 & 0\end{bmatrix}.
\label{S-symmetry}
\end{equation}
Let $\theta$ denote the phase $\theta(\lambda;x,t):= \lambda x + \lambda^2 t$ in \eqref{eq:jump-rhp-M}. Define $\mathbf{O}\big(\lambda;x,t;(c_1,c_2),\xi\big)$ in terms of the solution $\mathbf{M}\big(\lambda;x,t;(c_1,c_2),\xi\big)$ of Riemann-Hilbert Problem \ref{rhp-M} by
\begin{equation}
\mathbf{O}\big(\lambda;x,t;(c_1,c_2),\xi\big) = \sigma_3 \mathbf{M}\big(\lambda;x,t;(-c_2^*, -c_1^*),-\xi^* \big) \sigma_3,
\label{O-def}
\end{equation}
and recalling the jump condition \eqref{eq:jump-rhp-M} observe that
\begin{equation}
\begin{split}
\mathbf{O}_+\big(\lambda; x, t; (c_1,c_2), \xi\big) &= \sigma_3 \mathbf{M}_+\big(\lambda;x,t;(-c_2^*, -c_1^*),-\xi^* \big) \sigma_3 \\
&= \sigma_ 3 \mathbf{M}_-\big(\lambda;x,t;(-c_2^*, -c_1^*),-\xi^* \big) \\
&\qquad \times e^{-i \theta(\lambda;x,t)\sigma_3} \mathcal{S}(-c_2^*, -c_1^*) B\big(\lambda;-\xi^*\big)^{n\sigma_3}\mathcal{S}(-c_2^*, -c_1^*)^{-1}e^{ i \theta(\lambda;x,t)\sigma_3} \sigma_3\\
&= \mathbf{O}_-\big(\lambda; x, t; (c_1,c_2), \xi\big) \\
&\qquad \times \sigma_ 3 e^{-i \theta(\lambda;x,t)\sigma_3} \mathcal{S}(-c_2^*, -c_1^*) B\big(\lambda;-\xi^*\big)^{n\sigma_3}\mathcal{S}(-c_2^*, -c_1^*)^{-1}e^{ i \theta(\lambda;x,t)\sigma_3} \sigma_3\\
&= \mathbf{O}_-\big(\lambda; x, t; (c_1,c_2), \xi\big) \\
&\qquad \times e^{-i \theta(\lambda;x,t)\sigma_3} \left[\sigma_ 3 \mathcal{S}(-c_2^*, -c_1^*) \sigma_1\right] B\big(\lambda;-\xi^*\big)^{-n\sigma_3}[\sigma_1\mathcal{S}(-c_2^*, -c_1^*)^{-1} \sigma_3] e^{ i \theta(\lambda;x,t)\sigma_3}\\
&= \mathbf{O}_-\big(\lambda; x, t; (c_1,c_2), \xi\big) \\
&\qquad \times e^{-i \theta(\lambda;x,t)\sigma_3} \mathcal{S}(c_1,c_2) B\big(\lambda;-\xi^*\big)^{-n\sigma_3}\mathcal{S}(c_1,c_2)^{-1} e^{ i \theta(\lambda;x,t)\sigma_3},\\
\end{split}
\end{equation}
where we have used \eqref{S-symmetry} in the last equality. It now follows from \eqref{B-symmetry-1} and $\theta(-\lambda;x,t) = \theta(\lambda;-x,t)$ that $\mathbf{M}(\lambda;-x, t; (c_1,c_2), \xi)$ and $\mathbf{O}(-\lambda;x, t; (c_1,c_2), \xi)$ satisfy the same jump condition. Moreover, they satisfy the same analyticity and normalization condition as $\lambda\to\infty$. Therefore, by uniqueness of the solutions of Riemann-Hilbert Problem~\ref{rhp-M}, 
$\mathbf{O}(-\lambda; x, t; (c_1,c_2), \xi)=\mathbf{M}(\lambda; -x, t; (c_1,c_2), \xi)$. Then
\begin{equation}
\begin{split}
\psi^{[2n]}(-x,t;(c_1,c_2),\xi) & = 2i \lim_{\lambda\to\infty}\lambda  [  \mathbf{M}(\lambda; -x, t; (c_1,c_2), \xi)]_{12} \\
& = 2i \lim_{\lambda\to\infty} \lambda [  \mathbf{O}(-\lambda; x, t; (c_1,c_2), \xi)]_{12}\\
& = 2i \lim_{\lambda\to\infty} \lambda [ \sigma_3 \mathbf{M}(-\lambda; x, t; (-c_2^*, -c_1^*), -\xi^*) \sigma_3]_{12}\\
& = -  2i \lim_{\lambda\to\infty} \lambda [ \sigma_3 \mathbf{M}(\lambda; x, t; (-c_2^*, -c_1^*), -\xi^*) \sigma_3]_{12}\\
& =  2i \lim_{\lambda\to\infty} \lambda [  \mathbf{M}(\lambda; x, t; (-c_2^*, -c_1^*), -\xi^*) ]_{12}\\
& = \psi^{[2n]}(x,t; (-c_2^*, -c_1^*),-\xi^*),
\end{split}
\end{equation}
which proves \eqref{x-symmetry}. To prove \eqref{t-symmetry}, observe that  $B(\lambda^*; \xi)^* = B(\lambda;\xi)^{-1}$,
hence from \eqref{B-symmetry-1} we have
$B(- \lambda^*; \xi)^* = B(\lambda;-\xi^*)$.
From this, together with $[i \theta(- \lambda^*; x, -t)]^* = i\theta(\lambda;x,t) $, it similarly follows that $\mathbf{M}(\lambda;x, -t; (c_1,c_2), \xi )$ and $\mathbf{M}(-\lambda^*; x, t; (c_1^*, c_2^*), -\xi^* )^*$ solve the same Riemann-Hilbert Problem. Then, again by uniqueness,
\begin{equation}
\begin{split}
\psi^{[2n]}(x, -t;(c_1,c_2),\xi) & = 2i \lim_{\lambda\to\infty}\lambda  [  \mathbf{M}(\lambda; x, -t; (c_1,c_2), \xi)]_{12} \\
& = 2i \lim_{\lambda\to\infty} \lambda [  \mathbf{M}(-\lambda^*; x, t; (c_1^*,c_2^*), -\xi^*)^* ]_{12}\\
& =-  2i   \lim_{\lambda\to\infty} \left( \lambda^*[  \mathbf{M}(\lambda^*; x, t; (c_1^*,c_2^*), -\xi^*) ]_{12} \right)^*\\
& =\left(2i  \lim_{\lambda\to\infty}  [ \lambda  \mathbf{M}(\lambda; x, t; (c_1^*,c_2^*), -\xi^*)]_{12} \right)^*\\
&= \psi^{[2n]}(x, t;(c_1^*,c_2^*),-\xi^*)^*,
\end{split}
\end{equation}
which finishes the proof of Proposition~\ref{prop:symmetries}.
}

\end{document}